\documentclass[11pt]{amsart}
\usepackage[latin1]{inputenc}
\usepackage{amsfonts}
\usepackage[toc,page,title,titletoc,header]{appendix}
\usepackage{enumerate,enumitem} % Customize enumerate and itemize
\usepackage{graphicx,psfrag,epsfig,multirow,caption,subcaption}
\usepackage{amssymb,amsmath,amscd,amsthm,amssymb,verbatim,setspace}
\numberwithin{equation}{section}
\usepackage{mathrsfs,wrapfig}
\usepackage{indentfirst}
\usepackage{extarrows}
\usepackage{float}
\usepackage{xcolor}
\usepackage{marginnote}

\usepackage[colorlinks=true, backref=page]{hyperref}

\newtheorem{theorem}{Theorem}[section]
\newtheorem{definition}[theorem]{Definition}
\newtheorem{conjecture}[theorem]{Conjecture}
\newtheorem{proposition}[theorem]{Proposition}
\newtheorem{lemma}[theorem]{Lemma}
\newtheorem{remark}[theorem]{Remark}
\newtheorem{corollary}[theorem]{Corollary}
\newtheorem{example}[theorem]{Example}

\newtheorem{Thm}[theorem]{Theorem}

\newtheorem{Claim}[theorem]{Claim}
\newtheorem{Lem}[theorem]{Lemma}

\newtheorem{Cor}[theorem]{Corollary}
\newtheorem{Rem}[theorem]{Remark}

\newtheorem*{remark*}{Remark}

\numberwithin{equation}{section}

\newcommand{\eps}{\varepsilon}

\newcommand{\al}{\alpha}

\newcommand{\bn}{\mathbf{n}}
\newcommand{\mbfd}{\mathbf{d}}

\newcommand{\mbfC}{\mathbf{C}}
\newcommand{\mbfL}{\mathbf{L}}
\newcommand{\bt}{\mathbf{t}}
\newcommand{\bu}{\mathbf{u}}

\newcommand{\bx}{\hat{\mathbf{x}}}
\newcommand{\hy}{\check{\mathbf{y}}}
\newcommand{\by}{\hat{\mathbf{y}}}

\newcommand{\bM}{\mathbf{M}}
\newcommand{\mbfM}{\mathbf{M}}
\newcommand{\bN}{\mathbf{N}}
\newcommand{\bP}{\mathbf{P}}
\newcommand{\bS}{\mathbb{S}}

\newcommand{\bU}{\mathbf{U}}

\newcommand{\cA}{\mathcal{A}}
\newcommand{\cB}{\mathcal{B}}
\newcommand{\cC}{\mathcal{C}}

\newcommand{\cF}{\mathcal{F}}

\newcommand{\cH}{\mathcal{H}}
\newcommand{\cI}{\mathcal{I}}

\newcommand{\cK}{\mathcal{K}}

\newcommand{\cM}{\mathcal{M}}
\newcommand{\cN}{\mathcal{N}}
\newcommand{\cO}{\mathcal{O}}

\newcommand{\cQ}{\mathcal{Q}}

\newcommand{\cU}{\mathcal{U}}

\newcommand{\cW}{\mathcal{W}}

\newcommand{\R}{\mathbb{R}}

\newcommand{\pr}{\partial}

\newcommand{\BB}{\mathbb{B}}

\newcommand{\RR}{\mathbb{R}}
\newcommand{\SSp}{\mathbb{S}}

\newcommand{\ZZ}{\mathbb{Z}}

\newcommand{\rmC}{\mathrm{C}}
\newcommand{\rmD}{\mathrm{D}}

\newcommand{\rmp}{\mathrm{p}}

\DeclareMathOperator{\vol}{Vol}
\DeclareMathOperator{\area}{Area}

\DeclareMathOperator{\spt}{spt}

\DeclareMathOperator{\graph}{Graph}

\DeclareMathOperator{\dist}{dist}

\DeclareMathOperator{\loc}{loc}
\DeclareMathOperator{\tr}{tr}

\DeclareMathOperator{\id}{\text{Id}}

\newcommand{\orig}{\mathbf{0}}
\newcommand{\Div}{\mathrm{div}}

\newcommand{\spine}{\mathrm{spine}}
\newcommand{\odist}{\overline{\dist}}
\newcommand{\dstar}{\operatorname{d}^*}
\newcommand{\Sing}{\operatorname{Sing}}
\newcommand{\Reg}{\operatorname{Reg}}
\newcommand{\Int}{\operatorname{int}}
\newcommand{\Spt}{\operatorname{spt}}
\newcommand{\sgn}{\operatorname{sgn}}

\title{Passing through nondegenerate singularities in mean curvature flows}
\author{Ao Sun}
\address{Lehigh University, Department of Mathematics, Chandler-Ullmann Hall, Bethlehem, PA 18015}
\email{aos223@lehigh.edu}

\author{Zhihan Wang}
\address{Cornell University, Department of Mathematics, 310 Malott Hall, Ithaca, NY 14853}
\email{zw782@cornell.edu}

\author{Jinxin Xue}
\address{New Cornerstone Science Laboratory, Department of Mathematics, Rm A115, Tsinghua University, Haidian District, Beijing, 100084}
\email{jxue@tsinghua.edu.cn}

\date{}

\begin{document}

	\begin{abstract}
		In this paper, we study the properties of nondegenerate cylindrical singularities of mean curvature flow. We prove they are isolated in spacetime and provide a complete description of the geometry and topology change of the flow passing through the singularities. Particularly, the topology change agrees with the level sets change near a critical point of a Morse function, which is the same as performing surgery. The proof is based on a new $L^2$-distance monotonicity formula, which allows us to derive a discrete almost monotonicity of the ``decay order", a discrete mean curvature flow analog to Almgren's frequency function. 
	\end{abstract}
	\maketitle
	\setcounter{tocdepth}{1}
	\tableofcontents

	\section{Introduction}
	This is the first in a series of papers to study the connections between geometry, topology, and dynamics of cylindrical singularities of mean curvature flow. In this paper, we study the behavior of the mean curvature flows passing through \textit{nondegenerate cylindrical singularities}. 
	
	A mean curvature flow is a family of hypersurfaces in $\R^{n+1}$ moving with velocity equal to the mean curvature vector. From the first variational formula, a mean curvature flow is the fastest way to deform the hypersurface to decrease its area. It has potential for applications in geometry and topology, such as producing minimal surfaces, and studying the geometry and topology of space of hypersurfaces, among others. However, just like other nonlinear problems, mean curvature flows can develop singularities, thus a central question is to understand the singularities and how the flow passes through singularities.
	
	To motivate the work, let us start with the prototypical example: the \textit{dumbbell}, which is a thin neck connecting two large spheres. As the middle neck has huge mean curvature, it develops a singularity first, and it is natural to expect that the flow will pinch at the neck and then disconnect into two pieces. Several weak flow formulations were developed to describe this process. For instance, one way to continue the flow when a singularity appears is to perform \textit{surgery}, which, in the context of mean curvature flows, was first studied in \cite{HuiskenSinestrari09_Surgery2convex}. For the example of the dumbbell, one first observes that the singularity is modeled by a cylinder, and the surgery is done by removing a part of the cylinder immediately before the formation of singularity, gluing two caps, and continuing to run the flow for the two resulting spheres. Nevertheless, the surgery is not canonically defined since the size of the part of the cylinder to remove and the time to perform the surgery are not uniquely defined. 
	
	There are also canonical solutions of (weak) mean curvature flows, two of them that are discussed in this paper are: the level set flow \cite{EvansSpruck91, ChenGigaGoto91_LSF, Ilmanen92_LSF, White95_WSF_Top}, and the Brakke flow generated by elliptic regularization \cite{Brakke78, Ilmanen94_EllipReg}. However, the singular sets can be quite complicated and the topological change may be hard to describe. 
	
	Our main result describes the geometry and topology changes in the canonical process for mean curvature flow through \textit{nondegenerate cylindrical singularities}.
	
	Recall that if $p_\circ\in \RR^{n+1}\times \RR$ is a singular point of a mean curvature flow $t\mapsto \bM(t)$, by composing with a space-time translation, without loss of generality $p_\circ = (\orig, 0)$. Then by Huisken's Monotonicity \cite{Huisken90}, when $\lambda\searrow 0$, the parabolic blow up sequence of flows $t\mapsto \lambda^{-1}\cdot \bM(\lambda^2 t)$ (weakly) subconverges to some self similar mean curvature flow $t\mapsto \sqrt{-t}\cdot S$, where $S$ is known as a \textit{shrinker}. An equivalent way introduced by Huisken \cite{Huisken90} to describe this process is via the \textit{rescaled mean curvature flow} \[
	\tau\mapsto \cM(\tau):= e^{\tau/2}\cdot \bM(e^{-\tau}),  \]
	whose subsequential long time limit $\lim_{\tau\to+\infty}\cM(\tau)$ is a shrinker $S$. If the limit shrinker $S$ is a multiplicity $1$ sphere, we say the singularity $p_\circ$ is \textbf{spherical}; if the limit is a rotation $\cC$ of the generalized multiplicity $1$ cylinder \[
	\cC_{n,k} := \SSp^{n-k}(\sqrt{2(n-k)})\times \R^k 
	= \left\{(x,y)\in \R^{n-k+1}\times \R^k\ |\ |x|=\sqrt{2(n-k)} \right\}\subset\R^{n+1} \]
	for some $k\in \{1,\dots,n-1\}$, we say the singularity is {\bf cylindrical \& modeled by $\cC$}.
	
	Of central importance is the notion of \textit{nondegeneracy} for cylindrical singularities. Given a mean curvature flow $t\mapsto \bM(t)$ with a cylindrical singularity at $(\orig, 0)$ modeled on $\cC_{n,k}$, let us use $\theta$ to denote the coordinates on $\SSp^{n-k}$ factor, and we use $y$ to denote the coordinates on $\R^k$ factor. If we write $\cM(\tau)$ as a graph of the function $u(\cdot,\tau)$ over $\cC_{n,k}$ in a large compact region, in \cite{SunXue2022_generic_cylindrical}, it was proved that as $\tau\to\infty$, up to a rotation, there exists a subset (possibly empty) $\cI\subset\{1,2,\cdots, k\}$ such that
	\[
	u(\theta,y,\tau)=\frac{\sqrt{2(n-k)}}{4\tau}\sum_{i\in\cI}(y_i^2-2)+o(1/\tau).
	\]
	
	This asymptotic is called the \textit{normal form} using terminology from dynamical systems. Such an asymptotic was obtained in \cite{AngenentVelazquez97_DegenerateNeckpinches} when the flow is rotationally symmetric (and hence $k=1$), and see also \cite{Gang21_meanconvexity, Gang22_dynamics} for some special cylinders. 
	
	A {\bf nondegenerate} cylindrical singularity is one with $\cI=\{1,2\cdots,k\}$, namely the graph of the rescaled mean curvature flow has the following asymptotic
	$$u(\theta,y,\tau)=\frac{\sqrt{2(n-k)}}{4\tau}\sum_{i=1}^k(y_i^2-2)+o(1/\tau). $$
	
	Given a generalized cylinder $\cC_{n,k}$, we define its \textbf{dual cylinder} with radius $r$ to be the hypersurface $$\cC_{n,k}^*(r) := \RR^{n-k+1}\times \SSp^{k-1}(r) = \left\{(x,y)\in \R^{n-k+1}\times\R^k\ |\ |y|=r \right\}.$$
	
	In the following Theorem \ref{Thm_Isol_Main}, the term ``mean curvature flow'' stands for a weak solution known as a \textit{unit-regular cyclic $\text{mod}$ $2$ Brakke flow}. A detailed definition is presented in Section \ref{Subsec_Brakke and WSF}. 
	
	\begin{Thm}\label{Thm_Isol_Main}
		Let $1\leq k\leq n-1$, $t\mapsto \bM(t)$ be a mean curvature flow in $\RR^{n+1}$ over $(-1, 1)$ with a \textbf{nondegenerate} cylindrical singularity modeled by $\cC_{n,k}$ at $(\orig, 0)$. Let $Q_r:= \BB_{r}^{n-k+1}(0)\times\BB_{r}^k(0)\subset \RR^{n+1}$. 
		Then there exist $r_\circ, t_\circ\in (0, 1)$ such that
		\begin{enumerate}[label={\normalfont(\roman*)}] %\arabic*; \alph*
			\item\label{Enum_MainThm_Iso} $($isolatedness$)$ $(\orig,0)$ is the only singularity of $\bM$ in the whole parabolic neighborhood $Q_{r_\circ}\times [-t_\circ, t_\circ]$;
			\item\label{Enum_MainThm_MeanConvex} $($mean convexity$)$ $\bM\llcorner (Q_{r_\circ}\times [-t_\circ, t_\circ])$ is mean convex;
			\item\label{Enum_MainThm_Noncollap} $($noncollapsing$)$ For every $t\in [-t_\circ, t_\circ]$, $\bM(t)\llcorner Q_{r_\circ}$ is noncollapsing (see Definition \ref{Def_App_Noncollap});
			\item\label{Enum_MainThm_BdyEvol} $($smooth evolution on boundary$)$ For every $t\in [-t_\circ, t_\circ]$, $\spt\mbfM(t)\cap \partial Q_{r_\circ} = \partial \cU(t)$ for some tubular neighborhood $\cU(t)$ of $\{\orig\}\times \SSp^{k-1}_{r_\circ}$ in $\partial Q_{r_\circ}\cap\cC_{n,k}^*(r_\circ)$, which varies smoothly in $t$; 
			\item\label{Enum_MainThm_t<0} $($graphical before singular time$)$ When $t\in [-t_\circ, 0)$, $\bM(t)\llcorner Q_{r_\circ}$ is a $C^\infty$ graph over $\cC_{n,k}$.
			\item\label{Enum_MainThm_t=0} $($graphical at singular time$)$ When $t=0$, $\bM(0)\llcorner Q_{r_\circ}$ is a graph of function $u$ over $\cC_{n,k}$, and \[
			u(\theta,y)=\frac{\sqrt{2(n-k)}\ |y|}{2\sqrt{-\log(|y|)}}(1+o_y(1))-\sqrt{2(n-k)},  \] 
			where $\|o_y(1)\|_{C^0}\to 0$ as $y\to 0$\footnote{This means $\bM(0)$ has a cusp singularity at $\orig$.}.
			\item\label{Enum_MainThm_t>0} $($graphical after singular time$)$ When $t\in (0, t_\circ]$, the following projection map \[
			\bP_t: \spt\bM(t)\cap \overline{Q_{r_\circ}} \to \cC_{n,k}^*(r_\circ)\,, \quad (x, y)\mapsto \left(x, {r_\circ}\cdot \frac{y}{|y|}\right)
			\]
			is a smooth diffeomorphism onto its image $\cU(t)$. In particular $\bM(t)\llcorner Q_{r_\circ}$ is a graph over the dual cylinder $\cC_{n,k}^*(r_\circ)$.
			\item\label{Enum_MainThm_Top} $($topology change$)$ As a consequence, for every $t\in(0,t_\circ]$, within $Q_{r_\circ}$, topologically $\bM(t)$ is obtained by an $(n-k)$-surgery on $\bM(-t_\circ)\llcorner Q_{r_\circ}$.
		\end{enumerate}
	\end{Thm}
	Recall that we say an $n$-dimensional manifold $X$ is obtained by an \textbf{$(n-k)$-surgery} on an $n$-dimensional manifold $Y$ if we remove $\SSp^{n-k}\times \BB^k \Subset Y$ from $Y$ to get a manifold with boundary $\SSp^{n-k}\times \SSp^{k-1}$, then glue $\BB^{n-k+1}\times \SSp^{k-1}$ back by identifying the boundary $\pr \BB^{n-k+1}$ with the component $\SSp^{n-k}$, to get $X$.
	For the example of dumbbell, the surgery removes a segment diffeomorphism to $\SSp^{1}\times [-1,1]$ from the manifold immediately before the singular time, then glue $\BB^2\times \{\pm 1\}$ to the two resulting spheres with holes. 
	
	\begin{remark}
		Some of the results in Theorem \ref{Thm_Isol_Main} have been explored in other settings or special cases. In the case that the mean curvature flow is rotationally symmetric and generated by a graph over the rotation axis -- hence the singularities are modeled by $\cC_{n,1}$ -- but the singularities are not necessarily nondegenerate, items \ref{Enum_MainThm_Iso}, \ref{Enum_MainThm_t<0}, \ref{Enum_MainThm_t>0}, \ref{Enum_MainThm_Top} have been discussed by Altschuler-Angenent-Giga \cite{AAG95_RotationMCF}; Angenent-Vel\'azquez has discussed item \ref{Enum_MainThm_t=0} in \cite{AngenentVelazquez97_DegenerateNeckpinches}.
		Item \ref{Enum_MainThm_MeanConvex} without nondegenerate assumption, known as the mean convex neighborhood conjecture proposed by Ilmanen, has been proved by Choi-Haslhofer-Hershkovits \cite{ChoiHaslhoferHershkovits18_MeanConvNeighb} and Choi-Haslhofer-Hershkovits-White \cite{ChoiHaslhoferHershkovitsWhite22_AncientMCF} in the case when $k=1$. Items \ref{Enum_MainThm_Iso}, \ref{Enum_MainThm_MeanConvex}, \ref{Enum_MainThm_t<0} have been proved in \cite{SunXue2022_generic_cylindrical} before the singular time with the nondegenerate assumption; see also \cite{Gang21_meanconvexity, Gang22_dynamics} for the special case of $\cC_{4,3}$. \ref{Enum_MainThm_t=0} is studied by Zhou Gang and Shengwen Wang in \cite{GangWang25_AsymPrecise} for the special case of $\cC_{4,3}$ with the nondegenerate assumption.
	\end{remark}

	While Theorem \ref{Thm_Isol_Main} is local around a nondegenerate cylindrical singularity, using the idea of Morse theory, we can also go from local to global. As a corollary, we have a comprehensive description of the spacetime of mean curvature flow with only nondegenerate singularities. For example, from the Morse theory point of view, we can view this surgery as the transition of the level sets on the spacetime track $\spt\mbfM := \overline{\bigcup_{t\geq 0}\bM(t)\times\{t\}} \subset\R^{n+1}\times\R$. In fact, the $(n-k)$-surgery is exactly what happens in the following Morse theoretic setting: if $N$ is an $(n+1)$-dimensional manifold and $f$ is a Morse function with a critical point $p$ with index $(n-k+1)$, such that $f(p)=0$, then for $\eps>0$ that is small, near $p$, $f^{-1}(\eps)$ is obtained from $f^{-1}(-\eps)$ by an $(n-k)$-surgery. Our main theorem says that near a nondegenerate cylindrical singularity, the level sets of the time function $\mathfrak t:\R^{n+1}\times\R\to\R$ on $\spt\mbfM$ behave like the level sets of a Morse function. In particular, if a nondegenerate singularity $p$ is modeled by $\cC_{n,k}$, the index of $\mathfrak t$ at $p$ is $(n-k+1)$, which is always greater or equal to $2$.
	
	We remark that from \cite{ColdingMinicozzi18_Regularity_LSF, SunXue2022_generic_cylindrical}, the time function is not Morse (in fact, not even $C^2$) near a nondegenerate cylindrical singularity modeled by $\cC_{n,k}$.
	
	\begin{corollary}\label{cor:main}
		Suppose a mean curvature flow $t\mapsto \bM(t)$ of hypersurfaces in $\R^{n+1}$ ($n\geq 2$) has only nondegenerate cylindrical singularities and spherical singularities before extinction. Then 
		\begin{enumerate}[label={\normalfont(\roman*)}] %\arabic*; \alph*
			\item \label{cor:main i1} The flow is unique;
			\item \label{cor:main i2} The flow has only finitely many singularities,
			\item \label{cor:main i3} There exists a Morse function on the spacetime track with the same number of critical points of index $(n-k+1)$ as the number of cylindrical singular points modeled by $\cC_{n,k}$\footnote{We view $\cC_{n, 0}$ as the round sphere of radius $\sqrt{2n}$.}. In particular, the indices of the critical points of this Morse function are greater or equal to $2$.
		\end{enumerate}
	\end{corollary}
	
	\begin{proof}
		Item \ref{cor:main i1} is a consequence of the item \ref{Enum_MainThm_MeanConvex} of Theorem \ref{Thm_Isol_Main} and the criterion on the uniqueness of mean curvature flow whose singularities having mean convex neighborhood by Hershkovits-White \cite{HershkovitsWhite20_Nonfattening}; Item  \ref{cor:main i2} is a consequence of the item \ref{Enum_MainThm_Iso} of Theorem \ref{Thm_Isol_Main}; Item \ref{cor:main i3} is a consequence of the item \ref{Enum_MainThm_Top} of Theorem \ref{Thm_Isol_Main}, which implies that the spacetime track of mean curvature flow with nondegenerate cylindrical singularities and spherical singularities gives a cobordism which describes the surgery process. Then by the Morse theory of cobordism, e.g. Section 3 of Milnor's \cite{Milnor65_h-cobordism}, there exists a Morse function on the spacetime track so that the critical points correspond to the surgery process.
	\end{proof}
	
	The highlight of our theorem is that, although we need a nondegeneracy requirement for the cylindrical singularities, our results hold for mean curvature flow of hypersurfaces in $\R^{n+1}$ for all $n\geq 2$, and the singularities modeled by $\cC_{n,k}$ for all $1\leq k\leq (n-1)$. On the other hand, nondegeneracy seems not to be a very restrictive condition. In fact, motivated by the work in \cite{SunXue2022_generic_cylindrical}, nondegenerate cylindrical singularities should be the most generic type of singularities of mean curvature flow. We believe our results can serve as an important step towards the research and applications of higher dimensional mean curvature flow.

	\subsection{Why nondegenerate cylindrical singularities?}
	
	Although these generalized cylinders might seem to be simple models, there are several key reasons we are particularly interested in them.
	
	\begin{enumerate}[label={\normalfont(\arabic*)}] %\arabic*; \alph*
		\item From the pioneering work of Colding-Minicozzi, the generalized cylinders are the only ``linearly stable" shrinkers. Recent progress on generic mean curvature flows also suggested that the cylindrical and spherical singularities are the singularities of generic mean curvature flows in $\R^3$, as well as in $\RR^{n+1\geq 4}$ under low entropy assumptions, see \cite{ColdingMinicozzi12_generic, CCMS20_GenericMCF, CCMS21_GenericMCF_LowEntropy, chodoshchoischulze2023mean, CCS23_LowEntropy_ii, BamlerKleiner23_Multiplicity1, BernsteinWang17_small_entropy, SunXue2021_initial_closed, SunXue2021_initial_conical}.
		
		\item Round spheres and cylinders are the only possible tangent flows at singularities of mean convex mean curvature flow \cite{White97_Stratif, White00, White03, ShengWang09_SingMCF, Andrews12_Noncollapsing, HaslhoferKleiner17}, and rotationally symmetric mean curvature flow. They are also the only mean convex shrinkers \cite{Huisken93_C2regularity, ColdingMinicozzi12_generic}, complete embedded rotationally symmetric shrinkers, genus $0$ shrinkers in $\R^3$ \cite{Brendle16_genus0} and non-planar shrinkers with the smallest entropy in $\R^3$ \cite{CIMW13_EntropyMinmzer, BernsteinWang17_small_entropy}.
		
		\item While the generalized cylinders are simple models, they can model a highly complicated singular set. For example, the cylinder $\SSp^{n-1}\times\R$ can model either a neck pinch at an isolated singularity or a singular point located in a curve (see the examples below).
	\end{enumerate}
	
	The nondegeneracy in Theorem \ref{Thm_Isol_Main} is crucial. Without the nondegeneracy, the topology change passing through the singularity may be much more complicated. Let us list some possibly pathological examples or conjectural pictures.
	
	\begin{itemize} %[label={\normalfont(\arabic*)}] %\arabic*; \alph*
		\item {\bf Marriage ring.} This is a thin torus with rotation symmetry. Under the mean curvature flow it preserves rotation symmetry and becomes thiner and thiner. Ultimately it vanishes along a singular set which is a circle, each singularity is a cylindrical singularity. 
		
		\item {\bf Degenerate neckpinch.} Even if a cylindrical singularity is isolated, its influence on the topology of the flow may not match the intuition. For example, Altschuler-Angenent-Giga \cite{AAG95_RotationMCF} constructed an example called ``peanut'', which is a peanut-shaped surface that shrinks to a cylindrical singularity then vanishes under the mean curvature flow. Later, Angenenet-Vel\'azquez \cite{AngenentVelazquez97_DegenerateNeckpinches} constructed a large class of degenerate singularities. Some of those mean curvature flows are topologically spheres at the beginning, and they can generate a cylindrical singularity like a cusp, but the topology of the flows remains exactly the same spheres after the singular time.
		
		\item {\bf Sparkling bubbles.} This is a conjectural example. It has been proved in \cite{White00, White03, White15_SubseqSing_MeanConvex, ShengWang09_SingMCF, HaslhoferKleiner17} that the blow-up limit flow near a singularity of mean convex mean curvature flow must be a convex ancient flow. Among other examples, there is a class of compact convex ancient flows called \textit{ovals}, see \cite{Angenent13_Oval, DuHaslhofer21_UniqOvals}. If the ovals show up when we blow up a cylindrical singularity, we should expect that after passing through the singularity, the mean curvature flow generates very tiny convex bubbles, just like sparkling bubbles. At this moment, there is no explicit evidence of whether such a picture can really show up, and the conjecture is this can not happen, at least in the mean convex case, see \cite{ChoiHaslhoferHershkocits21_SelfSimilarity}.
	\end{itemize}
	
	Among cylindrical singularities, nondegenerate singularities are \textit{locally generic} by the work of the first and the third named authors. In \cite{SunXue2022_generic_cylindrical}, we proved a nondegenerate cylindrical singularity is stable under small perturbations and one can perturb a degenerate singularity to make it nondegenerate. We can say that the nondegenerate cylindrical singularities are {\it the most generic singularities}. In contrast, Altschuler-Angenent-Giga \cite{AAG95_RotationMCF} and Angenent-Vel\'azquez \cite{AngenentVelazquez97_DegenerateNeckpinches} constructed examples of degenerate singularities, which can be perturbed away. It is promising to have a positive answer to the following conjecture:
	\begin{conjecture}\label{conj:genericMCF}
		A mean curvature flow with generic initial data in $\R^{n+1}$ (or more generally, a general complete manifold with bounded geometry), only develops nondegenerate cylindrical singularities or spherical singularities in finite time. 
	\end{conjecture}
	
	If this conjecture is settled, we can use the result of this paper to study a large class of geometry and topology questions. For example, Corollary \ref{thm:MeanConvexSurgery} holds for any mean convex hypersurfaces, and other topological implications in Section \ref{SS:TopCor} can have assumptions relaxed. 
	
	It would be interesting to compare our results with other geometric flows with surgery. In a parallel realm, the Ricci flow with surgery was used by Perelman \cite{Perelman03_RFSurgery} as a key step to prove
	Poincar\'e Conjecture and the Geometrization Conjecture of Thurston. However, his surgery process depends on surgery parameters and is not canonical as well. It was mentioned by Perelman that
	
	\begin{quote}
		{
			``
			It is likely that ... one would get a canonically defined Ricci flow through singularities, but at the moment I don't have a proof of that... Our approach... is aimed at
			eventually constructing a
			canonical Ricci flow, ... -- a goal that has not been achieved yet in the present work.
			''
		}
	\end{quote}
	Kleiner-Lott \cite{KleinerLott17_weakRF} developed a weak flow theory for $3$-dimensional Ricci flow. The construction was based on Perelman's surgery process, to show that if the surgery scale becomes smaller and smaller, the flow with surgery will converge to a unique weak Ricci flow. For $2$-convex mean curvature flows, the surgery theory was developed by Huisken-Sinestrari \cite{HuiskenSinestrari09_Surgery2convex} in $\R^{n+1}$ with $n\geq 3$, and by Brendle-Huisken \cite{BrendleHuisken16_MCFSurgeryR3} and Haslhofer-Kleiner \cite{HaslhoferKleiner17_MCFsurgery} in $\R^3$ independently. Using the classification result for $2$-convex ancient flows by Choi-Haslhofer-Hershkovits \cite{ChoiHaslhoferHershkovits18_MeanConvNeighb} and Choi-Haslhofer-Hershkovits-White \cite{ChoiHaslhoferHershkovitsWhite22_AncientMCF}, Daniels-Holgate \cite{Daniels-Holgate22_SurgeryApprox} constructed smooth mean curvature flows with surgery that approximate weak mean curvature flows with only spherical and neck-pinch singularities.
	
	While all the results of geometric flows with surgery mentioned above do not require nondegeneracy, they strongly rely on the assumption that the cylindrical singularities are modeled by $\mathbb S^{n-1}\times\R$, namely the Euclidean factor has dimension $1$. In contrast, although our result requires the singularity to be nondegenerate, it allows the cylindrical singularities to have arbitrary dimensions of the Euclidean factor. This would be essential when we study mean curvature flows in $\R^{n+1}$ with $n\geq 3$, as more complicated cylindrical singularities can show up.

	\subsection{Topological Implications}\label{SS:TopCor}
	Geometric flows play a significant role in the study of geometry and topology. Highlights include Perelman's proof of Poincar\'e conjecture and Thurston's geometrization conjecture, Brendle and Schoen's proof of the differentiable sphere theorem \cite{BrendleSchoen09_Pinch}, and Bamler and Kleiner's proof of Smale's conjecture regarding the structure of the space of self-diffeomorphisms of $3$-manifolds \cite{BamlerKleiner23_Smale}. The topology of mean curvature flow has been studied by White \cite{White95_WSF_Top, White13_MeanConvex_Top}, see also some applications in \cite{ChuSun23_genus}. Bernstein and Lu Wang \cite{BernsteinWang22_closedR4}, and Chodosh-Choi-Mantoulidis-Schulze \cite{CCMS21_GenericMCF_LowEntropy} proved the Schoenflies Conjecture in $\R^4$ with low entropy assumption. The low entropy assumption is imposed to rule out possibly complicated singularities, such as those modeled by $\cC_{n,k}$ for $k\geq 2$.
	
	If Conjecture \ref{conj:genericMCF} is true, then our results provide many more potential applications of mean curvature flows, especially in $\mathbb R^{n+1}$ with $n\geq 3$. While a full resolution of Conjecture \ref{conj:genericMCF} seems currently out of our scope, we would like to point out some heuristic implications to topology. In the following, we will focus on flows with the following assumption:
	\begin{align}
		t\mapsto \bM(t)\ \text{ has only nondegenerate cylindrical and spherical singularities.} \tag{$\star$} \label{Equ_Intro_Star}
	\end{align}

	Using the result of our paper, we have the following corollary to describe the topology of the domain with $k$-convex boundary\footnote{Recall that the principal curvatures of a hypersurface are eigenvalues $\kappa_1\leq \kappa_2\leq\cdots\leq \kappa_n$ of the second fundamental form $A$. $\pr\Omega$ is $k$-convex for $m\geq 1$ if the sum of the first $k$ principal curvatures is nonnegative. Note that $1$-convexity is equivalent to convexity and $n$-convexity is equivalent to mean convexity.}. The $k$-convexity condition is preserved under mean curvature flow by Huisken-Sinestrari \cite{HuiskenSinestrari99_Acta} for smooth case and White \cite{White15_SubseqSing_MeanConvex} for elliptic regularizations and level set flows. In particular, it shows that the tangent flow can only be $\cC_{n,m}$ where $m\leq k-1$. 
	
	\begin{corollary}\label{thm:MeanConvexSurgery}
		Suppose $t\mapsto \bM(t)$ is the mean curvature flow starting from a closed hypersurface $M_0 = \partial \Omega$ for some bounded smooth domain $\Omega$ satisfying \eqref{Equ_Intro_Star}. If $M_0$ is $k$-convex, then there exists a Morse function on $\Omega$ with no index $m$ critical points for $m=0,1,2,\cdots,(n-k+1)$. As a consequence,
		\begin{itemize}
			\item $b_1(\Omega,\pr\Omega)=b_2(\Omega,\pr\Omega)=\cdots=b_{n-k+1}(\Omega,\pr\Omega)=0$;
			\item $\Omega$ can be obtained by a finite union of standard balls in $\R^{n+1}$ after attaching finitely many $m$-handles for $m \in \{1,\cdots,k-1\}$. 
		\end{itemize}
	\end{corollary}
	
	In general, the hypersurface may not be $k$-convex, but we can still obtain some topological information of the spacetime:
	
	\begin{corollary}\label{thm:SpacetimeSurgery}
		Suppose $t\mapsto \bM(t)$ is the mean curvature flow starting from a closed hypersurface $M_0\subset\R^{n+1}$ and $\bM$ satisfies \eqref{Equ_Intro_Star}. Then the spacetime track $\spt \bM$ of the flow can be obtained by attaching finitely many $m$-handles for $m \in \{1,\cdots,n-1\}$ to a finite disjoint union of standard balls in $\R^{n+1}$.
	\end{corollary}
	
	If Conjecture \ref{conj:genericMCF} is true, Corollary \ref{thm:MeanConvexSurgery} and \ref{thm:SpacetimeSurgery} can be viewed as a ``missing handle'' property. From the Morse theory point of view, the handle decomposition of a manifold can imply the homology and homotopy information of the manifold. The converse question of whether the homology and homotopy information of the manifold can determine the handle decomposition is widely open. For example, an open question proposed by Kirby \cite{Kirby89_4Mfd} is whether a closed simply connected $4$-manifold admits a handle decomposition with no $3$-handles. Our main theorem suggested that $n$-handles and $(n+1)$-handles\footnote{By the basic Morse theory, a compact $(n+1)$-manifold with boundary always has a handle decomposition without $(n+1)$-handle.} can be missed in the handle decomposition of the spacetime track of an embedded hypersurface in $\R^{n+1}\times\R$ with only nondegenerate cylindrical and spherical singularities under the mean curvature flow. In particular, assuming Conjecture \ref{conj:genericMCF}, if the embedded hypersurface is the boundary of a mean convex domain, then this domain has a handle decomposition missing $n$-handles and $(n+1)$-handles.

	Corollary \ref{thm:MeanConvexSurgery} can be also viewed as the counterpart of the surgery theory of manifolds with positive scalar curvatures. 
	In \cite{Gromov19_MeanCurScalarCur}, Gromov observed that mean convex domains in $\R^{n+1}$ and manifolds with positive scalar curvature have similar properties. He also proposed to use the properties of one of them to study the other. 
	Schoen-Yau \cite{SchoenYau79_StructurePSC} and Gromov-Lawson \cite{GromovLawson80_ClassificationPSC} showed that given a closed manifold $M$ of dimension $n\geq 3$ with positive scalar curvature, after performing $0,1,\cdots,(n-3)$-surgeries, the resulting manifold still has positive scalar curvature. Conversely, it is not clear what are the building blocks from which any positive scalar curved manifold can be obtained through those surgeries. When $n=3$, Perelman proved that the building blocks are spherical space forms, and when $n=4$, Bamler-Li-Mantoulidis \cite{BamlerLiMantoulidis23_4mfdRF} provided some reductions. Corollary \ref{thm:MeanConvexSurgery} and Conjecture \ref{conj:genericMCF} suggest that for mean convex domains in $\R^{n+1}$, the building blocks are standard balls.
	
	In summary, we expect the result of this paper can illustrate the topological structure of closed embedded hypersurfaces in $\R^{n+1}$ or in an $(n+1)$-dimensional manifold.
	
	\medskip

	Another perspective of topological implications is a lower bound on the number of singular points of mean curvature flow with only nondegenerate cylindrical singularities and spherical singularities in terms of the topology of the initial data.

	In \cite{White13_MeanConvex_Top}, White showed that certain types of cylindrical singularities of mean convex mean curvature flow must occur according to the topology of the complement of the initial hypersurface. Because the topological change of the mean curvature flow passing through nondegenerate singularities can be characterized comprehensively, we can prove that certain types of cylindrical singularities must occur in mean curvature flow with only nondegenerate singularities. For simplicity, we state the theorem for homology groups with coefficients $\R$, but similar theorems hold for other coefficients.
	
	\begin{corollary}[Lower bound on numbers of singularities]
		Suppose $\bM$ is the mean curvature flow starting from a closed hypersurface $M_0$ and $\bM$ satisfies \eqref{Equ_Intro_Star}. For $1\leq k\leq n-1$, let $b_k(M_0)$ be the Betti number of the $k$-the homology with $\R$-coefficient. Then 
		\begin{itemize}
			\item when $n\neq 2k$, $\bM$ has at least $b_k(M_0)$ number of nondegenerate singularities modeled by $\cC_{n,k}$ or $\cC_{n,n-k}$;
			\item when $n=2k$, $\bM$ has at least $\frac{b_{n/2}(M_0)}{2}$ number of nondegenerate singularities modeled by $\cC_{n,n/2}$.
		\end{itemize}
	\end{corollary}
	
	\begin{proof}
		Let us view the spacetime track (still denoted by $\bM$ for simplicity) as a manifold with $\pr\bM=M_0$, and let $b_{k}(\bM,M_0)$ be the relative Betti number with $\R$-coefficient. By the Poincare-Lefschetz duality, we have
		\[
		H_k(\bM) \cong H^{n-k+1}(\bM, M_0) \cong (H_{n-k+1}(\bM, M_0))^* 
		\]
		Then the long exact sequence of relative homology gives the following exact sequence
		\[
		\cdots \to H_{k+1}(\bM,M_0)\to H_k(M_0)\to H_k(\bM)\cong (H_{n-k+1}(\bM, M_0))^* \to \cdots
		\]
		This implies the Betti number inequality
		\begin{equation}
			b_{k+1}(\bM,M_0)+b_{n-k+1}(\bM,M_0)\geq b_k(M_0).
		\end{equation}
		Now apply Corollary \ref{cor:main} and notice that for every $j\geq 0$, $b_j(\bM, M_0)$ is less or equal than the number of index $j$ critical point of the Morse function therein. We thus obtain the estimated number of singularities.
	\end{proof}
	
	If $M_0=\pr\Omega$ is mean convex, then time slices of $\bM$ sweep out the whole region $\Omega$, and $\bM$ is homeomorphic to $\Omega$. Therefore we have a more precise estimate of the number of nondegenerate singularities.
	
	\begin{corollary}
		Suppose $\bM$ is the mean curvature flow starting from a closed hypersurface $M_0$ and $\bM$ satisfies \eqref{Equ_Intro_Star}. If $M_0=\pr\Omega$ is mean convex, then $\bM$ has at least $b_{n-k+1}(\Omega,\pr\Omega)$ number of nondegenerate singularities modeled by $\cC_{n,k}$.
	\end{corollary}

	\subsection{Main Idea of the Proof}
	Let $t\mapsto \bM(t)$ be a mean curvature flow with a non-degenerate singularity modeled on $\cC_{n,k}$ at $(\orig, 0)$.

	\ref{Enum_MainThm_BdyEvol} (boundary evolution), \ref{Enum_MainThm_t<0} (graphical when $t<0$) and \ref{Enum_MainThm_t=0} (graphical when $t=0$) of Theorem \ref{Thm_Isol_Main} follows from \cite{SunXue2022_generic_cylindrical} with a pseudo-locality argument. Then \ref{Enum_MainThm_MeanConvex} (mean convexity) of Theorem \ref{Thm_Isol_Main} follows from avoidance principle by a time translation, and \ref{Enum_MainThm_Noncollap} (noncollapsing) is proved via elliptic regularization. These will be discussed in Section \ref{SSSec_PseudoLoc} and \ref{SSSec_Ellip Reg}. The bulk of this paper is devoted to proving \ref{Enum_MainThm_Iso} (isolatedness) and \ref{Enum_MainThm_t>0} (graphical when $t>0$) of Theorem \ref{Thm_Isol_Main}.
	
	\subsubsection{Exponential growing modes out of nondegeneracy}
	By definition of non-degeneracy, for a fixed $t_\circ<0$ such that $\tau_\circ:= -\log(-t_\circ)\gg 1$, $\sqrt{-t_\circ}^{-1}\cdot\bM(t_\circ)$ can be written as a graph over $\cC_{n,k}$ of \[
	\frac{\sqrt{2(n-k)}}{4\tau_\circ}\cdot (|y|^2 - 2) + o(1/\tau_\circ)
	\]
	within any ball of radius $\sim \mbfL \gg 1$. 
	
	If we translate $\bM$ in $\{0\}\times\RR^k$ direction by $\sqrt{-t_\circ}\,\by\in \RR^k$ and then in time direction by $\bt$, the resulting flow $\bM'$ will have the rescaled $t_\circ$-time slice $\sqrt{-t_\circ}^{-1}\cdot\bM'(t_\circ)$ to be approximately a graph over $\cC_{n,k}$ of \[
	\frac{\sqrt{2(n-k)}}{4\tau_\circ}\cdot (|y-\by|^2 - 2k + \bt') = \frac{\sqrt{2(n-k)}}{4\tau_\circ}\cdot \Big((|y|^2 - 2k) - 2y\cdot \by + (|\by|^2 + \bt')\Big)
	\] 
	for some constant $\bt'\sim e^{\tau_\circ}\bt$, as long as $|\by|\ll \mbfL$, $|\bt|\ll e^{-\tau_\circ}$. 
	When $|\by|\gg 1$, such a function is dominated by either the constant term $|\by|^2 + \bt'$ or the linear term $y\cdot \by$, both of which are unstable modes of the rescaled mean curvature flow near $\cC_{n,k}$. This means, if we start the flow from $\bM'(t_\circ)$ at time $t_\circ$, then after some time, $\sqrt{-t}^{-1}\bM'(t)$ must leave a neighborhood of $\cC_{n,k}$ with a strict drop of Gaussian area and hence never comes back by Huisken's monotonicity formula. That at least forces $(\orig, 0)$ not to be a singularity of $\bM'$ modeled on $\cC_{n,k}$, or equivalently, $(0, -\sqrt{-t_\circ}\by, -\bt)$ not to be  a singularity of $\bM$ modeled on $\cC_{n,k}$. Further analysis can also take into account of translations in $\RR^{n-k+1}$ directions.
	
	To carry out this process and rule out \textit{every} singular point in a neighborhood of $(\orig, 0)$, there are two main difficulties we need to overcome, discussed in the next two paragraphs. 
	\subsubsection{Nonconcentration at infinity}
	(i) We need to rule out the possible effect of infinity to the non-linear evolution of (rescaled) mean curvature flow near $\cC_{n,k}$. To do that,   we shall focus  on the \textbf{$L^2$-distance} to the round cylinder: for any $X\in\R^{n+1}$, let $\odist_{n,k}(X)=\min\{\dist(X,\cC_{n,k}),1\}$\footnote{In real application, we use a regularized version of this, see \eqref{Equ_L^2 Mono_odist to C_(n,k)}}. Then we define the $L^2$-distance of a hypersurface $\Sigma\subset\R^{n+1}$ to $\cC_{n,k}$ by
	\[
	\mathbf{d}_{n,k}(\Sigma)^2 := \int_\Sigma \odist_{n,k}(X)^2e^{-\frac{|X|^2}{4}}\ dX.
	\]
	We prove the following \textit{non-concentration near infinity} of $L^2$-distance for rescaled mean curvature flow $\tau\mapsto \cM(\tau)$ (see Corollary \ref{Cor_L^2 Mono_Weighted L^2 Mono and Est for RMCF}): $\forall\, \tau>0$, 
	\begin{align}
		\int_{\cM(\tau)} \odist_{n,k}(X)^2(1+\tau|X|^2)e^{-\frac{|X|^2}{4}}\ dX \leq C_ne^{K_n\tau}\cdot\mbfd_{n,k}(\cM(0))^2\,.  \label{Equ_Intro_Noncon near infty}        
	\end{align}
	An analogs non-concentration estimate has been proved in \cite{AngenentDaskalopoulosSesum19_AncientConvexMCF} for ancient mean curvature flow asymptotic to $\cC_{n, k}$.
	
	Another key quantity we introduced in this paper is the {\bf decay order}. Suppose $\tau\mapsto \cM(\tau)$ is a rescaled mean curvature flow, we define the {\bf decay order} of $\cM$ at time $\tau$ to be \[
	\cN_{n,k}(\tau;\cM):=\log\left( \frac{\mbfd_{n,k}(\cM(\tau))}{\mbfd_{n,k}(\cM(\tau+1))}\right).
	\]
	It is a discrete parabolic analog of Almgren's frequency function and doubling constant in the elliptic problems. Roughly speaking, it characterizes the asymptotic rate of the rescaled mean curvature flow converging to the limiting cylinder. Particularly, if $\cM(\tau)$ is a graph of function $u(X,\tau)\approx e^{-d\tau}w(X)+\text{errors}$ as $\tau\to\infty$, then $\cN_{n,k}(\tau;\cM)\approx d$ when $\tau\gg 1$. For some other parabolic analogs of Almgren's frequency, see \cite{Lin90_UniqueParabolic, Poon96_ParaFrequency, ColdingMinicozzi22_ParaFrequency, BaldaufHoLee24_ParaFrequency}.
	
	The $L^2$-distance and decay order have three essential features. 
	
	First, it can be defined for weak flows, such as a Brakke flow, and the non-concentration estimate \eqref{Equ_Intro_Noncon near infty} still holds. 
	
	Second, the decay order can indeed capture the dynamical information of a cylindrical singularity. For example, near a non-degenerate singularity $(\orig, 0)$, the corresponding rescaled mean curvature flow $\cM$ has $\lim_{\tau\to+\infty}\cN_{n,k}(\tau;\cM)=0$, see Example \ref{Eg_Nondegen sing has cN = 0}. On the other hand, if $\cN_{n,k}(0;\cM)$ is uniformly bounded from above, then combined with the non-concentration estimate \eqref{Equ_Intro_Noncon near infty}, we have that $\mbfd_{n,k}(\cM(\tau)\cap \BB_R)$ dominates $\mbfd_{n,k}(\cM(\tau)\setminus \BB_R)$ when $\tau\in (0, 1]$ and $\sqrt\tau R\gg 1$. This enables us to capture the dynamics of  $\cM(\tau)$ in $\BB_R$, which is well modeled by (parabolic) Jacobi fields on $\cC_{n,k}$ when $\cM$ is close to $\cC_{n,k}$. In this way, we prove a \textit{discrete almost monotonicity for $\cN_{n,k}(\tau, \cM)$} in Corollary \ref{Cor_L^2 Mono_Discrete Growth Mono}, which is an analogue of the frequency monotonicity for solutions to linear parabolic equations on $\cC_{n,k}$, see Appendix \ref{app_Jacobi}.     
	
	Third, the decay order provides a practical way to study the flow after small spacetime translations and dilation. This feature also help us to characterize the dynamical information of degenerate singularities, which we will discuss in the forthcoming paper.

	\subsubsection{Topology of the flow after passing through a nondegenerate singularity}
	
	(ii) Near a non-degenerate singularity, we need to rule out not only singularities modeled on $\cC_{n,k}$, but also other possible singularities. To achieve this, we prove a Classification Theorem, see Theorem \ref{Thm_Isol_Blowup Model}, of every blow-up model $\tilde\bM_\infty$ (i.e. limit flows) of $\bM$ near $(\orig, 0)$, asserting that with appropriate choice of blow-up rates, $\tilde\bM_\infty(0)$ must be either a translation and dilation of $\cC_{n,k}$, or a translation, dilation and rotation of an $(n-k+1)$-dimensional \textit{bowl soliton} times $\RR^{k-1}$. Note that \ref{Enum_MainThm_Iso} of Theorem \ref{Thm_Isol_Main} follows directly from this classification and Brakke-White's epsilon regularity \cite{White05_MCFReg}. 
	
	This classification is achieved by exploiting our analysis of decay order for limiting flows blown up near a nondegenerate singularity, combined with the classification result of non-collapsing ancient mean curvature flows by Wenkui Du and Jingze Zhu \cite[Theorem 1.10]{DuZhu22_quantization}. 
	More precisely, for an arbitrary sequence $p_j=(x_j, y_j, t_j)\in \spt\bM$ approaching $(\orig, 0)$, we basically show that when $j\gg 1$, some appropriate parabolic dilation $\tilde\bM_j$ of the translated flow $\bM - p_j$ has its associated rescaled mean curvature flow $\tau\mapsto\tilde\cM_j(\tau)$ graphical over large subdomains in $\cC_{n,k}$ for all $\tau\leq 0$, but $\mbfd_{n,k}(\tilde\cM_j(2))$ has a uniform positive lower bound, and the decay order satisfies $\cN_{n,k}(\tau; \tilde\cM_j)\leq -1/4$ for $\tau\leq 0$. In particular, $\tilde\cM_j$ subconverges to some limit flow $\tilde\cM_\infty$ which does not coincide with the round cylinder $\cC_{n,k}$, but $\mbfd_{n,k}(\tilde\cM_\infty(\tau))$ decays exponentially when $\tau\searrow -\infty$. A more careful analysis via avoidance principle and pseudo-locality proves that $\tilde\cM_\infty$ is a convex non-collapsing flow. 
	This enables us to apply \cite[Theorem 1.10]{DuZhu22_quantization} to complete the proof.
	
	To extract topological information and prove \ref{Enum_MainThm_t>0} of Theorem \ref{Thm_Isol_Main}, we also extract refined information in this Classification Theorem \ref{Thm_Isol_Blowup Model}. More precisely, using the notations above and let $\tilde\bM_\infty$ be the subsequential limit of $\tilde\bM_j$, we show that if the base points $p_j$ satisfy $y_j/|y_j|\to \by$ for some unit vector $\by\in \RR^k$, then, 
	\begin{enumerate}[label={\normalfont(\alph*)}] %\arabic*; \alph*
		\item if $\tilde\bM_\infty(0)$ is a bowl soliton times $\RR^{k-1}$, then it must translate in $(0, \by)$-direction;
		\item if $\tilde\bM_\infty(0)$ is a round cylinder, then for $j\gg 1$, the outward normal vector $\tilde\nu_j(\orig, 0)$ of $\tilde\bM_j(0)$ at $\orig$ satisfies $\tilde\nu_j(\orig,0)\cdot (0, \by)<0$.
	\end{enumerate}
	These two refined blow-up information follow both from the fact that once a linear mode $y\cdot \by$ dominates the graphical function of a rescaled mean curvature flow $\cM$ near $\cC_{n,k}$ at time $\tau = a$, then \textbf{this} mode remains domination until $\cM(\tau)$ leaves a small neighborhood of $\cC_{n,k}$.
	
	Finally, (a), (b) together with a topological argument allow us to conclude that the projection map $\bP_t$ in Theorem \ref{Thm_Isol_Main} \ref{Enum_MainThm_t>0} is a diffeomorphism when $r_\circ, t_\circ\ll 1$.

	\subsection{Organization of the paper}
	In Section \ref{S:preliminary}, we discuss preliminary concepts and results. In particular, we will recall the notions of weak flows, non-collapsing, and previous results of nondegenerate singularities in \cite{SunXue2022_generic_cylindrical}. In Section \ref{Sec_L^2 Mono}, we introduce the central analytic tool of this paper, the $L^2$-monotonicity formula and decay order. In Section \ref{S:GandTProperty}, we prove the geometric and topological properties of flow passing through nondegenerate singularities. In Section \ref{S:Application}, we complete the proof of Theorem \ref{Thm_Isol_Main}. Finally, we have three Appendices with some technical details.
	
	\subsection*{Acknowledgment}
	We thank Boyu Zhang for the helpful discussion about topology. A.S. is supported by the AMS-Simons Travel Grant. J. X. is supported by NSFC grants (No. 12271285) in China, the New Cornerstone investigator program and the Xiaomi endowed professorship of Tsinghua University.

	\section{Preliminary}
	\label{S:preliminary}
	In this section, we provide some preliminaries that will be used in later proofs. These include:
	\begin{enumerate}
		\item The Jacobi operator $L_{n,k}$ on the generalized cylinder as well as its eigenvalues and eigenfunctions;
		\item The notions of weak mean curvature flows, including Brakke flow and weak set flow;
		\item Nondegenerate cylindrical singularities;
		\item Partial classification of noncollapsing ancient mean curvature flows.
	\end{enumerate}
	\subsection{Geometry of generalized cylinders as shrinkers}\label{SS:PreCylinder}
	Given $0<k<n$, let 
	\begin{itemize}
		\item $\R^{n+1}=\R^{n-k+1}\times\R^k$, parametrized by $X = (x, y)$. Let $\orig$ be the origin in $\RR^{n+1}$. For every $R>0$, we shall also work with $Q_R:= \BB^{n-k+1}_R\times \BB^k_R$.
		\item For every subset $E\subset \RR^{n+1}$, $X_\circ\in \RR^{n+1}$ and $\lambda>0$, we use $\lambda\cdot (E+X)$ to denote the image of $E$ under the translation and dilation map $\eta_{X_\circ, \lambda}: X\mapsto \lambda(X+X_\circ)$; 
		\item $\cC_{n,k}:= \SSp^{n-k}(\sqrt{2(n-k)})\times\R^k$ be the round cylinder in $\RR^{n+1}$, parametrized by $X=(\theta, y)$. 
		The radius of the spherical part, usually denoted by $\varrho:=\sqrt{2(n-k)}$, is chosen such that $t\mapsto \sqrt{-t}\,\cC_{n,k}$ is a mean curvature flow on $t\leq 0$, or equivalently, $\cC_{n,k}$ satisfies the shrinker equation $\vec H+\frac{X^\perp}{2}=0$. 
		We shall write $\spine(\cC_{n,k}):= \{0\}\times \RR^k$, which is the linear subspace of $\RR^{n+1}$ in which $\cC_{n,k}$ is translation invariant.
	\end{itemize}
	
	Throughout this paper, we use $L^2=L^2(e^{-|X|^2/4})$ to denote the weighted $L^2$ space, which is the completion of compactly support smooth function space $C_c^\infty(\cC_{n,k})$ with respect to the weighted norm \[
	\|u\|_{L^2}=\left(\int_{\cC_{n,k}}|u|^2e^{-\frac{|X|^2}{4}}d\cH^n(X)\right)^{1/2}. 
	\]
	Similarly, if $\Omega\subset \cC_{n,k}$ is a measurable subset, we let $\|u\|_{L^2(\Omega)}:= \|u\cdot\chi_\Omega\|_{L^2}$, where $\chi_\Omega$ is the characteristic function of $\Omega$.
	
	For every function $u>-\varrho$ defined on a subdomain $\Omega\subset \cC_{n,k}$, we define its graph to be \[
	\graph_{\cC_{n,k}}(u):= \{(\theta, y)+ u(\theta, y)\theta/|\theta|: (\theta, y)\in \Omega\}\,.
	\]
	Geometric properties of graphs used in this paper are collected in Appendix \ref{Append_Graph over Cylind}. For a smooth hypersurface $M\subset\RR^{n+1}$, we call the following $R$ \textbf{graphical radius} of $M$: \[
	R:= \sup\{R'>\varrho: M\cap Q_{R'} = \graph_{\cC_{n,k}}(u), \text{ with } \|u\|_{C^1}\leq \min\{\kappa_n, \kappa_n', \kappa_n''\}\}
	\]
	where $\kappa_n, \kappa_n', \kappa_n''\in (0, 1/2)$ are determined by Lemma \ref{Lem_App_Graph over Cylinder} (iii), (iv) and \ref{Lem_App_RMCF equ}. $Q_R$ is called the \textbf{graphical domain}, and the corresponding $u\in C^1(\cC_{n,k}\cap Q_R)$, extended outside $Q_R$ by $0$, is called the \textbf{graphical function of $M$ over $\cC_{n,k}$}. If the set on the right hand side is empty, we just ask the graphical radius and graphical function to be both $0$. Similar notion can be defined when $M$ is a Radon measure, in which case we replace ``$M\cap Q_{R'} = \graph_{\cC_{n,k}}(u)$" by ``$M\llcorner Q_{R'}$ is the assocaited Radon measure of $\graph_{\cC_{n,k}}(u)$".
	
	We are interested in the following linear operator on $\cC_{n,k}$, known as the \textit{Jacobi operator}:
	\begin{equation}
		L_{n,k} u=\Delta u-\frac{1}{2}\langle y,\nabla_y u\rangle +u.
	\end{equation}
	It is self-adjoint with respect to the weighted $L^2$ space.  
	
	In the following, we use the notation that $\lambda$ is an \textit{eigenvalue} of an elliptic operator $-\mathscr L$ if $\mathscr Lf+\lambda f=0$ has a nonzero $L^2$ solution. In \cite[Section 5.2]{SunWangZhou20_MinmaxShrinker}, it was proved that the eigenvalues of $-L_{n,k}$ are given by
	\begin{align}
		\sigma(\cC_{n,k}) := \{\mu_i+j/2-1\}_{i,j=0}^\infty  \,, \label{Equ_Pre_sigma(C_n,k)}
	\end{align}
	with corresponding eigenfunctions spanned by $\{\phi_i(\theta) h_j(y)\}_{i,j=0}^\infty$, where $\mu_i = \frac{i(i-1+n-k)}{2(n-k)}$ and $\phi_i$ are eigenvalues and eigenfunctions of $-\Delta_{\SSp^{n-k}(\varrho)}$, and $h_j$ is a degree $j$ Hermitian polynomial on $\R^{k}$.
	
	The first several eigenvalues and eigenfunctions of $-\Delta_{\SSp^{n-k}(\varrho)}$ are listed as follows: 
	\begin{itemize}
		\item constant functions for eigenvalue $0$;
		\item $\theta_i$, the restriction of linear functions in $\R^{n-k+1}$ to $\bS^{n-k}(\varrho)$, for eigenvalue $1/2$;
		\item $\theta_{i_1}^2-\theta^2_{i_2},\cdots$ for eigenvalue $\frac{n-k+1}{n-k}$.
	\end{itemize}
	The Hermite polynomials are eigenfunctions of $-(\Delta_{\R^k}-\frac{1}{2}\langle y, \nabla_{\R^k} \cdot\rangle)$ on $\R^{k}$, and degree $j$ Hermite polynomial has eigenvalue $j/2$.  We summarize the first several eigenvalues and eigenfunctions of $-L_{n,k}$ as follows (see \cite[Section 2]{SunXue2022_generic_cylindrical})
	\begin{table}[H]
		\begin{tabular}{|l|l|}
			\hline
			eigenvalues of $-L_{n,k}$ & corresponding eigenfunctions \\ \hline
			$-1$ & $1$ \\ \hline
			$-1/2$ & $\theta_i,y_j,\ i=1,2,\ldots,n-k+1,\ j=1,2,\ldots,k$ \\ \hline
			$0$ & $\theta_iy_j,\ h_2(y_j)=(y_j^2-2),\ y_{j_1}y_{j_2}$ \\ \hline
			$\min\{1/(n-k),1/2\}$ & $\ldots$ 
			\\ \hline
		\end{tabular}
		\caption{Eigenvalues and eigenfunctions of $-L_{n,k}$.}
		\label{TableEigen}
	\end{table}

	\subsection{Brakke flow and weak set flow} \label{Subsec_Brakke and WSF}
	
	We first recall some basic notions for mean curvature flows. Suppose $M$ is a compact $n$-dimensional manifold with or without boundary, $\Int(M)$ is the interior of $M$. Let $F: M\times[a,b]\to \R^{n+1}$ be a continuous one-one map that is smooth on $\Int(M)\times(a,b]$ such that $f(\cdot,t)$ smoothly embeds $\Int(M)$ for each $t\in (a,b]$. If $F$ satisfies the equation 
	\[
	(\pr_t F)^\perp(x,t) =\vec{H}(x,t)
	\]
	for all $(x,t)\in \Int(M)\times(a,b]$, then 
	\begin{equation}
		\cK:=\{(F(x,t),t):x\in M, t\in[a,b]\}
	\end{equation}
	is called the spacetime of a \textbf{classic mean curvature flow}, or just a classic flow for short. For $t\in\R$, $\cK(t)$ denotes the time-slice $\{X\in \RR^{n+1}: (X, t)\in\cK\}$, and for an interval $(a,b)\subset\R$, we use $\cK(a,b)$ to denote $\cK\cap(\R^{n+1}\times(a,b))$. We use $\pr\cK$ to denote the \textbf{heat boundary} of $\cK$, defined by
	\[\{(F(x,t),t):\text{either $t=a$ or $x\in\pr M$}\}.\]

	In this paper, we need the following two notions of weak mean curvature flows. The first one is a measure-theoretic weak solution called \textit{Brakke flow}. 
	
	\begin{definition}
		An $n$-dimensional \textbf{(integral) Brakke flow} in $\RR^{n+1}$ over an interval $I\subset \RR$ is a one-parameter family of Radon measures $t\mapsto\Sigma_t$, such that for almost every $t\in I$, $\Sigma_t$ is a Radon measure associated to an $n$-dimensional integral varifold with mean curvature $\vec H_t\in L^2(\Sigma_t)$, and for every non-negative function $\Phi\in C^2(\R^{n+1}\times\R_{\geq 0})$, we have
		\begin{equation}
			\frac{d}{dt}\int \Phi\ d\Sigma_t
			\leq 
			\int \left(
			\frac{\partial \Phi}{\partial t}
			+\nabla \Phi\cdot \vec{H}_t
			-\Phi |\vec H_t|^2
			\right)
			\ d\Sigma_t, \label{Equ_Pre_DefBrakke}
		\end{equation}
		in the distribution sense. 
		
		Note that for almost every $t$, by \cite[$\S\, 5.8$]{Brakke78}, $\vec H_t$ is perpendicular to the varifold tangent $\Sigma_t$-almost everywhere. The support of the Brakke flow is defined to be $\overline{\bigcup_t \Spt \Sigma_t\times\{t\}}$, where the closure is taken in the spacetime.
	\end{definition}
	Recall that a point $p_\circ = (X_\circ, t_\circ)\in \RR^{n+1}\times \RR$ in the support of a Brakke flow $\bM: t\mapsto \bM(t)$ is {\bf regular} if in a spacetime neighborhood of $p_\circ$, $\bM$ is the varifold associated to a classic mean curvature flow. Otherwise, we say $p_\circ$ is a {\bf singularity}.
	
	Given a Brakke flow $\bM: t\mapsto \bM(t)$, a point $p_\circ = (X_\circ, t_\circ)\in \RR^{n+1}\times \RR$ and a constant $\lambda>0$, we use the notation $\lambda\cdot(\bM - p_\circ)$ to denote the Brakke flow $t\mapsto \bM'(t)$ given by the space-time translation and parabolic dilation of $\bM$: \[
	\bM'(t):= \lambda\cdot (\bM(\lambda^{-2}t+t_\circ) -X_\circ) \,.
	\]
	Here we use the convention that for an integral varifold $M$, $\lambda (M - X_\circ)$ is the push forward of $M$ by the translation-rescaling map $X\mapsto \lambda(X-X_\circ)$. 
	
	The Gaussian density plays an important role in the study of mean curvature flow. Recall that given a hypersurface $M\subset\R^{n+1}$, the $n$-dimensional {\bf Gaussian area} is defined by \[
	\cF[M]:=\int_{\RR^{n+1}} (4\pi)^{-n/2}e^{-\frac{|X|^2}{4}}d\|M\|(X).
	\]
	Here $\|M\|$ denotes the Radon measure associated to $M$, for a hypersurface, this is just the $n$-dimensional volume measure. The same notion can be defined for a Radon measure $\mu$, where we replace $d\|M\|$ by $d\mu$.
	
	Colding-Minicozzi \cite{ColdingMinicozzi12_generic} introduced a quantity called {\bf entropy} that is the supremum of the Gaussian area of all possible translations and dilations of a hypersurface (or an integral $n$-varifold). \[
	\lambda[M]:=\sup_{(X_\circ,t_\circ)\in\R^{n+1}\times\R_{>0}} \cF[t_\circ^{-1}(M-X_\circ)].
	\]
	For a Brakke motion $t\mapsto \bM(t)$ over $I$, we define its entropy as \[
	\lambda[\bM]:= \sup_{t\in I} \lambda[\bM(t)] \,.
	\]
	
	Given a Brakke flow $t\mapsto \Sigma_t$ over $I$ and a spacetime point $p_\circ = (X_\circ,t_\circ)\in\R^{n+1}\times\R$, we let \[
	\Theta_{p_\circ}(\tau) := \int (4\pi\tau)^{-n/2}e^{-\frac{|X-X_\circ|^2}{4\tau}}\ d\Sigma_{t_\circ-\tau},
	\]
	and the {\bf Gaussian density} of $(\Sigma_t)_{t\in I}$ at $p_\circ$ is defined by \[
	\Theta((\Sigma_t)_{t\in I}, p_\circ)=\lim_{\tau\searrow 0} \Theta_{p_\circ}(\tau).
	\]
	By Huisken's monotonicity formula \cite{Huisken90}, $\Theta_{p_\circ}(\tau)$ is monotone nondecreasing in $\tau$, thus this limit always exists. 
	
	Following \cite{White09_CurrentsVarifolds}, an integral Brakke flow $t\mapsto \Sigma_t$ over $I$ is called {\bf unit-regular} if for any $p_\circ\in \spt \bM$ with $\Theta((\Sigma_t)_{t\in I},p_\circ)=1$ the Brakke flow is regular in a parabolic neighborhood of $p_\circ$. It is called {\bf cyclic} (mod-$2$) if for a.e. $t\in I$, $\Sigma_t$ is the Radon measure associated to an integral varifold $V(t)$, whose associated rectifiable mod-$2$ flat chain $[V(t)]$ satisfies $\pr[V(t)]=0$.  
	White \cite{White09_CurrentsVarifolds} proved that the unit-regular cyclic Brakke flows can be obtained by Ilmanen's elliptic regularization, which we will briefly review in section \ref{SSSec_Ellip Reg}. 
	
	Based on his monotonicity formula, Huisken introduced a blow-up scheme, defined as follows. Given a Brakke flow $t\mapsto \mbfM(t)$ and a spacetime point $p_\circ=(X_\circ,t_\circ)$, one can define a new flow $\cM: \tau\mapsto \cM(\tau)$ associated to $\mbfM$, called \textit{rescaled mean curvature flow (RMCF)} based at $p_\circ$, where its time slice $\cM(\tau)$ is defined by
	\begin{equation}
		\cM(\tau)=e^{\tau/2}(\mbfM(t_\circ-e^{-\tau})-X_\circ).
	\end{equation}
	Huisken proved that the RMCF is the gradient flow of the Gaussian area, and the limit is a shrinker. Hence the RMCF is a central tool to study singularities.
	
	RMCF associated to a given mean curvature flow $\mbfM$ rely on the choice of the base point. In fact, if $\cM$ is the RMCF associated to $\bM$ (i.e. based at $(\orig, 0)$), and $\cM^{p_\circ}$ is the RMCF of $\bM$ based at $p_\circ$, then we have 
	\begin{equation}\label{eq:RMCF change of base point}
		\cM^{p_\circ}(\tau)=\sqrt{1-t_\circ e^{\tau}}\cdot
		\cM(\tau-\log(1-t_\circ e^\tau))-e^{\tau/2}X_\circ
		\,.
	\end{equation}

	Another notion of weak flow is motivated by the maximum principle, called \textbf{weak set flow}. 
	
	\begin{definition}[Weak set flow defined by White \cite{White95_WSF_Top}]\label{def:WSF}
		Given a closed set $\Gamma\subset\R^{n+1}\times\R_{\geq 0}$. A \emph{weak set flow generated by $\Gamma$} is a closed subset $\cK\subset \R^{n+1}\times\R_{\geq0}$ with the following significances
		\begin{itemize}
			\item $\cK$ and $\Gamma$ coincide at time $0$.
			\item If $\cK'$ is the spacetime of a mean curvature flow of smoothly embedded hypersurfaces, such that the heat boundary $\pr\cK'$ is disjoint from $\cK$ and $\cK'$ is disjoint from $\Gamma$, then $\cK'$ is disjoint from $\cK$.
		\end{itemize}
	\end{definition}
	
	\begin{remark}\label{Rem:LSF}
		There exists a ``biggest flow'', namely a weak set flow that contains all the weak set flows generated by $\Gamma$. Such a special flow is called the {\bf level set flow}. The existence of such a flow was proved in the pioneering work of level set flow by Evans-Spruck \cite{EvansSpruck91}, and its relation to the weak set flow was discovered by Ilmanen \cite{Ilmanen92_LSF} and White \cite{White95_WSF_Top}. 
		
		A particularly interesting class of the weak set flow is the \textbf{mean convex} weak set flows, namely the flows whose different time-slices are disjoint. The name follows from the fact that if the flow is the boundary of some domain, then all the regular point of the flow has positive mean curvature with respect to the outward unit normal vector field. 
	\end{remark}
	
	These two definitions of weak flows have the following relations.
	
	\begin{theorem}[\cite{Ilmanen94_EllipReg, HershkovitsWhite23_Avoidance}]
		The closure of the support of a unit regular Brakke flow in $\R^{n+1}$ is a weak set flow.
	\end{theorem}

	The weak set flow can be very different from the spacetime of a mean curvature flow. For example, the weak set flow can generate interior, and such a phenomenon is known as {\bf fattening}. Such a phenomenon is proved to exist by Ilmanen-White \cite{IlmanenWhitw24_Fattening} (also see another approach by \cite{LeeZhao24_MCFconical} using the results of \cite{AngenentIlmanenVelazquez17_fattening} and \cite{Ketover24_Fattening}). Ilmanen \cite{Ilmanen94_EllipReg} (see also \cite{White09_CurrentsVarifolds}) proved that if the level set flow does not fatten, then the level set flow is the support of a unit regular cyclic Brakke flow that is constructed via the elliptic regularization. Namely, the two notions of weak flows can be identified.
	
	Evans-Spruck \cite{EvansSpruck91} and Ilmanen \cite{Ilmanen94_EllipReg} proved that the level set flow generated by a mean convex hypersurface in $\R^{n+1}$ will not fatten. Moreover, Evans-Spruck proved the following: if $\Omega$ is a mean convex domain in $\R^{n+1}$, then the spacetime track of mean curvature flow starting from $\Omega$ can be written as a function $\mathbf{f}:\overline\Omega\to \R$, such that $\mathbf{f}(x)=0$ for $x\in\pr\Omega$ and $\{(x,t):\mathbf{f}(x)=t\}\subset\R^{n+1}\times\R$ is the level set flow generated by $\pr\Omega$. $\mathbf{f}$ is called the {\bf arrival time function} because $\{(x,t):\mathbf{f}(x)=t\}$ is the time slice of the level set flow. For example, the arrival time function of a shrinking sphere starting with radius $r_0$ in $\R^{n+1}$ is given by $(r_0^2-|x|^2)/2n$.
	
	In general, it is hard to determine if a level set flow will fatten or not. Hershkovits-White \cite{HershkovitsWhite20_Nonfattening} proved that if the singularities of a level set flow have mean convex neighborhood, then the flow does not fatten. This is in fact a property of nondegenerate cylindrical singularities. Hence, throughout this paper, the flows that we study will not fatten (see Proposition \ref{Prop:MCF with boundary}, especially conclusion \ref{Item_EllReg_LSFnonFat}). Thus, we do not specify which flow or weak flow that we are referring to. We also remark that the definition of rescaled mean curvature flow can be naturally extended to all types of weak flows.

	\subsection{Nondegenerate singularity and its property before singular time.}
	A spacetime singularity of a mean curvature flow is said to be {\bf cylindrical} if the rescaled mean curvature flow with the singularity as the based point $C_{\loc}^\infty$-converges to a generalized cylinder $\cC_{n,k}$ as $\tau\to\infty$. Colding-Ilmanen-Minicozzi \cite{ColdingIlmanenMinicozzi15} and Colding-Minicozzi \cite{CM15_Lojasiewicz} proved that if a rescaled mean curvature flow $\cM(\tau)$ subsequentially converges to $\cC_{n,k}$, then it converges to $\cC_{n,k}$ smoothly in any compact subset of $\R^{n+1}$. In particular, when $\tau$ is getting larger and larger, $\cM(\tau)$ can be written as a smooth graph of a function $u(\cdot,\tau)$ over a larger and larger domain in $\cC_{n,k}$.
	
	In \cite{SunXue2022_generic_cylindrical}, the first and third named authors proved a normal form theorem on $u$ as follows.
	\begin{theorem}[Theorem 1.3 and 1.4 in \cite{SunXue2022_generic_cylindrical}]\label{thm:NormalForm}
		Let $\cM=\{\cM(\tau)\}_{\tau\geq 0}$ be a RMCF such that $\cM(\tau)$ $C^\infty_{loc}$-converges to $\cC_{n,k}$ as $\tau\to \infty$, and suppose $\lambda[\cM]<+\infty$. Then there exist $T_0>0$, $\cI\subset\{1,2,\cdots,k\}$ and $u\in C^\infty(\cC_{n,k}\times \R_{\geq T_0})$ such that $\cM(\tau)=\graph_{\cC_{n,k}}(u(\cdot, \tau))$ in $Q_{\sqrt\tau}$\footnote{In \cite{SunXue2022_generic_cylindrical}, it was stated that there exists $K>0$ such that the normal form theorems hold inside $Q_{K\sqrt{\tau}}$, but in fact for any $K>0$, the normal form theorems hold where $T_0$ depend on $K$.}, with
		\begin{align}
			\left\|u(\theta,y,\tau)-\varrho\cdot\left(\sqrt{1+\tau^{-1}\sum_{i\in \cI}\frac{y_i^2-2}{2}}-1\right)\right\|_{C^1(\SSp^{n-k}(\varrho)\times \BB_{\sqrt{\tau}}^k)} \to 0,\quad \tau \to\infty. \label{Equ_Pre_C^1 asymp in B_sqrt(tau)}
		\end{align}
		and 
		\begin{equation}
			\left\|u(\theta,y,\tau)-\sum_{i\in \mathcal I}\frac{\varrho }{4\tau}(y_i^2-2)\right\|_{H^1(\SSp^{n-k}(\varrho)\times \BB_{\sqrt{\tau}}^k)}=O(1/\tau^2), \quad \tau \to\infty. \label{Equ_Pre_H^1 asymp in B_sqrt(tau)}
		\end{equation}
	\end{theorem}

	\begin{definition}\label{DefNondegenerate}
		A cylindrical singularity is called {\bf nondegenerate} if the associated rescaled mean curvature flow with the base point at this singularity satisfies the condition of Theorem \ref{thm:NormalForm} with the index set $\cI=\{1,2,\ldots, k\}$. 
	\end{definition}

	Notice that when $\tau$ is very large, for $|y|$ bounded by a constant, we have 
	\[
	\varrho\left(\sqrt{1+\tau^{-1}\sum_{i\in \cI}\frac{y_i^2-2}{2}}-1\right)\approx \frac{\varrho}{4\tau}\sum_{i\in\cI}(y_i^2-2).
	\]
	So nondegeneracy can also be understood as the leading order asymptotic of the graph function $u$ is given by $\frac{\varrho}{4\tau}\sum_{i=1}^k (y_i^2-2)$ in $L^2$-sense.
	
	A key feature of a nondegenerate singularity is that the associated rescaled mean curvature flow is almost a generalized cylinder near the boundary of $\SSp^{n-k}(\varrho)\times \BB_{\sqrt{\tau}}^k$, with slightly larger radius compared with the shrinking cylinder. For any $R_0>0$ and sufficiently large $\tau$, within $\SSp^{n-k}(\varrho)\times( \BB_{\sqrt{\tau}}^k\backslash \BB_{\sqrt{\tau}-R_0}^k)$, $u(\cdot,\tau)\approx \varrho(\sqrt{1+1/2}-1)=:\varrho'-\varrho>0$, and hence $\cM(\tau)$ is very close to a cylinder with radius $\varrho'$ inside this annulus region.

	\subsection{Classification of noncollapsing ancient solutions} \label{Subsec_Classif by Du-Zhu}
	\textit{Noncollapsing} is a central feature of mean convex mean curvature flow. In a series of papers, White \cite{White00, White03, White15_SubseqSing_MeanConvex} proved that the mean convex mean curvature flows do not admit ``collapsing''(i.e. multiplicity $\geq 2$) blow-up limit. Later Weimin Sheng and Xu-Jia Wang \cite{ShengWang09_SingMCF} introduced a quantitative version of the concept of noncollapsing. Sheng and Wang \cite{ShengWang09_SingMCF} and Andrews \cite{Andrews12_Noncollapsing} proved that this quantity is preserved under the mean convex mean curvature flow, giving an alternative proof of White's result.
	
	\begin{definition} \label{Def_App_Noncollap}
		Given $\alpha>0$. A smooth mean convex hypersurface $M = \partial \Omega$ is called {\bf $\alpha$-noncollapsing} if
		\begin{align}
			\frac{1}{2\alpha} H(x)|x-y|^2  \geq \big|\langle x-y,\bn(x)\rangle \big|\,.      
		\end{align} 
		holds for all $x,y\in M$. Here $\bn$ is the unit outward normal vector field. In particular, the largest $\alpha$ is called the \textbf{Andrews constant}. 
		
		If $M$ is not smooth, it is $\alpha$-noncollapsing if the above inequalities hold for all regular points $x\in M$ respectively.
	\end{definition} 
	
	Geometrically, this means that there is a ball of radius $\alpha H^{-1}(x)$ which lies inside/outside the region bounded by $M$ which touches $M$ at $x$. $\alpha$-noncollapsing is scaling invariant and can be passed to limit flows.
	
	In this paper, the ``noncollapsing" is used to apply a classification theorem of Wenkui Du and Jingze Zhu \cite{DuZhu22_quantization} in every dimension (see also the early work \cite{AngenentDaskalopoulosSesum19_AncientConvexMCF, BrendleChoi19_AncientMCFR3, BrendleChoi19_AncientMCFRn, ChoiHaslhoferHershkovits18_MeanConvNeighb, ChoiHaslhoferHershkovitsWhite22_AncientMCF, ChoiHaslhoferHershkovits21_TranslR4, ChoiDaskalopoulosDuHaslhoferSesum22_BubblesheetR4, DuHaslhofer24_Hearing} in other settings).
	
	To state their main theorem that we need, we first discuss some background on bowl soliton.
	For $m = n-k+1$, a {\bf translator} in $\R^{m+1}$ is a hypersurface $S$ satisfying the equation $\vec{H}=\vec{V}^\perp$, for some non-zero vector $\vec V\in \RR^{m+1}$. The name follows from the fact that for such $S$, $t\mapsto S+t\vec V$ is a mean curvature flow over $\R$.
	
	For $m\geq 2$, there exists a translator $\cB^{m}\subset\R^{m+1}$ called the {\bf bowl soliton}, first discovered by Altschuler-Wu \cite{AltschulerWu94_Transl}. $\cB^{m}$ is the boundary of a convex set and we let $\nu$ be the unit normal vector field pointing outwards from it. If we use $(x_1,\cdots,x_{m}, z)$ as the coordinates of $\R^{m+1}$, then $\cB^m$ is constructed as a graph of a convex function $U(x)$ over $\{z=0\}$ with the asymptotic 
	\[
	U(x)=\frac{|x|^2}{2(m-1)}-\log |x|+O(|x|^{-1}),\quad x\to\infty,
	\]
	And $t\mapsto \cB^m+t\partial_z$ is a mean curvature flow, where $\pr_z := (0, \dots, 0, 1)$. 
	
	Using $\cB^m$, we can construct a family of translators in $\RR^{n+1} = \RR^{m+1}\times \RR^{k-1}$: For every orthogonal matrix $\Omega\in O(n+1)$, $X_\circ\in \RR^{n+1}$ and $\lambda>0$, $\cB':=\lambda\cdot \Omega(\cB^m\times \RR^{k-1})+X_\circ$ is a translator in $\RR^{n+1}$. We call $\Omega(\partial_z\oplus 0)$ the \textbf{translating direction} of $\cB'$ \footnote{Note that $t\mapsto (\cB^m\times\R^k)+t(\pr_z\oplus\mathbf e')$ is a mean curvature flow for any $\mathbf e'\in\R^k$.}, 
	and denote by $\spine(\cB'):= \Omega(0\oplus\RR^{k-1})$. 
	For later reference, we also let 
	\begin{align*}
		\mathscr{B}_{n,k} & := \left\{\lambda\cdot \Omega(\cB^m\times \RR^{k-1})+X_\circ: X_\circ\in \RR^{n+1},\ \lambda>0,\ \Omega\in O(n+1)\right\}\,;
		\\
		\overline{\mathscr{B}_{n,k}} & := \mathscr{B}_{n,k}\cup \left\{\lambda\cdot \Omega(\cC_{n,k})+X_\circ: X_\circ\in \RR^{n+1},\ \lambda>0,\ \Omega\in O(n+1)\right\} \,.
	\end{align*}
	The following Lemma can be proved by a direct compactness argument which we skip here.
	\begin{Lem} \label{Lem_Pre_Compare h_1 dominate direction with transl direction}
		There exists $\delta_0(n)\in (0, 1/4)$ and $R_0(n)>2n$ such that if $\hy, \by\in \{\orig\}\times\RR^k$ are unit vectors, $R\geq R_0$, $\cB\in \mathscr{B}_{n,k}$ has translating direction $\hy$ and $\spine = \hy^\perp\cap \{\orig\}\times\RR^k$ such that \[
		\mbfd_{n,k}(\cB)\leq \delta_0\,,
		\]
		where $\mbfd_{n,k}$ will be defined in \eqref{Equ_L^2 Mono_Def L^2 dist to C_(n,k)}, and that $\cB$ is graphical over $\cC_{n,k}$ in $Q_R$ with graphical function $v$ ($0$-extended outside $Q_R$). Then \[
		\inf_{c>0, c'\in \R} \|c^{-1}v - c' - \by\|_{L^2} \geq \delta_0|\by-\hy|\,.
		\]
		And the outward normal vector field $\nu$ of $\cB$ satisfies $\nu\cdot (0, \hy) <0$ on $\cB$. 
	\end{Lem}
	
	Now we summarize the result by Du and Zhu. Recall that an \textit{ancient flow} is a mean curvature flow that is defined for time $(-\infty,T)$. First, Du and Zhu proved a normal form theorem for the ancient mean curvature flows whose blow-down limit at $-\infty$ time is a generalized cylinder. Such an ancient flow is called an \textit{ancient asymptotically cylindrical flow}.
	
	\begin{theorem}[Theorem 1.2 of \cite{DuZhu22_quantization}]\label{ThmDuZhu22_quantization}
		For any ancient asymptotically cylindrical flow in $\R^{n+1}$ whose tangent flow at $-\infty$ is given by $\SSp^{n-k}(\sqrt{2(n-k)|t|})\times \R^k$ for some $1\leq k\leq n-1$, the flow can be written as a graph of the cylindrical profile function $u$ satisfies the following sharp asymptotic
		\begin{equation}
			\lim_{\tau\to-\infty} \left\|
			|\tau|u(y,\theta,\tau)-y^\top Qy+2\tr(Q)
			\right\|_{C^p(\BB_R)}=0
		\end{equation}
		for all $R>0$ and all integer $p$, where $Q$ is a $k\times k$-matrix whose eigenvalues are quantized to be either $0$ or $-\frac{\sqrt{2(n-k)}}{4}$.
	\end{theorem}
	
	Next, among other things, they classify all such noncollapsing ancient asymptotically cylindrical flow with $\text{rk}(Q)=0$.
	
	\begin{theorem}[Theorem 1.10 in \cite{DuZhu22_quantization}]
		For an ancient noncollapsed mean curvature flow in $\R^{n+1}$ whose tangent flow at $-\infty$ is given by $\SSp^{n-k}(\sqrt{2(n-k)|t|})\times \R^k$ for some $1\leq k\leq n-1$, and $\text{rk}(Q)=0$. Then it is either a round shrinking cylinder $\SSp^{n-k}(\sqrt{2(n-k)|t|})\times \R^k$ or $(n-k+1)$ dimensional bowl soliton times $\R^{k-1}$ where the bowl soliton has translating direction orthogonal to $\R^{k-1}$ factor.
	\end{theorem}
	\begin{remark} \label{Rem_Pre_exp decay near -infty implies bowl or cylinder}
		The way that we shall apply the last theorem is as follows. We shall construct an ancient solution from the flow after passing through a nondegenerate singularity by controlling the decay order carefully to make sure that the solution converges to a cylinder exponentially fast as $t\to-\infty$, which   by Theorem \ref{ThmDuZhu22_quantization} forces $Q=0$ then by the last theorem can only be a shrinking cylinder or a bowl soliton, that is, a mean curvature flow with time slices in $\overline{\mathscr B_{n,k}}$. 
	\end{remark}

	\subsection{Use of constants.} %Conventions.
	Throughout this paper we shall use the letter $C$ to denote a constant that is allowed to vary from line to line (or even within the same line); we shall stress the functional dependence of any such constant on geometric quantities by including them in brackets, writings things like $C=C(n,\eps)$.
	We shall also use $\Psi(\eps|C_1, C_2, \dots, C_l)$ to denote a constant depending on $\eps, C_1, \dots, C_l$ and tending to $0$ when $C_1, \dots, C_l$ are fixed and $\eps\to 0$.
	
	\section{An $L^2$-distance monotonicity and applications}  \label{Sec_L^2 Mono}
	The goal of this section is to prove the following monotonicity of $L^2$-distance function and discuss its applications.
	We begin with the following set-up. Let $\nabla = \nabla_{\RR^{n+1}}$, $\Delta:= \Delta_{\RR^{n+1}}$. 
	\begin{enumerate}[label={}] %\arabic*; \alph*
		\item [$\mathbf{(S1)}$] Let $\eta\in C^2(\RR^{n+1}\times \RR)$ be a non-negative function. 
		% such that $\spt(\eta)$ is simply connected.
		\item [$\mathbf{(S2)}$] Let $T_\circ>0$, $f_\circ\in C^\infty(\RR)$ satisfies $f_\circ(-s)=-f_\circ(s)$ and
		\begin{align*}
			f_\circ(s) = \begin{cases}
				s \,, & \text{ if } |s|\leq T_\circ\,, \\
				3T_\circ/2\,, & \text{ if }s>2T_\circ\,, 
			\end{cases}  & & f_\circ''(s)\leq 0 \, \text{ on }\RR_{\geq 0}\,, & & |f_\circ'(s)|+|f_\circ(s)f_\circ''(s)| \leq 2024\, \text{ on }\RR\,.
		\end{align*}
		\item [$\mathbf{(S3)}$] Let $\Omega\subset \RR^{n+1}$ be an open subset containing the origin $\orig$, and $\bu:\Omega\to \RR$ be a twice differentiable arrival-time-function to the level set equation, i.e. 
		\begin{align}
			|\nabla \bu|^2(\Delta \bu + 1) = \nabla^2\bu (\nabla \bu, \nabla \bu)\,. \label{Equ_L^2 Mono_Arrv time func-LSF equ}
		\end{align}
		Let $\Omega_{\bu, T_\circ, \eta}\subset \Omega$ be the orthogonal projection of $\cO_{\bu,T_\circ}\cap \spt(\eta)$ onto $\RR^{n+1}$, where
		\begin{align*}
			\cO_{\bu, T_\circ}:= \{(X, t)\in \Omega\times [-T_\circ, {T_\circ}]: |\bu(X)-t|\leq 2T_\circ\}\subset \RR^{n+1}\times \RR \,,         
		\end{align*}
		
		Suppose $\bu$ satisfies $\bu(\orig) = |\nabla \bu(\orig)| = 0$ and the following for some $\beta\in (0, 1)$, 
		\begin{align}
			\overline{\cO_{\bu, T_\circ}\cap \spt(\eta)}\subset \Omega\times \RR\,; & &  |\nabla \bu|^2 \geq -2\beta \bu \; \text{ and } \; \frac{|\nabla \bu|}{\sqrt{T_\circ}}+|\nabla^2 \bu| \leq \beta^{-1} \;\;\; \text{ on }\Omega_{\bu, T_\circ, \eta}\,. \label{Equ_L^2 Mono_Restric of time arriv func u}
		\end{align}
		
		Also suppose that $F\eta$ extends to a $C^{1,1}$ function on $\RR^{n+1}\times [-T_\circ, 0)$, which equals to $\pm 3T_\circ\eta/2$ on $\RR^{n+1}\times [-T_\circ, 0)\setminus \cO_{\bu,T_\circ}$, where \[
		F(X, t):= f_\circ(\bu(X)-t) \,.
		\]  
		\item [$\mathbf{(S4)}$] Let $\rho = e^{2\phi}\in C^\infty(\RR^{n+1}\times [-T_\circ, 0))$ be satisfying that for any hyperplane $L\subset\R^{n+1}$,
		\begin{align}
			\partial_t \phi + \Div_{L}(\nabla\phi) + 2|\nabla\phi|^2 \leq 0\,, & & |t\nabla^2 \phi|\leq \beta^{-1}\,. \label{Equ_L^2 Mono_Backward Heat Equ along MCF}
		\end{align}
	\end{enumerate}
	\[ \]
	
	\begin{Thm} \label{Thm_L^2 Mono_Main}
		Let $T_\circ>0$, $\beta\in (0, 1)$; $\eta$, $f_\circ$, $\bu: \Omega\to \RR$, $F: \RR^{n+1}\times [-T_\circ, 0)\to \RR$ be satisfying $\mathbf{(S1)}$-$\mathbf{(S4)}$. 
		Then there are $K(n, \beta),\ c(n, \beta)>0$ with the following significance.
		
		Suppose $\{\Sigma_t\}_{t\in I}$ is a Brakke motion in $\RR^{n+1}$ with interval $I\subset [-T_\circ, 0)$.
		Then,
		\begin{align}
			\begin{split}
				& \frac{d}{dt} \left[ (-t)^{2K}\int_{\RR^{n+1}} F^2\eta^2 \rho\ d\Sigma_t \right] + (-t)^{2K}\int_{\RR^{n+1}} c(n, \beta) F^2\eta^2|\nabla \phi|^2 \rho\ d\Sigma_t \\
				&\;\; \leq (-t)^{2K}\int_{\RR^{n+1}}\left(\partial_t(\eta^2)+\Div_{\Sigma_t}\nabla(\eta^2) +4\nabla\phi\cdot \nabla (\eta^2) + 8|\nabla\eta|^2 \right)F^2\rho\ d\Sigma_t \,.         
			\end{split} \label{Equ_L^2 Mono_Main Est}
		\end{align} 
	\end{Thm}  
	
	The present section is organized as follows. We start with the proof of Theorem \ref{Thm_L^2 Mono_Main} in section \ref{Subsec_Pf L^2 Mono Thm} by choosing appropriate testing function in the definition of Brakke flow. The proof is a bit technical and we encourage the first time reader to skip it.
	Then we present an application of Theorem \ref{Thm_L^2 Mono_Main} in section \ref{Subsec_L^2 dist Mono} in a special case, i.e. when $\bu$ is the arrival-time-function of a round cylinder. Under this setting, we obtain an $L^2$-nonconcentration estimate near infinity (see Corollary \ref{Cor_L^2 Mono_Weighted L^2 Mono and Est for RMCF}), which allows us to derive an almost monotonicity of the \textit{decay order} in section \ref{Subsec_Doub Const, AR of RMCF}. 
	
	Further applications of Theorem \ref{Thm_L^2 Mono_Main} with different choices of $\bu$ will be discussed in the future work. 
	
	\subsection{Proof of Theorem \ref{Thm_L^2 Mono_Main}.}\label{Subsec_Pf L^2 Mono Thm}
	The goal of this subsection is to prove Theorem \ref{Thm_L^2 Mono_Main}. 
	Recall that $\rho=e^{2\phi}$. For any non-negative function $G\in C^2(\RR^{n+1}\times \RR)$, taking $\Phi:= G\rho$ in \eqref{Equ_Pre_DefBrakke} gives,
	\begin{align*}
		\frac{d}{dt} \int_{\RR^{n+1}} G\rho\ d\Sigma_t \leq \int_{\RR^{n+1}} (\partial_t G + \nabla G\cdot \vec H_t)\rho + (\partial_t\rho+\nabla \rho\cdot \vec H_t - \rho|\vec H_t|^2)G\ d\Sigma_t
	\end{align*}
	Also by the first variation formula, whenever $\Sigma_t$ has generalized mean curvature $\vec H_t$, we have 
	\begin{align*}
		\int_{\RR^{n+1}} (\rho\nabla G - G\nabla \rho)\cdot \vec H_t\ d\Sigma_t 
		& = \int_{\RR^{n+1}} -\Div_{\Sigma_t}(\rho\nabla G - G\nabla \rho)\ d\Sigma_t \\
		& = \int_{\RR^{n+1}} -\rho\, \Div_{\Sigma_t} (\nabla G) + G\,\Div_{\Sigma_t}(\nabla\rho)\ d\Sigma_t\,.
	\end{align*}
	By combining them with the first inequality of \eqref{Equ_L^2 Mono_Backward Heat Equ along MCF}, we obtain,
	\begin{align}
		\frac{d}{dt} \int_{\RR^{n+1}} G \rho\ d\Sigma_t \leq & \int_{\RR^{n+1}} 
		\Big(\partial_t G - \Div_{\Sigma_t}(\nabla G)\Big)\rho   - |\vec{H} - 2\nabla^\perp \phi|^2 G\rho\ d\Sigma_t  \label{Equ_L^2 Mono_Ecker-Huisken's Variation Ineq}
	\end{align}
	
	We let $G:=F^2\eta^2$ and, to save notation, we denote $d\mu_t = \rho(\cdot, t)\ d\Sigma_t$. Note that,
	\begin{align*}
		& \partial_t G - \Div_{\Sigma_t}(\nabla G) \\
		& \quad = \Big(\partial_t (F^2) - \Div_{\Sigma_t}(\nabla (F^2))\Big)\eta^2 + \Big(\partial_t (\eta^2) - \Div_{\Sigma_t}(\nabla (\eta^2))\Big)F^2- 2\nabla(F^2)\cdot \nabla_{\Sigma_t}(\eta^2)
	\end{align*}
	While again by first variation formula,
	\begin{align*}
		\int_{\RR^{n+1}} F^2\nabla(\eta^2)\cdot \vec H_t\ d\mu_t 
		& = \int_{\RR^{n+1}} -\Div_{\Sigma_t}(F^2\nabla(\eta^2)\rho)\ d\Sigma_t \\
		& = \int_{\RR^{n+1}} -\nabla (F^2)\cdot \nabla_{\Sigma_t}(\eta^2) - F^2\Div_{\Sigma_t}\nabla(\eta^2) - 2F^2\nabla(\eta^2)\cdot \nabla_{\Sigma_t}\phi\ d\mu_t
	\end{align*}
	Combining them gives,
	\begin{align*}
		\int_{\Sigma_t} \Big(\partial_t G - \Div_{\Sigma_t}(\nabla G)\Big)\ d\mu_t 
		& = \int_{\Sigma_t} \Big(\partial_t (F^2) - \Div_{\Sigma_t}(\nabla (F^2))\Big)\eta^2 + \Big(\partial_t (\eta^2) + \Div_{\Sigma_t}(\nabla (\eta^2))\Big)F^2\\
		& \qquad +\, 2F^2\nabla^\perp(\eta^2)\cdot \vec{H}_t  + 4F^2 \nabla(\eta^2)\cdot\nabla_{\Sigma_t}\phi\ d\mu_t 
	\end{align*}
	Plugging this back to \eqref{Equ_L^2 Mono_Ecker-Huisken's Variation Ineq} we get,
	\begin{align*}
		\frac{d}{dt} \int_{\RR^{n+1}} G\ d\mu_t 
		&\leq  \int_{\RR^{n+1}} \Big(\partial_t (F^2) - \Div_{\Sigma_t}(\nabla (F^2))\Big)\eta^2 + \Big(\partial_t (\eta^2) + \Div_{\Sigma_t}(\nabla (\eta^2))\Big)F^2 \\
		& \qquad\;\;\, +\left(4\nabla^\perp\eta\cdot \eta\vec{H}_t + 4\nabla(\eta^2)\cdot \nabla_{\Sigma_t}\phi - \left|\eta\vec{H}_t - 2\eta\nabla^\perp \phi\right|^2\right) F^2\ d\mu_t        
	\end{align*}
	Notice that 
	\begin{align*}
		& 4\nabla^\perp\eta\cdot \eta\vec{H}_t + 4\nabla(\eta^2)\cdot \nabla_{\Sigma_t}\phi - \frac12\left|\eta\vec{H}_t - 2\eta\nabla^\perp \phi\right|^2 \\
		=\; & 4 \left(\nabla(\eta^2)\cdot \nabla\phi + 2|\nabla^\perp \eta|^2\right) - \frac12\left|\eta\vec{H}_t - 2\eta\nabla^\perp \phi - 4\nabla^\perp \eta \right|^2 \leq 4 \left(\nabla(\eta^2)\cdot \nabla\phi + 2|\nabla^\perp \eta|^2\right)\,.
	\end{align*}
	We then deduce,
	\begin{align}
		\begin{split}
			\frac{d}{dt} \int_{\RR^{n+1}} G\ d\mu_t 
			& \leq  \int_{\Sigma_t} \Big(\partial_t (F^2) - \Div_{\Sigma_t}(\nabla (F^2))\Big)\eta^2 + \Big(\partial_t (\eta^2) + \Div_{\Sigma_t}(\nabla (\eta^2))\Big)F^2 \\
			& \qquad + 4 \left(\nabla(\eta^2)\cdot \nabla\phi + 2|\nabla^\perp \eta|^2\right)F^2 - \frac12\left|\vec{H}_t - 2\nabla^\perp \phi \right|^2\eta^2F^2\ d\mu_t\,.        
		\end{split} \label{Equ_L^2 Mono_App of Variation Ineq}
	\end{align}
	
	\textbf{Claim 1.} There exists $K(\beta, n)>0,\ c(\beta, n)\in(0,1)$ such that on $\{|\nabla \bu|>0\}\cap \{-T_\circ\leq t<0\}$, we have \[
	\Big(\partial_t (F^2) - \Div_{\Sigma_t}(\nabla (F^2))\Big)(X, t) \leq -\frac{2K}{t}F^2 - c(n, \beta)|\nabla_{\Sigma_t} F|^2 \,.
	\]
	
	\textit{Proof of Claim 1.} Let $\zeta:= \nabla \bu/|\nabla \bu|$; $\nu_t$ be a unit normal field of $\Sigma_t$, then by assumption (\ref{Equ_L^2 Mono_Arrv time func-LSF equ}),
	\begin{align*}
		\Big(\partial_t (F^2) & - \Div_{\Sigma_t}(\nabla (F^2))\Big)(X, t)
		= -2f_\circ f_\circ' - \Div_{\Sigma_t}(2f_\circ f_\circ'\nabla \bu) \\
		&\;\; = -2f_\circ f_\circ' - 2f_\circ f_\circ'\cdot(\Delta \bu - \nabla^2 \bu(\nu_t, \nu_t)) - 2\left(f_\circ f_\circ''+(f_\circ')^2\right)|\nabla_{\Sigma_t}\bu|^2 \\
		&\;\; = -2f_\circ f_\circ'\cdot\underbrace{\big(\nabla^2\bu(\zeta, \zeta) - \nabla^2\bu(\nu_t, \nu_t)\big)}_{\cA} - 2\left(f_\circ f_\circ''+(f_\circ')^2\right)|\nabla \bu|^2\underbrace{(1-(\zeta\cdot \nu_t)^2)}_{\cB} \,.
	\end{align*}

	\textbf{Case I.} When $|s|:=|\bu(X)-t|< T_\circ$, $f_\circ(s)=s$. By the assumption \eqref{Equ_L^2 Mono_Restric of time arriv func u}, it suffices to show that \[
	-(\bu(X)-t)\cA + \beta\cdot \bu(X)\cB + \frac{K}{t}(\bu(X)-t)^2 \leq 0\,,
	\]
	or equivalently (since $t<0$), \[
	K \bu(X)^2 - (\cA+2K-\beta\cB)\cdot t\bu(X)+ (K+\cA)t^2 \geq 0\,.
	\]
	So it suffices to find $K(\beta, n)>1$ such that \[
	0\leq 4K(K+\cA) - (\cA +2K - \beta \cB)^2 = 4K\beta\cdot \cB - (\cA-\beta\cB)^2 \,.
	\]
	Since $0\leq \cB\leq 1$ and $|\nabla^2 \bu|\leq \beta^{-1}$ by the assumption \eqref{Equ_L^2 Mono_Restric of time arriv func u}, the existence of such $K(\beta, n)$ follows immediately from the Claim below. \\
	\textbf{Claim 2.} For every symmetric bilinear form $S$ on $\RR^n$ and every unit vectors $v, w\in \RR^n$, we have \[
	(S(v, v) - S(w, w))^2 \leq 8\|S\|_{l^2}^2(1-(v\cdot w)^2) \,.
	\]
	\textit{Proof of Claim 2.} WLOG $S$ is diagonal with eigenvalues $\lambda_1, \dots, \lambda_n$. Set $v = (v_1, \dots, v_n)$, $w=(w_1, \dots, w_n)$. By possibly replacing $w$ by $-w$, WLOG $v\cdot w\geq 0$. Then 
	\begin{align*}
		(S(v, v) - S(w, w))^2 & = \left( \sum_{j=1}^n \lambda_j(v_j-w_j)(v_j+w_j) \right)^2 \\ 
		& \leq \left( \sum_{j=1}^n 4\lambda_j^2 \right)\left( \sum_{j=1}^n (v_j-w_j)^2\right) \\
		& = 4\|S\|_{l^2}^2 |v-w|^2 = 8\|S\|^2_{l^2}(1-v\cdot w) \leq 8\|S\|^2_{l^2}(1-(v\cdot w)^2) \,.
	\end{align*}
	\textbf{Case II.} When $|s|:=|\bu(X)-t|\geq T_\circ$, $F(s)^2\geq T_\circ^2$, so by assumption $\mathbf{(S2)}$ and \eqref{Equ_L^2 Mono_Restric of time arriv func u} we have, for every $-T_\circ\leq t<0$,
	\begin{align*}
		-f_\circ(s)f_\circ'(s)\cA - \left(f_\circ f_\circ''(s)+f_\circ'(s)^2 \right)|\nabla \bu|^2\cB \leq C(\beta)T_\circ - |\nabla_{\Sigma_t}F|^2 \leq -\frac{K(n, \beta)}{t} F^2(s) - |\nabla_{\Sigma_t}F|^2 \,.
	\end{align*}
	if we take $K(\beta, n)\gg 1$. This finishes the proof of Claim 1. \\
	
	Now using the estimate in Claim 1, (\ref{Equ_L^2 Mono_App of Variation Ineq}) implies,
	\begin{align}
		\begin{split}
			\frac{d}{dt} \int_{\RR^{n+1}} F^2\eta^2\ d\mu_t & \leq \int_{\RR^{n+1}} -\frac{2K}{t}F^2\eta^2 - c(n, \beta)|\nabla_{\Sigma_t} F|^2\eta^2 - \frac12\left|\vec{H}_t - 2\nabla^\perp \phi\right|^2 F^2\eta^2  \\
			& \qquad\ + \Big(\partial_t (\eta^2) + \Div_{\Sigma_t}(\nabla (\eta^2))+ 4\nabla(\eta^2)\cdot\nabla\phi + 8|\nabla^\perp \eta|^2 \Big)F^2 \ d\mu_t \,.
		\end{split} \label{Equ_L^2 Mono_App of Claim 1}
	\end{align}
	While by Cauchy-Schwarz inequality, 
	\begin{align*}
		\int_{\RR^{n+1}} |\nabla_{\Sigma_t} F|^2\eta^2\ d\mu_t & = \int_{\RR^{n+1}} |\nabla_{\Sigma_t}(F\eta)|^2 + |\nabla_{\Sigma_t} \eta|^2F^2 - 2\nabla_{\Sigma_t}(F\eta)\cdot (F\nabla_{\Sigma_t}\eta)\ d\mu_t \\
		& \geq \int_{\RR^{n+1}} \frac12 |\nabla_{\Sigma_t}(F\eta)|^2 - |\nabla_{\Sigma_t}\eta|^2F^2\ d\mu_t \,.
	\end{align*}
	Recall that $|t\nabla^2 \phi|\leq \beta^{-1}$.  Combining this with \eqref{Equ_L^2 Mono_App of Claim 1} and the following Ecker-type Sobolev inequality Lemma \ref{Lem_L^2 Mono_Ecker's Sobolev Ineq} proves \eqref{Equ_L^2 Mono_Main Est} with a slightly larger $K$. 
	
	\begin{Lem} \label{Lem_L^2 Mono_Ecker's Sobolev Ineq}Let $\Sigma$ be the associated Radon measure to an integral $n$-varifold in $\RR^{n+1}$ with generalized mean curvature $\vec H_\Sigma\in L^2(\Sigma)$, $\rho = e^{2\phi}$ be a $C^2$ positive function on $\RR^{n+1}$; $\mathrm F\in C^1_c(\RR^{n+1})$. Then, 
		\begin{align}
			\int_{\RR^{n+1}} \left(|\nabla_{\Sigma}  \mathrm F|^2 + \frac{\mathrm F^2}4 \left|\vec{H}_\Sigma - 2\nabla^\perp \phi \right|^2 \right)\rho\ d\Sigma \geq \int_{\RR^{n+1}} \left(|\nabla \phi|^2 + tr_\Sigma(\nabla^2 \phi) \right)\mathrm F^2\rho\ d\Sigma \,. \label{Equ_L^2 Mono_Ecker type Sobolev Ineq}
		\end{align}
	\end{Lem}
	\begin{proof}
		Observe that (again using the first variation formula), 
		\begin{align*}
			0 & \leq \int_{\RR^{n+1}} |\nabla_\Sigma(\mathrm Fe^\phi)|^2\ d\Sigma = \int_{\RR^{n+1}} |\nabla_\Sigma \mathrm F|^2\rho + |\nabla_\Sigma \phi|^2\mathrm F^2\rho + \nabla_\Sigma (\mathrm F^2)\cdot  \rho\nabla_\Sigma \phi\ d\Sigma \\
			& = \int_{\RR^{n+1}} |\nabla_\Sigma \mathrm F|^2 \rho - \left(tr_\Sigma(\nabla^2\phi) + |\nabla \phi|^2 + \vec{H}_\Sigma\cdot \nabla^\perp \phi - |\nabla^\perp \phi|^2 \right) \mathrm F^2\rho \ d\Sigma \\
			& \leq \int_{\RR^{n+1}} |\nabla_\Sigma \mathrm F|^2 \rho - \left(tr_\Sigma(\nabla^2\phi) + |\nabla \phi|^2 - \frac14 \left|\vec{H}_\Sigma - 2\nabla^\perp \phi\right|^2 \right) \mathrm F^2\rho \ d\Sigma \,.
		\end{align*}
	\end{proof}

	\subsection{$L^2$ Non-concentration near Infinity} \label{Subsec_L^2 dist Mono}
	
	We shall focus on a model example that satisfies $\mathbf{(S1)}$-$\mathbf{(S4)}$: Let $\eta\geq 0$ on $\RR^{n+1}\times \RR$ such that $\Spt \eta = \RR^{n+1}\times I$ for some interval $I\subset [0, 1]$; $T_\circ=1$, $\Omega= \BB^{n-k+1}_{\sqrt{8(n-k)}}\times \RR^k$, $\bu=\bU_{n,k}: \Omega\to \RR$ given by 
	\begin{align}
		\bU_k(x, y) = -\frac{|x|^2}{2(n-k)} \,, \label{Equ_L^2 Mono_Arriv time func for k-cylind}      
	\end{align}
	which describes the generalized shrinking cylinder: \[
	\mbfC_{n, k} := \coprod_{t\leq 0}\{(x, y)\in \RR^{n+1} : \bU_{n,k}(x, y)=t\}\times \{t\} = \coprod_{t\leq 0} (\sqrt{-t}\,\cC_{n, k})\times \{t\} \subset \RR^{n+1}\times \RR\,.
	\]
	Let $f_\circ$ be satisfying $\mathbf{(S2)}$, $F(X, t) = f_\circ(\bu(X)-t)$ be as in $\mathbf{(S3)}$, and  \[
	\rho(X, t) := \sqrt{-4\pi t}^{-n}e^{\frac{|X|^2}{4t}} 
	\]
	be the Gaussian density. It's easy to check that with the choice of $(\eta, \Omega, \bu, \rho)$ as above, $\mathbf{(S3)},\,\mathbf{(S4)}$ are satisfied with some dimensional constant $\beta=\beta(n)$.
	
	Moreover, we fix a non-decreasing odd function $\chi\in C^\infty(\RR)$ such that 
	\begin{itemize}
		\item $\chi''\leq 0$ on $[0, +\infty)$;
		\item $\chi(s) = s$ for $|s|\leq 1/2$, $\chi(s) = \sgn(s)$ for $|s|\geq \sqrt2$.
	\end{itemize}
	If we denote by 
	\begin{align}
		\odist_{n, k}(X) := \chi\left(|x|-\sqrt{2(n-k)} \right) \,, \label{Equ_L^2 Mono_odist to C_(n,k)}      
	\end{align}
	which is a cut-off and regularization of signed distance function to $\cC_{n,k}$, and \[
	\rmD_{n,k}(\tilde X, \tau) = \rmD_{n,k}(\tilde x,\tilde y, \tau) := F(e^{-\tau/2}\tilde X, -e^{-\tau}) = f_\circ\left(e^{-\tau}\cdot\frac{2(n-k)-|\tilde x|^2}{2(n-k)}\right) \,.
	\] 
	Then it's easy to check that for every $(\tilde X, \tau)\in \RR^{n+1}\times \RR_{\geq 0}$,
	\begin{align}
		C(n)^{-1}e^{-2\tau} \overline{\dist}_{n, k}(\tilde X)^2 \leq \rmD_{n,k}(\tilde X,\tau)^2 \leq C(n) \overline{\dist}_{n, k}(\tilde X)^2 \,.  \label{Equ_L^2 Mono_D_(n,k) approx dist_(n,k)}      
	\end{align}
	This lead to the following non-concentration near infinity for rescaled mean curvature flow. 
	\begin{Cor}\label{Cor_L^2 Mono_Weighted L^2 Mono and Est for RMCF}
		There exist dimensional constants $K_n, C_n>0$ with the following property. Let $\tau_0>0$, $\tau\mapsto \cM(\tau)$ be a rescaled mean curvature flow in $\RR^{n+1}$ over $[0, \tau_0]$. Then for every $0<\tau \leq \tau_0$, 
		\begin{align*}
			\int_{\RR^{n+1}} \overline{\dist}_{n,k}(X)^2 (1 + \tau|X|^2) e^{-\frac{|X|^2}{4}}\ d\cM(\tau) \leq C_n\, e^{K_n\tau}\int_{\RR^{n+1}} \overline{\dist}_{n,k}(X)^2 e^{-\frac{|X|^2}{4}}\ d\cM(0) \,.
		\end{align*}
	\end{Cor}
	\begin{proof}
		Let $\Omega, \bu, f_\circ, F, \rho$ be specified as above, $K=K(n)$ be specified in Theorem \ref{Thm_L^2 Mono_Main}. Let $t\mapsto \bM(t):= \sqrt{-t}\cM(-\log(-t))$ be the integral Brakke flow associated to $\cM$.  
		
		When $\eta\equiv 1$, \eqref{Equ_L^2 Mono_Main Est} becomes 
		\begin{align}
			\frac{d}{dt} \left[ (-t)^{2K}\int_{\RR^{n+1}} F^2\rho\ d\bM(t) \right] + (-t)^{2K}\int_{\RR^{n+1}} \frac{c(n)|X|^2}{t^2} F^2 \rho\ d\bM(t) \leq 0 \label{Equ_L^2 Mono_Model Eg with eta = 1}.
		\end{align}
		
		When $\eta(X, t) = \sqrt{1-|X|^2/t\, } \xi(t)$, where $\xi\in C^2(-1, 0)$ with $\Spt\xi$ a subinterval of $(-1, 0)$, \eqref{Equ_L^2 Mono_Main Est} implies,
		\begin{align}
			\begin{split}
				\frac{d}{dt} & \left[  (-t)^{2K}\int_{\RR^{n+1}} F^2\cdot(1-|X|^2/t)\,\xi(t)^2\rho\ d\bM(t) \right] \\ 
				\leq \; & (-t)^{2K}\int_{\RR^{n+1}} \left[\left(-\frac{|X|^2}{t^2} - \frac{2n+8}{t} \right)\xi^2 + \left(1-\frac{|X|^2}{t}\right)|\partial_t(\xi^2)| \right] F^2\rho\ d\bM(t).
			\end{split} \label{Equ_L^2 Mono_Model Eg with eta = |X|}
		\end{align}
		
		We now write everything under rescaled mean curvature flow $\tau\mapsto \cM(\tau) = e^{\tau/2}\bM(-e^{-\tau})$ parametrized by $(\tilde X, \tau)$ using the change of variable $(X, t) = (e^{-\tau/2}\tilde X, -e^{-\tau})$. Denote for simplicity $d\tilde\mu_\tau = \rho(X, -1)\ d\cM(\tau)$.
		Then \eqref{Equ_L^2 Mono_Model Eg with eta = 1} is equivalent to, 
		\begin{align}
			\begin{split}
				\frac{d}{d\tau} \left[e^{-2K\tau}\int_{\RR^{n+1}} \rmD_{n,k}(\tilde X,\tau)^2\ d\tilde\mu_\tau \right] 
				+ e^{-2K\tau}\int_{\RR^{n+1}} c(n)\rmD_{n,k}(\tilde X,\tau)^2|\tilde X|^2 \ d\tilde\mu_\tau \leq 0 \,.
			\end{split} \label{Equ_L^2 Mono_Mono under RMCF param w eta=1}
		\end{align}
		And let $\tilde{\xi}(\tau):= \xi(-e^{-\tau})$, then \eqref{Equ_L^2 Mono_Model Eg with eta = |X|} is equivalent to,
		\begin{align}
			\begin{split}
				\frac{d}{d\tau} & \left[ \, e^{-2K\tau}\int_{\RR^{n+1}} \rmD_{n,k}(\tilde X,\tau)^2(1+|\tilde X|^2)\tilde{\xi}(\tau)^2 \ d\tilde\mu_\tau \right] \\
				\leq & \; e^{-2K\tau}\int_{\RR^{n+1}} \rmD_{n,k}(\tilde X,\tau)^2\left[ (2n+8-|\tilde X|^2)\tilde{\xi}(\tau)^2 + (1+|\tilde X|^2)
				\left|(\tilde{\xi}(\tau)^2)'\right| \right] \ d\tilde\mu_\tau\,.
			\end{split} \label{Equ_L^2 Mono_Mono under RMCF param w eta=|X|}
		\end{align}  
		
		Integrate (\ref{Equ_L^2 Mono_Mono under RMCF param w eta=1}) and (\ref{Equ_L^2 Mono_Mono under RMCF param w eta=|X|}) over $[0,\tau]$, choose $\tilde{\xi}(s)=s/\tau$ and use (\ref{Equ_L^2 Mono_D_(n,k) approx dist_(n,k)}), we then complete the proof of the corollary. 
	\end{proof}
	
	\subsection{Decay order and asymptotic Rate} \label{Subsec_Doub Const, AR of RMCF}  
	
	We define the $L^2$-distance of a Radon measure $\Sigma$ on $\RR^{n+1}$ to the cylinder $\cC_{n,k}$ by,
	\begin{align}
		\mbfd_{n, k}(\Sigma)^2 := \int_{\RR^{n+1}}  \overline{\dist}_{n,k}(X)^2 e^{-\frac{|X|^2}{4}}\ d\Sigma\, .  \label{Equ_L^2 Mono_Def L^2 dist to C_(n,k)}
	\end{align} 
	When $M\subset \RR^{n+1}$ is a hypersurface, we denote for simplicity $\mbfd_{n,k}(M):= \mbfd_{n,k}(\|M\|)$, where $\|M\|$ is the volume measure of $M$.
	
	We call a rescaled mean curvature flow $\tau\mapsto  \cM(\tau)$ over interval $I$ \textbf{$\delta$-$L^2$ close to $\cC_{n,k}$} if $\forall\, s\in I$, 
	\begin{align}
		\cF[\cM(s)] \leq \frac32 \cF[\cC_{n,k}] \,, & &
		\mbfd_{n,k}(\cM(s)) \leq \delta \,, \label{Equ_L^2 Mono_delta L^2 close to cylinder C_(n,k)}
	\end{align}
	where $\cF$ denotes the Gaussian area functional.  By White's regularity \cite{White05_MCFReg}, if $I$ is compact and $\cM$ is $\delta$ close to $\cC_{n,k}$ in the Brakke sense, then $\cM$ is $\Psi(\delta|n, I)$-$L^2$ close to $\cC_{n,k}$, and vice versa. 
	
	Let $\tau\mapsto  \cM(\tau)$ be a rescaled mean curvature flow in $\RR^{n+1}$ over interval $I$. For $\tau, \tau+1\in I$, we define the \textbf{decay order}   of $\cM$ at time $\tau$ relative to $\cC_{n,k}$ by 
	\begin{align}
		\cN_{n,k}(\tau; \cM) := \log \left(\frac{\mbfd_{n,k}(\cM(\tau))}{\mbfd_{n,k}(\cM(\tau+1))} \right).  \label{Equ_L^2 Mono_Def Doubl Const}
	\end{align}
	As an example, if $\cM(\tau)$ is the volume measure of $\graph_{\cC_{n,k}}(u(\cdot, \tau))$, where \[
	u(X, \tau) = e^{-d\tau} w(X) + \text{ errors}\,,
	\]
	for some small $C^0$ function $w(X)$ on $\cC_{n,k}$.  Then heuristically we have, \[
	\cN_{n,k}(\tau; \cM) \approx d \,.  \]
	
	Also note that by Corollary \ref{Cor_L^2 Mono_Weighted L^2 Mono and Est for RMCF}, for every rescaled mean curvature flow $\cM$, we always have the dimensional lower bound 
	\begin{align}
		\cN_{n,k}(\tau; \cM) \geq -C_n \,. \label{Equ_L^2 Mono_Dim Lower Bd for cN}
	\end{align}
	We may omit the subscript $n,k$ if there's no confusion of which cylinder we are taking the distance relative to.
	
	For $R>0$, recall $Q_R:= \BB^{n-k+1}_R\times \BB_R^k$. We also define the decay order of a rescaled mean curvature flow $\tau\mapsto \cM(\tau)$ restricted in $Q_R$: 
	
	\begin{align*}
		\cN_{n,k}(\tau; \cM\llcorner Q_R) := \log \left(\frac{\mbfd_{n,k}(\cM(\tau)\llcorner Q_R)}{\mbfd_{n,k}(\cM(\tau+1)\llcorner Q_R)} \right).
	\end{align*}  
	An immediate consequence of Corollary \ref{Cor_L^2 Mono_Weighted L^2 Mono and Est for RMCF} is,
	\begin{Cor}\label{Cor_L^2 Mono_Compare cN(M) and cN(M in Q_R)}
		For every $\epsilon\in (0, 1)$, there exists $R_0(\epsilon, n)\gg1$ such that the following hold.  If $\cM$ is a rescaled mean curvature flow in $\RR^{n+1}$ over $[0, 2]$ such that 
		\begin{align*}
			\cN_{n,k}(0, \cM), \;\; \cN_{n,k}(1, \cM) \leq \epsilon^{-1}\,.
		\end{align*}
		Then for every $\tau \in (0, 1]$ and every $R\geq \tau^{-1/2}R_0$, we have $\cM(\tau)\llcorner Q_R,\ \cM(\tau+1)\llcorner Q_R\neq 0$ and \[
		\left| \cN_{n,k}(\tau, \cM\llcorner Q_R) - \cN_{n,k}(\tau, \cM) \right| \leq \frac{R_0}{\tau R^2} \,.
		\]
	\end{Cor}
	\begin{proof} %In the following, $C(n,\epsilon)$ may vary line to line.
		By Corollary \ref{Cor_L^2 Mono_Weighted L^2 Mono and Est for RMCF}, for every $\tau\in (0, 2]$ and $R>0$, 
		\begin{align*}
			0\leq 1-\frac{\mbfd_{n,k}(\cM(\tau)\llcorner Q_R)}{\mbfd_{n,k}(\cM(\tau))} \leq \frac{C(n)}{\tau R^2}\cdot \frac{\mbfd_{n,k}(\cM(0))}{\mbfd_{n,k}(\cM(\tau))} \leq \frac{C(n,\epsilon)}{\tau R^2}\cdot \frac{\mbfd_{n,k}(\cM(2))}{\mbfd_{n,k}(\cM(\tau))} \leq \frac{C(n,\epsilon)}{\tau R^2}.
		\end{align*}
		Therefore, when $\tau R^2\geq R_0(n,\epsilon)^2\gg1$, $\cM(\tau)\llcorner Q_R\neq 0$ and \[
		\left| e^{\cN(\tau; \cM\llcorner Q_R) - \cN(\tau; \cM)} - 1 \right| = \left| \frac{\mbfd_{n,k}(\cM(\tau)\llcorner Q_R)}{\mbfd_{n,k}(\cM(\tau))} \cdot \left(\frac{\mbfd_{n,k}(\cM(\tau+1)\llcorner Q_R)}{\mbfd_{n,k}(\cM(\tau+1))}\right)^{-1} - 1\right| \leq \frac{C(n, \epsilon)}{\tau R^2}\,.
		\]
		Then \[
		|\cN_{n,k}(\tau; \cM\llcorner Q_R) - \cN_{n,k}(\tau; \cM)| \leq \frac{C(n, \epsilon)}{\tau R^2} \,.
		\]
	\end{proof}
	
	Another application of Corollary $\ref{Cor_L^2 Mono_Weighted L^2 Mono and Est for RMCF}$ is that, the decay order upper bound allows us to take normalized limit of graphical function of rescaled mean curvature flow over round cylinders.
	\begin{Lem} \label{Lem_L^2 Mono_Induced Parab Jac}
		Let $\cM_j$ be a sequence of rescaled mean curvature flow in $\RR^{n+1}$ over $[0, T]$ converging to $\cC_{n,k}$ in the Brakke sense, where $T\geq 1$. Suppose \[
		\limsup_{j\to \infty} \cN_{n,k}(0, \cM_j) <+\infty \,.
		\]
		Let $u_j(\cdot, \tau)$ be the graphical function of $\cM_j(\tau)$ over $\cC_{n,k}$, defined on a larger and larger domain as $j\to \infty$\footnote{Recall that as defined in Section \ref{SS:PreCylinder}, $u_j$ is set to be zero outside the graphical domain}. Then after passing to a subsequence, $\hat{u}_j:= \mbfd_{n,k}(\cM_j(1))^{-1}u_j$ converges to some non-zero $\hat{u}$ in $C^\infty_{loc}(\cC_{n,k}\times (0, T])$ solving \[
		\partial_\tau \hat u - L_{n,k} \hat u = 0\,.
		\]
		Moreover, there exists $\bar c_{n,k}>0$ such that for every $\tau\in (0, T]$, we have \[
		\|\hat{u}(\cdot, \tau)\|_{L^2(\cC_{n,k})} = \bar c_{n,k}\lim_{j\to \infty} \mbfd_{n,k}(\cM_j(1))^{-1}\mbfd_{n,k}(\cM_j(\tau)) < +\infty\,.     \]
	\end{Lem}
	We shall call such non-zero renormalized limit $\hat{u}$ a \textbf{induced (parabolic) Jacobi field} from the sequence $\{\cM_j\}$.
	\begin{proof}
		First note that by Brakke-White regularity \cite{White05_MCFReg} of mean curvature flow and interior parabolic estimate, $u_j$ is defined on a larger and larger domain exhausting $\cC_{n,k}$ and $u_j\to 0$ in $C^\infty_{loc}(\cC_{n,k}\times (0, T])$ as $j\to \infty$.  
		By Lemma \ref{Lem_App_Graph over Cylinder} (i), (iii), Corollary \ref{Cor_L^2 Mono_Weighted L^2 Mono and Est for RMCF} and the upper bound of decay order (denoted by $\epsilon^{-1}$), there exists $R_1(n,\epsilon)\gg1$ such that for every $\tau\in (0, T]$, $R\geq \tau^{-1}R_1$ and $j\gg1$, we have \[
		\|u_j(\tau, \cdot)\|_{L^2(\cC_{n,k}\setminus Q_R)} + \mbfd_{n,k}(\cM_j(\tau)\llcorner (Q_R)^c) \leq \frac{C(n,\epsilon)}{\tau R^2}\cdot \mbfd_{n,k}(\cM_j(1)\llcorner Q_R).
		\]
		This is a uniform $L^2$ non-concentration property near infinity. Therefore, combined with Lemma \ref{Lem_App_Graph over Cylinder}, \ref{Lem_App_RMCF equ} and the classical parabolic regularity estimates, $\hat{u}_j:= \mbfd_{n,k}(\cM_j(1))^{-1}\, u_j$ subconverges to some non-zero $\hat{u}\in C^\infty(\cC_{n,k}\times (0, T])$, and such that for every $\tau\in (0, T]$, \[
		\|\hat{u}(\cdot,\tau)\|_{L^2(\cC_{n,k})} = \bar c_{n,k} \lim_{j\to \infty} \mbfd_{n,k}(\cM_j(1))^{-1}\cdot \mbfd_{n,k}(\cM_j(\tau)) <+\infty \,, 
		\] 
		here $\bar c_{n,k}$ is determined by Lemma \ref{Lem_App_Graph over Cylinder} (iii).
	\end{proof}
	
	In the following, for $\sim \in \{\geq, >, =, <, \leq\}$ and $\gamma\in \RR$, we denote by 
	\begin{align}
		\Pi_{\sim \gamma}: L^2(\cC_{n,k})\to L^2(\cC_{n,k})  \label{Equ_Pre_Proj onto sum of eigenspace}      
	\end{align}
	to be the orthogonal projection onto the direct sum of eigensubspaces of $-L_{n,k}$ with eigenvalue $\sim \gamma$. Also recall $\sigma(\cC_{n,k})$ is defined in \eqref{Equ_Pre_sigma(C_n,k)}.
	Note that when $\gamma<-1$, $\Pi_{\leq \gamma} = 0$; and when $\gamma\notin \sigma(\cC_{n,k})$, $\Pi_{=\gamma} = 0$. 
	\begin{Cor} \label{Cor_L^2 Mono_RMCF w graphical eigenfunc has cN = spectrum}
		For every $\epsilon\in (0, 1/2)$, there exists a $\delta_1(n,\epsilon)\in (0, \epsilon)$ with the following significance. 
		Let $\cM$ be a rescaled mean curvature flow in $\RR^{n+1}$ $\delta_1$-close to $\cC_{n,k}$ over $[0, T]$ with $1\leq T \leq \epsilon^{-1}$ such that \[
		\cN_{n,k}(0, \cM) \leq \epsilon^{-1}\,.
		\] 
		Let $\gamma\in [-\epsilon^{-1}, \epsilon^{-1}]$, $\tau_0\in [\epsilon, T]$, $\sim\in \{\geq, =, \leq\}$, and $u$ be the graphical function of $\cM$ over $\cC_{n,k}$. Also suppose that \[
		\|\Pi_{\sim\gamma}(u(\cdot, \tau_0))\|_{L^2} \geq (1-\delta_1)\|u(\cdot, \tau_0)\|_{L^2} \,.  \]
		Then for every $\tau\in [\epsilon, T]$, 
		\begin{align*}
			\cN_{n,k}(\tau; \cM) - \gamma \begin{cases}
				\leq  \epsilon , & \text{ if }\sim \text{ is }\leq\,; \\
				\geq -\epsilon, & \text{ if }\sim \text{ is }\geq\,; \\
				\in [-\epsilon, \epsilon], & \text{ if }\sim \text{ is }=\,.  \\
			\end{cases}
		\end{align*}
	\end{Cor}
	\begin{proof}
		The corollary follows by a direct contradiction argument combing Lemma \ref{Lem_L^2 Mono_Induced Parab Jac} and \ref{Lem_App_Analysis of Parab Jacob field} (ii).
	\end{proof}

	Conversely, an analogue of Lemma \ref{Lem_App_Analysis of Parab Jacob field} (iii) holds for the nonlinear decay order when $\cM$ is sufficiently close to $\cC_{n,k}$.
	\begin{Cor}[Discrete Monotonicity of the decay order] \label{Cor_L^2 Mono_Discrete Growth Mono}
		For every $\epsilon\in (0, 1/2)$, there exists $\delta_2(\epsilon, n)\in (0, \epsilon)$ such that the following hold.  If $\cM$ is a rescaled mean curvature flow in $\RR^{n+1}$ over $[0, 2]$ $\delta_2$-$L^2$ close to $\cC_{n,k}$, and satisfy the decay order bound, 
		\begin{align}
			\cN_{n,k}(0; \cM) \leq \epsilon^{-1} \,. \label{Equ_L^2 Mono_Entropy and Doubl Const Bd}
		\end{align}
		Then at least one of the following holds,
		\begin{align}
			&\text{either}  & -1-\epsilon \leq \cN_{n,k}(1; \cM) & \leq \cN_{n,k}(0; \cM) - \delta_2\,;  \label{Equ_L^2 Mono_Strict Doubl Const drop}  \\
			&\text{or}  & \sup_{\tau\in [\epsilon, 1]}|\cN_{n,k}(\tau; \cM) - \gamma| & \leq \epsilon\,, \;\;\;\;\; \text{ for some }\gamma\in \sigma(\cC_{n,k})\,.  \label{Equ_L^2 Mono_Doubl Const close to spectrum}
		\end{align}
		Moreover, if (\ref{Equ_L^2 Mono_Strict Doubl Const drop}) fails and $\gamma$ is given by (\ref{Equ_L^2 Mono_Doubl Const close to spectrum}), then the graphical function $u$ of $\cM$ over $\cC_{n,k}$ satisfies 
		\begin{align}
			\|\Pi_{=\gamma} (u(\cdot, \tau))\|_{L^2} \geq (1-\epsilon)\|u(\cdot, \tau)\|_{L^2}\,,  \quad \forall\, \tau\in [\epsilon, 2] \,. \label{Equ_L^2 Mono_Jac field concentra in eigenspace}        
		\end{align}
	\end{Cor}
	\begin{Rem} \label{Rem_L^2 Mono_cN<gamma pass to next scale}
		A useful consequence of at least one of (\ref{Equ_L^2 Mono_Strict Doubl Const drop}) and (\ref{Equ_L^2 Mono_Doubl Const close to spectrum}) being true is that, if $\gamma\in \RR$ with $\dist_\RR(\gamma, \sigma(\cC_{n,k}))\geq \epsilon$, then $\cN_{n,k}(0; \cM)\leq \gamma$ implies $\cN_{n,k}(1; \cM)\leq \gamma$. 
	\end{Rem}
	\begin{proof}
		In view of Corollary \ref{Cor_L^2 Mono_Weighted L^2 Mono and Est for RMCF}, it suffices to show that for every $\epsilon>0$ and sequence of rescaled mean curvature flow $\cM_j=\{M_j(\tau)\}_{\tau\in [0,2]}$ satisfying (\ref{Equ_L^2 Mono_Entropy and Doubl Const Bd}) and converging to the multiplicity $1$ static flow $\cC_{n,k}$ in the Brakke sense as $j\to \infty$, if \eqref{Equ_L^2 Mono_Strict Doubl Const drop} fails, i.e.  %and $\tau_j\in [\epsilon, 1]$ such that 
		\begin{align*}
			\text{ either }\;\; \cN_{n,k}(1, \cM_j) \geq \cN_{n,k}(0, \cM_j) - \frac1{j} \,, & &
			\text{ or }\;\; \cN_{n,k}(1; \cM_j) < -1-\epsilon = \inf \sigma(\cC_{n,k}) - \epsilon \,. 
		\end{align*}
		Then there exists $\gamma\in \sigma(\cC_{n,k})$ such that
		\begin{align*}
			\limsup_{j\to \infty} \sup_{\tau\in [\epsilon, 1]} |\cN_{n,k}(\tau, \cM_j)-\gamma| = 0 \,. %\label{Equ_L^2 Mono_In pf_Doubl Const close to spectrum}
		\end{align*}
		and that the induced Jacobi fields from $\cM_j$ exist and are all given by $e^{-\gamma\tau}w$ for some $\gamma$-eigenfunction $w$ of $-L_{n,k}$.
		These follow directly from Lemma \ref{Lem_L^2 Mono_Induced Parab Jac} and \ref{Lem_App_Analysis of Parab Jacob field} (iii).
	\end{proof}
	
	\begin{remark}
		Using this discrete monotonicity, it's not hard to show that for every rescaled mean curvature flow $\tau\mapsto \cM(\tau)$ over $(0, +\infty)$ in $\RR^{n+1}$ with finite entropy such that $\cM(\tau)$ $C^\infty_{loc}$-converges to $\cC_{n,k}$ when $\tau\to +\infty$, the following limit exists \[
		\lim_{\tau\to +\infty} \cN_{n,k}(\tau, \cM) \; \in \left(\sigma(\cC_{n,k})\cap \RR_{\geq 0}\right) \cup \{+\infty\}\,.
		\] 
		However, since this fact is not used in the current paper, we shall not dive into its proof here, but postpone it to some slightly generalized statement in a subsequent manuscript. Instead, we would like to mention the following example:
	\end{remark}
	
	\begin{example}\label{Eg_Nondegen sing has cN = 0}
		Let $\bM$ be a Brakke flow in $\RR^{n+1}$ over some interval $I\ni 0$ of finite entropy, with a non-degenerate (or more generally, not-fully-degenerate) singularity at $(\orig, 0)$ modeled on $\cC_{n,k}$. Let $\tau\mapsto \cM(\tau)$ be the associated rescaled mean curvature flow at $(\orig, 0)$. Then 
		\begin{align}
			\lim_{\tau\to +\infty} \cN_{n,k}(\tau; \cM) = 0 \,. \label{Equ_Eg_Nondegen sing has cN = 0}        
		\end{align}
		
		To see this, let $u(\cdot, \tau)$ be the graphical function of $\cM(\tau)$ over $\cC_{n,k}$. By definition of non-degeneracy, \eqref{Equ_Pre_C^1 asymp in B_sqrt(tau)}, \eqref{Equ_Pre_H^1 asymp in B_sqrt(tau)} and Lemma \ref{Lem_App_Graph over Cylinder}, we have,
		\begin{align*}
			\mbfd_{n,k}(\cM(\tau))^2 & = c_{n,k}(1+o(1))\|u(\cdot, \tau)\|_{L^2(\cC_{n,k}\setminus Q_{\sqrt{\tau}})}^2 + \mbfd_{n,k}\left(\cM(\tau)\llcorner (\RR^{n+1}\setminus Q_{\sqrt{\tau}})\right)^2 \\
			& = c_{n,k}(1+o(1))\tau^{-2} + O(e^{-\tau/4})\cdot\lambda[\bM] \,.
		\end{align*}
		\eqref{Equ_Eg_Nondegen sing has cN = 0} then follows directly from this and definition. 
	\end{example}

	In later applications, it is also convenient if we can bound the decay order of a slightly translated flow. Recall that if $p_\circ:= (X_\circ, t_\circ)\in \RR^{n+1}\times \RR$, then $\cM^{p_\circ} = \tau \mapsto \cM^{p_\circ}(\tau)$ is the rescaled mean curvature flow given by \eqref{eq:RMCF change of base point}. We first derive a general bound on $L^2$ distance of a translation and dilation of hypersurface.
	\begin{Lem} \label{Lem_L^2 Mono_Bd d(aS+X) by d(S)}
		Let $\Sigma$ be a Radon measure on $\RR^{n+1}$ with finite $n$-dimensional entropy, $a>0$, $X_\circ=(x_\circ, y_\circ)\in \RR^{n-k+1}\times \RR^k$, $R\geq 0$, $a^\star:= \max\{a, 1\}$. Then 
		\begin{align}
			\begin{split}
				\mbfd_{n,k}\big(a\cdot(\Sigma+X_\circ)\big)^2 
				& \leq a^n e^{(2a^2|X_\circ|R+|a^2-1|R^2)/4} \cdot\big( a^\star\, \mbfd_{n,k}(\Sigma) + C_n(a|X_\circ|+|a-1|)\big)^2 \\
				& + C_n a^2 e^{-a^2(R^2-4|X_\circ|^2)/16}(1+|X_\circ|^2)\cdot \lambda[\Sigma]\,. 
			\end{split} \label{Equ_L^2 Mono_Bd d(aS+X) by d(S)}
		\end{align}
		In particular, for every $\eps\in (0, 1)$, there exists $\delta_3(n, \eps)\in (0, 1)$ such that if 
		\begin{align*}
			\lambda[\Sigma] \leq \eps^{-1}\,, & &
			\mbfd_{n,k}(\Sigma)\leq 1\,, & &
			|a-1| + |X_\circ| \leq \epsilon\, \mbfd_{n,k}(\Sigma) \leq \delta\mbfd_{n,k}(\Sigma)\,,        
		\end{align*}
		for some $\delta\in (0, \delta_3)$, then 
		\begin{align}
			\left| \frac{\mbfd_{n,k}(a\cdot (\Sigma+X_\circ))}{\mbfd_{n,k}(\Sigma)} - 1\right| \leq C(n, \eps)\sqrt{\delta}\,. \label{Equ_L^2 Mono_d(aS+X) approx d(S) when |a-1|, |X| ll 1}
		\end{align}
	\end{Lem}
	\begin{proof}
		By taking $R=\delta^{-1/4}\mbfd_{n,k}(\Sigma)^{-1/2}$, it's easy to see that (\ref{Equ_L^2 Mono_Bd d(aS+X) by d(S)}) implies (\ref{Equ_L^2 Mono_d(aS+X) approx d(S) when |a-1|, |X| ll 1}). 
		
		To prove (\ref{Equ_L^2 Mono_Bd d(aS+X) by d(S)}), notice that for every $X\in \RR^{n+1}$,
		\begin{align*}
			|X+X_\circ|^2 & \geq \frac{|X|^2+R^2}{4} - |X_\circ|^2\,, & & \text{ if }|X|\geq R\,; \\
			a^2|X+X_\circ|^2 & \geq |X|^2 - 2a^2 |X_\circ|R-|a^2-1|R^2\,, & & \text{ if } |X|\leq R\,; 
		\end{align*}
		Also, recall the definition of $\odist$ in \eqref{Equ_L^2 Mono_odist to C_(n,k)} and the specified $\chi$ above it, since $\chi$ is non-decreasing and concave on $[0, +\infty)$, we have for every $s, t\in \RR$, 
		\begin{align*}
			\chi(|s|) \leq |s|\,, & & \chi(|s+t|) \leq \chi(|s|) + \chi(|t|)\,, & & \chi(a|s|) \leq a^\star\chi(|s|)\,.
		\end{align*}
		Therefore its easy to check that 
		\begin{align*}
			|\odist_{n,k}(a(X+X_\circ))| & \leq a^\star\cdot |\odist_{n,k}(X)| + a|X_\circ| + \sqrt{2n}|a-1|\,.         
		\end{align*}

		Now let $B_1:=(2a^2|X_\circ|R+|a^2-1|R^2)/4$, $B_2:= a^2(R^2-4|X_\circ|)/16$, we get,
		\begin{align*}
			\mbfd_{n,k}\big(a\cdot(\Sigma+X_\circ)\big)^2 
			& = \int_{\RR^{n+1}} \odist_{n,k}(a(X + X_\circ))^2 \cdot a^n e^{-a^2|X+X_\circ|^2/4}\ d\Sigma \\
			& \leq \int_{\BB_R} \left( a^\star\cdot |\odist_{n,k}(X)| + a|X_\circ| + \sqrt{2n}|a-1| \right)^2 \cdot a^n e^{-|X|^2/4 + B_1}\ d\Sigma \\
			& \; + \int_{\RR^{n+1}\setminus \BB_R} C_n a^2(1+|X_\circ|)^2\cdot a^n e^{-a^2|X|^2/16 - B_2} \ d\Sigma \\
			& \leq a^n e^{B_1}\big( a^\star\, \mbfd_{n,k}(\Sigma) + C_n(a|X_\circ|+|a-1|)\big)^2  \\
			& \; + C_n a^2 e^{-B_2}(1+|X_\circ|)^2 \int_{\RR^{n+1}}  e^{-|X|^2/4} \ d(\frac{a}{2}\cdot\Sigma) \,.
		\end{align*}
	\end{proof}
	
	\begin{Cor} \label{Cor_L^2 Mono_cN(cM^X) bd}
		For every $\epsilon\in (0, 1/2)$, there exists $\delta_4(n, \epsilon)\in (0, \epsilon)$ with the following significance.
		Let $\tau_\circ\in \RR$, $I\supset [\tau_\circ, \tau_\circ + 2]$ be an interval, $\tau\mapsto \cM(\tau)$ be a rescaled mean curvature flow in $\RR^{n+1}$ $\delta_4$-$L^2$ close to $\cC_{n,k}$ over $I$. Suppose 
		\begin{align}
			\lambda[\cM] \leq \epsilon^{-1}\,, & &
			\sup_{\tau_\circ \leq \tau\leq \tau_\circ + 1}\cN_{n,k}(\tau; \cM) \leq \epsilon^{-1}\,;  \label{Equ_L^2 Mono_Assump I on off-center RMCF}
		\end{align}
		And $\tau \in I_{\geq \tau_\circ + 1}$, $p_\circ:= (X_\circ, t_\circ)\in \RR^{n+2}$ so that,
		\begin{align}
			t_\circ e^\tau<1\,, & & 
			e^{\tau_\circ}|t_\circ| + e^{\tau_\circ/2}|X_\circ| \leq \delta_4\cdot \mbfd_{n,k}(\cM(\tau_\circ))\,, & &
			\mbfd_{n,k}(\cM^{p_\circ}(\tau)) \leq \delta_4\,. \label{Equ_L^2 Mono_Assump II on off-center RMCF}
		\end{align}
		Then we have \[
		\cN_{n,k}(\tau, \cM^{p_\circ}) \leq \sup_{\tau_\circ \leq \tau'\leq \tau_\circ+1}\cN_{n,k}(\tau'; \cM) + C(n, \epsilon)\,.
		\]
	\end{Cor}
	\begin{proof}
		By a time translation, WLOG $\tau_\circ = 0$. First note that by \eqref{Equ_L^2 Mono_Assump II on off-center RMCF} and a compactness argument, $\cM^{p_\circ}$ is $\Psi(\delta_4|n, \epsilon)$ close to $\cC_{n,k}$ in the Brakke sense on $[0, 2]$ and $[\tau, \tau+2]$. Then by \cite[Section 6]{CM15_Lojasiewicz}, when $\delta_4(n,\epsilon)\ll 1$, $\cM^{p_\circ}$ is $\Psi(\delta_4|n, \epsilon)$-$L^2$ close to $\cC_{n,k}$ on the whole interval $[0, \tau+2]$. Hence, to prove the corollary, by Corollary \ref{Cor_L^2 Mono_Discrete Growth Mono} and Remark \ref{Rem_L^2 Mono_cN<gamma pass to next scale}, it suffices to show that when $\delta(n,\epsilon)\ll 1$, for every $\tau'\in [1, 2]$, 
		\begin{align}
			\cN_{n,k}(\tau', \cM^{p_\circ}) \leq \sup_{\tau''\in [0, 2]}\cN_{n,k}(\tau''; \cM) + C(n, \epsilon)\,.  \label{Equ_L^2 Mono_cN(0, cM^p)<cN(0, cM)} 
		\end{align}
		While this follows directly from Corollary \ref{Cor_L^2 Mono_Weighted L^2 Mono and Est for RMCF} and Lemma \ref{Lem_L^2 Mono_Bd d(aS+X) by d(S)}.
		
	\end{proof}

	\subsection{Preservation of $h_1$ domination.}
	The goal of this subsection is to prove that, for every unit vector $\by\in \RR^k$ and a rescaled mean curvature flow $\cM$ over interval $[a, b]$ sufficiently close to $\cC_{n,k}$ with graphical function $u$ over $\cC_{n,k}$, modulo constant mode, if linear mode $y\cdot\by$ dominates $u(\cdot, a)$, then so it does for $u(\cdot, b)$. 
	In particular, neither is there an extra spherical mode suddenly appearing, nor is the direction $\by$ of the dominated linear mode changing much along the flow, no matter how long the time interval $[a, b]$ is. More precisely, we shall prove the following.
	\begin{Lem} \label{Lem_L^2 Mono_h_1 dominate in all scale}
		For every $\eps\in (0, 1/4)$, there exists $\delta_5(n, \eps)\in (0, \eps)$ and $R_1(n, \eps)\gg 1$ with the following significance.
		
		Let $a+2< b$, $\cM$ be a rescaled mean curvature flow $\delta_5$-$L^2$ close to $\cC_{n,k}$ over $[a-1, b+2]$, $\by\in \spine(\cC_{n,k}) =\RR^k$ be a unit vector. For every $\tau\in [a, b+1]$, let $u(\cdot, \tau)$ be the graphical function of $\cM(\tau)$ over $\cC_{n,k}\cap Q_{R}$, where $R\geq R_1(n, \eps)$, and $0$-extend it to an $L^\infty$ function on $\cC_{n,k}$. Suppose, 
		\begin{align}
			\lambda[\cM] \leq \eps^{-1}\,, & &
			\sup_{|s|\leq \eps e^{-b}} \cN_{n,k}(-\log(e^{1-a}+s); \cM^{(\orig, s)}) \leq \eps^{-1}\,; \label{Equ_L^2 Mono_Time transl doubl const upper bd}
		\end{align}
		\begin{align}
			\inf_{c>0, c'\in \RR}\|c^{-1}u(\cdot, a) - c' - y\cdot\by\|_{L^2} \leq \delta_5\,, \label{Equ_L^2 Mono_h_1 dominate graphical func of cM} 
		\end{align}
		Then we have 
		\begin{align}
			\inf_{c>0, c'\in \RR}\|c^{-1}u(\cdot, b) - c' - y\cdot\by\|_{L^2} \leq \eps\,. \label{Equ_L^2 Mono_h_1 dominate conclusion}   
		\end{align}
	\end{Lem}
	
	Before diving into its proof, we need the following lemma which helps to modulo the effect of constant mode.  
	\begin{Lem}\label{Lem_L^2 Mono_Time transl to make cN > -1/2}
		For every $\eps\in (0, 1)$, there exists $\delta_6(n, \eps)\in (0, \eps)$ with the following significance.
		
		Let $a\in \RR$, $\cM$ be a rescaled mean curvature flow $\delta_6$-$L^2$ close to $\cC_{n,k}$ on $[a-1, a+2]$ with 
		\begin{align*}
			\lambda[\cM]\leq \eps^{-1}\,, & & 
			\sup_{|s|\leq \eps e^{-a}} \cN_{n,k}(-\log(e^{1-a}+s); \cM^{(\orig, s)})\leq \eps^{-1}\,.
		\end{align*}
		Then there exists $s_\circ\in [-\eps e^{-a}, \eps e^{-a}]$ such that \[
		\cN_{n,k}(-\log (e^{-a}+s_\circ); \cM^{(\orig, s_\circ)}) \geq -\frac12 - \eps\,.
		\]
	\end{Lem}
	\begin{proof}
		By a time translation, it suffices to prove the case when $a=0$. Suppose for contradiction that there exists a sequence of rescaled mean curvature flow $\tau\mapsto\cM_j(\tau)$ over $[-1, 2]$ converging to $\cC_{n,k}$ in the Brakke sense such that 
		\begin{align*}
			\sup_{|s|\leq\eps}\cN_{n,k}(-\log (e+s); \cM_j^{(\orig, s)}) \leq \eps^{-1}\,, & &
			\sup_{|s|\leq\eps}\cN_{n,k}(-\log (1+s); \cM_j^{(\orig, s)}) < -\frac12 - \eps\,.        
		\end{align*}
		We recall that $\cM_j^{(\orig, s)}(\tau) = \sqrt{1-se^{\tau}}\cdot \cM_j(- \log(e^{-\tau} - s))$. Let $s_j\in [-\eps, \eps]$ be a minimizer of 
		\begin{align*}
			s\mapsto \mbfd_{n,k}(\cM_j^{(\orig, s)}(-\log(e^{-1}+s)))^2 & = \int_{\RR^{n+1}} \odist^2_{n,k}(X)\ e^{-\frac{|X|^2}{4}} d(\lambda(s)\cdot \cM_j(1)) \\
			& = \int_{\RR^{n+1}} \odist^2_{n,k}(\lambda(s)X')\ \lambda(s)^n e^{-\frac{\lambda(s)^2|X'|^2}{4}}\ d(M_j(1))\,,
		\end{align*}
		where $\lambda(s):= \sqrt{1-s(e^{-1}+s)^{-1}}$. Note that since $\cM_j$ is approaching $\cC_{n, k}$, $s_j$ should also tend to $0$ and hence is attained in the open interval when $j\gg1$. Thus by taking derivative in $s$ and recalling the definition of $\odist_{n,k}$ in \eqref{Equ_L^2 Mono_odist to C_(n,k)}, we find, 
		\begin{align}
			\begin{split}
				0 & = \int_{M_j(1)} \left[ 2\dstar_{n,k}(\lambda_j|x'|)|x'| + \odist_{n,k}(\lambda_jX')^2(\frac{n}{\lambda_j}-\frac{\lambda_j|x'|^2}{2}) \right]\ \lambda_j^n e^{-\frac{\lambda_j^2|X'|^2}{4}}\ dX' \\
				& = \int_{\cM_j'(-\log(e^{-1}+s_j))} \left[ 2\dstar_{n,k}(|x|)\frac{|x|}{\lambda_j} + \odist_{n,k}(X)^2(\frac{n}{\lambda_j}-\frac{|x|^2}{2\lambda_j}) \right]\  e^{-\frac{|X|^2}{4}}\ dX \,,
			\end{split} \label{Equ_L^2 Mono_scal to cancel base spectrum}
		\end{align}
		where we denote for simplicity $2\dstar_{n,k}(a):= (\chi^2)'(a-\sqrt{2(n-k)})$, $\lambda_j = \lambda(s_j)$, $\cM_j':= \cM_j^{(\orig, s_j)}$. 
		
		By the contradiction assumption, 
		\begin{align*}
			\cN_{n,k}(-\log (e+s_j); \cM_j') \leq \eps^{-1}\,, & &
			\cN_{n,k}(-\log (1+s_j); \cM_j') < -\frac12 - \eps\,.        
		\end{align*}
		By Lemma \ref{Lem_L^2 Mono_Induced Parab Jac}, $\{\cM_j'\}$ induces a nonzero Jacobi field $\hat{v}_\infty\in C^\infty_{loc}(\cC_{n,k}\times (-1, 2])$ with finite $L^2$ norm on each time slice, and satisfies
		\begin{align*}
			(\partial_\tau - L_{n,k})\hat{v}_\infty = 0\,, & &
			\log\left( \frac{\|\hat{v}_\infty(\cdot, 0)\|_{L^2}}{\|\hat{v}_\infty(\cdot, 1)\|_{L^2}}\right) \leq -\frac12-\eps \,.
		\end{align*}
		
		On the other hand, notice that by the choice of $\chi$, when $||x|-\sqrt{2(n-k)}|\leq 1/2$, $X\mapsto \dstar_{n,k}(|x|)$ is the signed distance function to $\cC_{n,k}$, and hence coincides with the graphical function of $\cM_j'(1)$ over $\cC_{n,k}$ after restricting to $\cM_j'(1)$ and projecting to $\cC_{n,k}$ in a larger and larger domain as $j\to \infty$. Hence by (\ref{Equ_L^2 Mono_scal to cancel base spectrum}) and the nonconcentration near infinity, we have \[
		\int_{\cC_{n,k}} \hat{v}_\infty(X)\ e^{-\frac{|X|^2}{4}}\ dX = 0 \,.
		\]
		By Lemma \ref{Lem_App_Analysis of Parab Jacob field} (ii), \[
		\log\left( \frac{\|\hat{v}_\infty(\cdot, 0)\|_{L^2}}{\|\hat{v}_\infty(\cdot, 1)\|_{L^2}}\right) = \mathcal N_{n,k}(0; \hat v_\infty) \geq -\frac12 \,.
		\]
		This is a contradiction.
	\end{proof}
	%  \[ \]
	
	\begin{proof}[Proof of Lemma \ref{Lem_L^2 Mono_h_1 dominate in all scale}.]
		Suppose for contradiction there exist $\eps\in (0, 1/4)$, $a_j+2<b_j$, unit vectors $\by_j\in \RR^k$ and a sequence of rescaled mean curvature flow $\cM_j$ over $[a_j-1, b_j+2]$ $1/j$-$L^2$ close to $\cC_{n,k}$, with graphical function $u_j(\cdot, \tau)$ defined over $\cC_{n,k}\cap Q_{R_j}$ for some $R_j\to +\infty$ (and zero extended outside $Q_{R_j}$), satisfying \eqref{Equ_L^2 Mono_Time transl doubl const upper bd} and \eqref{Equ_L^2 Mono_h_1 dominate graphical func of cM} with $(a_j, b_j, \by_j, \cM_j, u_j, 1/j)$ in place of $(a, b, \by, \cM, u, \delta_5)$.
		But 
		\begin{align}
			\inf_{c>0, c'\in \RR}\|c^{-1}u_j(\cdot, b_j) - c' - y\cdot\by_j\|_{L^2} > \eps\,, \quad \forall\, j\geq 1\,.   \label{Equ_L^2 Mono_In Pf_Contrad assump}
		\end{align}
		
		By Lemma \ref{Lem_L^2 Mono_Time transl to make cN > -1/2}, there exists $s_j\to 0$ such that  
		\begin{align}
			\liminf_{j\to \infty} \cN_{n,k}(\bar{b}_j, \bar{\cM}_j)\geq -\frac12 \,, \label{Equ_L^2 Mono_In Pf_bar(cM) doubl at b > -1/2}
		\end{align}
		where we let $\bar{\cM}_j:= \cM_j^{(\orig, s_j)}$, $\bar{a}_j':= 1-\log(e^{1-a_j}+s_j)$, $\bar{b}_j:= -\log(e^{-b_j}+s_j)$. Note that $\bar{\cM}_j$ is also $o(1)$-$L^2$ close to $\cC_{n,k}$ over $[\bar{a}_j'-1, \bar{b}_j]$.  By \eqref{Equ_L^2 Mono_Time transl doubl const upper bd} and Corollary \ref{Cor_L^2 Mono_Discrete Growth Mono}, when $j\gg 1$,
		\begin{align}
			\cN_{n,k}(\bar{a}_j'-1, \bar{\cM}_j),\;\; \cN_{n,k}(\bar{a}'_j, \bar{\cM}_j) \leq \eps^{-1}+1\,.  \label{Equ_L^2 Mono_In Pf_bar(cM) doubl upper bd}
		\end{align}
		While by \eqref{Equ_L^2 Mono_h_1 dominate graphical func of cM} and Lemma \ref{Lem_App_Graph over Cylinder} (iv), the graphical function $\bar{u}_j(\cdot, \tau)$ of $\bar{\cM}_j(\tau)$ over $\cC_{n,k}$ satisfies 
		\begin{align}
			\lim_{j\to \infty}\|\bar{u}_j\|_{L^2}^{-1}\cdot\|\bar{u}_j(\cdot, \bar{a}_j) - \bar{c}_j' - \bar{c}_jy\cdot \by_j\|_{L^2} = 0 \,.  \label{Equ_L^2 Mono_In Pf_h_1 and h_2 dominate graphical func of bar(cM)} 
		\end{align}
		for some $\bar{c}_j'\in \RR, \bar{c}_j>0$. where $\bar{a}_j:= -\log(e^{-a_j}+s_j) \in [\bar{a}'_j-1/4, \bar{a}'_j+1/4]$ when $j\gg 1$. Then by \eqref{Equ_L^2 Mono_In Pf_bar(cM) doubl upper bd} and Corollary \ref{Cor_L^2 Mono_RMCF w graphical eigenfunc has cN = spectrum}, 
		\begin{align}
			\limsup_{j\to \infty} \cN_{n,k}(\bar{a}_j', \bar{\cM}_j) \leq -\frac12 \,.  \label{Equ_L^2 Mono_In Pf_bar(cM) doubl at a < -1/2}
		\end{align}
		Combine this with \eqref{Equ_L^2 Mono_In Pf_bar(cM) doubl at b > -1/2} and Corollary \ref{Cor_L^2 Mono_Discrete Growth Mono}, we find 
		\begin{align}
			\lim_{j\to \infty}\sup_{\tau\in [\bar{a}_j'+1, \bar{b}_j]}\left|\cN_{n,k}(\tau, \bar{\cM}_j) + \frac12\right| = 0\,.  \label{Equ_L^2 Mono_In Pf_bar(cM) approx -1/2}
		\end{align}
		\begin{align}
			\lim_{j\to \infty} \inf_{\tau \in [\bar{a}'_j-1/2, \bar{b}_j]}\|\bar{u}_j(\cdot, \tau)\|_{L^2}^{-1}\cdot \|\Pi_{=-1/2}\bar{u}_j(\cdot, \tau)\|_{L^2} = 1\,.  \label{Equ_L^2 Mono_In Pf_h_2 dominate graphical func of bar(cM)}
		\end{align}
		
		Let $\zeta\in C^\infty(\RR)$ be a non-increasing function with $\zeta = 1$ on $\RR_{\leq 0}$, $\zeta = 0$ on $\RR_{\geq 1}$ and $|\zeta'|, |\zeta''|\leq 20$. We define a cut-off function on $\cC_{n,k}$: $\zeta_R(\theta, y) = \zeta(|y|-R+1)$.  
		By Corollary \ref{Cor_L^2 Mono_Weighted L^2 Mono and Est for RMCF}, whenever $R\geq R(n, \eps)\gg1$ and $j\gg1$, we have $R_j\geq R$ and for every $\tau\in [\bar{a}_j'-1/2, \bar{b}_j]$, 
		\begin{align}
			\mbfd_{n,k}(\tau, \bar{\cM}_j) \leq 2\|\bar{u}_j(\cdot, \tau)\|_{L^2(\cC_{n,k}\cap Q_{R-1})}\,, \;\; & &
			\|\bar{u}_j(\cdot, \tau)(1-\zeta_R)\|_{L^2} \leq \Psi(R^{-1}|n) \|\bar{u}_j(\cdot, \tau)\zeta_R\|_{L^2}.   \label{Equ_L^2 Mono_In Pf_Small error from cut off}
		\end{align}
		
		By Lemma \ref{Lem_App_RMCF equ}, 
		\begin{align}
			(\partial_\tau - L_{n,k})(\bar{u}_j\zeta_R) = \cQ_{j,R} \quad \text{ on }\ \cC_{n,k}\cap Q_R\times [\bar{a}_j, \bar{b}_j+1]\,,  \label{Equ_L^2 Mono_In Pf_RMCF equ for graphical func}        
		\end{align}
		where by interior parabolic estimate,
		\begin{align*}
			\|\cQ_{j,R}(\cdot, \tau)\|_{L^\infty} \leq C(n)\|\bar{u}_j(\cdot, \tau)\|_{C^2(\cC_{n,k}\cap Q_R)}^2 \leq C(n, R)\mbfd_{n,k}(\bar{\cM}_j(\tau-1))^2 \,.
		\end{align*}
		Combining \eqref{Equ_L^2 Mono_In Pf_bar(cM) doubl upper bd}, \eqref{Equ_L^2 Mono_In Pf_bar(cM) approx -1/2} and Corollary \ref{Cor_L^2 Mono_Weighted L^2 Mono and Est for RMCF} with the definition of decay order, this implies for $j\gg 1$ and $\tau\in [\bar{a}_j, \bar{b}_j]$,
		\begin{align}
			\|\cQ_{j,R}(\cdot, \tau)\|_{L^\infty} \leq C(n, R, \eps)e^{3(\tau-\bar b_j)/4}\mbfd_{n,k}(\bar{\cM}_j(\bar{b}_j))^2 \,.  \label{Equ_L^2 Mono_In Pf_Growth Bd for error Q}
		\end{align}
		
		Let $w_{j, R}(\cdot, \tau)$ be the solution of,
		\begin{align*}
			\begin{cases}
				(\partial_\tau - L_{n,k})w_{j, R} = \Pi_{=-1/2}(\cQ_{j, R}) \,, & \text{ on }\cC_{n,k}\times [\bar{a}_j, \bar{b_j}+1] \,,\\
				w_{j, R}(\cdot, \bar{a}_j) = 0 \,.
			\end{cases}
		\end{align*}    
		
		By multiplying this with $w_{j, R}(\cdot, \tau)$ and integrating over $\cC_{n,k}$, we get (denote $d\mu:= e^{-\frac{|X|^2}{4}}dX$)
		\begin{align*}
			\frac{d}{d\tau}\int_{\cC_{n,k}} \frac{w_{j, R}^2}{2} \ d\mu 
			& = \int_{\cC_{n,k}} \left( -|\nabla w_{j, R}|^2 + w_{j, R}^2 + w_{j, R}\Pi_{=-1/2}(\cQ_{j, R})\right)\ d\mu \\
			& \leq \int_{\cC_{n,k}} \frac{w_{j, R}^2}{2}\ d\mu + \left(\int_{\cC_{n,k}} w_{j, R}^2\ d\mu \right)^{1/2}\left(\int_{\cC_{n,k}} \cQ_{j, R}^2\ d\mu \right)^{1/2}
		\end{align*}
		where the inequality follows from that $w_{j, R}(\cdot, \tau)\perp 1$ in $L^2(\cC_{n,k})$. Define $W_{j, R}(\tau):= \|w_{j, R}(\cdot, \tau)\|_{L^2}$, by \eqref{Equ_L^2 Mono_In Pf_Growth Bd for error Q} and the inequality above, we find
		\begin{align*}
			W_{j, R}'(\tau) - \frac12 W_{j, R}(\tau) \leq \|\cQ_{j, R}(\cdot, \tau)\|_{L^2} \leq C(n, R, \eps)e^{3(\tau-\bar b_j)/4}\mbfd_{n,k}(\bar{\cM}_j(\bar{b}_j))^2 \,.
		\end{align*}
		And then, 
		\begin{align}
			W_{j, R}(\bar{b}_j) & \leq e^{\bar{b}_j/2} \int_{\bar{a}_j}^{\bar{b}_j} C(n, R, \eps)e^{-3\bar{b}_j/4 + \tau/4}\mbfd_{n,k}(\bar{\cM}_j(\bar{b}_j))^2\ d\tau \leq C(n, R, \eps)\mbfd_{n,k}(\bar{\cM}_j(\bar{b}_j))^2 \,.  \label{Equ_L^2 Mono_In Pf_Upper bd error W}
		\end{align}
		
		On the other hand, by \eqref{Equ_L^2 Mono_In Pf_RMCF equ for graphical func} and the definition of $w_{j, R}$, $(\partial_\tau-L_{n,k})(\Pi_{=-1/2}(\bar{u}_j\zeta_R) - w_{j, R}) = 0$. To save notation, let $\tilde{u}_{j, R}:= \bar{u}_j(\cdot, \bar{b}_j)\zeta_R$. Then by Lemma \ref{Lem_App_Analysis of Parab Jacob field} (ii), \eqref{Equ_L^2 Mono_In Pf_Small error from cut off} and \eqref{Equ_L^2 Mono_In Pf_Upper bd error W}, 
		\begin{align*}
			\lim_{j\to \infty}\|\tilde{u}_{j, R}\|_{L^2}^{-1}\cdot \|\Pi_{=-1/2}(\tilde{u}_{j, R}) - \Pi_{=-1/2}\left(\bar{u}_{j, R}(\cdot, \bar{a}_j)\zeta_R \right)e^{(\bar{b}_j - \bar{a}_j)/2}\|_{L^2} = 0 
		\end{align*}
		Combine this with \eqref{Equ_L^2 Mono_In Pf_h_1 and h_2 dominate graphical func of bar(cM)}, \eqref{Equ_L^2 Mono_In Pf_h_2 dominate graphical func of bar(cM)} and \eqref{Equ_L^2 Mono_In Pf_Small error from cut off}, we obtain 
		\begin{align}
			\limsup_{j\to \infty}\|\bar{u}_{j}(\cdot, \bar b_j)\|_{L^2}^{-1}\cdot \|\bar{u}_j(\cdot, \bar b_j) - \bar{c}_j e^{(\bar{b}_j-\bar{a}_j)/2}y\cdot \by_j\|_{L^2}  \leq \Psi(R^{-1}|n, \eps) \,.  \label{Equ_L^2 Mono_In Pf_h_1 dominate at b time slice}
		\end{align}
		
		Recall that $\bar\cM_j$ is given by a scaling of $\cM_j$, hence by Lemma \ref{Lem_App_Graph over Cylinder} (iv), there exists some constant $\lambda_j\to 1$ such that \[
		\bar u_j(\theta, y, \bar b_j) = \lambda_j u_j(\theta, \lambda_j^{-1}y) + \sqrt{2(n-k)}(\lambda_j-1)\,. 
		\]
		Hence \eqref{Equ_L^2 Mono_In Pf_h_1 dominate at b time slice} implies \[
		\limsup_{j\to \infty} \inf_{c>0, c'\in \RR}\|c^{-1}u_j(\cdot, b_j) - c' - y\cdot \by_j\|_{L^2} \leq \Psi(R^{-1}|n, \eps)\,.
		\]
		By taking $R\gg 1$, this contradicts to \eqref{Equ_L^2 Mono_In Pf_Contrad assump}.    
	\end{proof}
	
	We close this section by the following consequence of $h_1$-domination condition \eqref{Equ_L^2 Mono_h_1 dominate conclusion}, which will be used in Section \ref{S:GandTProperty}.
	\begin{Lem} \label{Lem_L^2 Mono_nu cdot by <0}
		For every $\eps\in (0, 1)$, there exists $\delta_7(n, \eps)\in (0, \eps)$ with the following significance. 
		Let $\by\in \RR^k$ be a unit vector, $T>1$, $\cM$ be a unit regular cyclic mod $2$ rescaled mean curvature flow over $[-1, T]$ $\delta_7$-$L^2$ close to $\cC_{n,k}$ over $[-1, 2]$, such that
		\begin{align}
			\lambda[\cM] < \eps^{-1}\,, & &
			\sup_{|s|\leq \eps}\cN_{n,k}(-1, \cM^{(\orig, s)}) < \eps^{-1} \,.  \label{Equ_L^2 Mono_Time transl Doubl Const Upper Bd for seq}
		\end{align}
		Suppose the graphical function $u(\cdot, \tau)$ of $\cM(\tau)$ over $\cC_{n,k}$ satisfies \[
		\inf_{c>0, c'\in \RR} \|c^{-1}u(\cdot, 0) - c' - y\cdot \by\|_{L^2} \leq \delta_7\,.
		\]
		Then for every $\tau\in [0, \min\{T, \eps^{-2}\}]$ and every $X\in \Spt\cM(\tau) \cap Q_{\eps^{-1}}$, the unit normal vector $\nu(X, \tau)$ of $\Spt\cM(\tau)$ at $X$ pointing away from $\spine(\cC_{n,k})$ satisfies \[
		\nu(X, \tau)\cdot (0, \by) < 0 \,.
		\] 
	\end{Lem}
	\begin{proof}
		Suppose for contradiction that there exist a sequence of unit vectors $\by_j\in \RR^k$, $T_j>1$, unit regular cyclic mod $2$ rescaled mean curvature flow $\cM_j$ over $[-1, T_j]$ converging to $\cC_{n,k}$ in the Brakke sense satisfying \eqref{Equ_L^2 Mono_Time transl Doubl Const Upper Bd for seq} with $\cM_j$ in place of $\cM$, and $(X_j, \tau_j)\in \Spt\cM_j\cap (Q_{\eps^{-1}}\times [0, \eps^{-2}])$ such that 
		\begin{enumerate}[label={\normalfont(\alph*)}] %\arabic*; \alph*
			\item The graphical functions $u_j$ of $\cM_j$ over $\cC_{n,k}$ satisfy, \[
			\lim_{j\to \infty} \inf_{c>0, c'\in \RR} \|c^{-1}u_j(\cdot, 0) - c' - y\cdot \by\|_{L^2} = 0\,;
			\]
			\item $\nu_j(X_j, \tau_j)\cdot (0, \by_j)\geq 0$.
		\end{enumerate} 
		
		After passing to a subsequence, set $\by_j\to \by_\infty$. By \eqref{Equ_L^2 Mono_Time transl Doubl Const Upper Bd for seq} and Lemma \ref{Lem_L^2 Mono_Time transl to make cN > -1/2}, there exists $s_j\to 0$ such that $\tilde\cM_j:= \cM_j^{(0, s_j)}$ satisfies 
		\begin{align*}
			\limsup_{j\to \infty}\cN_{n,k}(-\log(e+s_j), \tilde\cM_j) <+\infty \,, & &
			\liminf_{j\to \infty}\cN_{n,k}(-\log(1+s_j), \tilde\cM_j) \geq -\frac 12 \,.
		\end{align*}
		Let $\tilde u_j(\cdot, \tau)$ be the graphical function of $\tilde\cM_j(\tau)$ over $\cC_{n,k}$. 
		Then by Lemma \ref{Lem_L^2 Mono_Induced Parab Jac}, the subsequential $C^\infty_{loc}$-limit $\hat u_\infty$ of $\hat u_j:=\|\tilde u_j(\cdot, -\log(1+s_j))\|_{L^2}^{-1}\cdot\tilde u_j$ is a nonzero parabolic Jacobi field on $\cC_{n,k}\times (-1, +\infty)$ satisfying 
		\begin{align}
			\log\left( \frac{\|\hat{u}_\infty(\cdot, 0)\|_{L^2}}{\|\hat{u}_\infty(\cdot, 1)\|_{L^2}}\right) \geq -\frac12 \,. \label{Equ_L^2 Mono_In Pf_Linear Doubl const geq -1/2}        
		\end{align}
		
		On the other hand, by (a) and Lemma \ref{Lem_App_Graph over Cylinder}, 
		\begin{align*}
			\lim_{j\to \infty}\inf_{c>0, c'\in \RR} \|c^{-1}\tilde u_j(\cdot, -\log(1+s_j)) - c' - y\cdot \by_j\|_{L^2} = 0\,.
		\end{align*}
		Hence $\hat u_\infty(\theta, y, 0) = \hat c y\cdot \by_\infty + \hat c'$ for some constants $\hat c\geq 0$ and $\hat c'\in \RR$. While by \eqref{Equ_L^2 Mono_In Pf_Linear Doubl const geq -1/2} and Lemma \ref{Lem_App_Analysis of Parab Jacob field}, we must have $\hat c' = 0$ and $\hat c>0$. Since the convergence of $\hat u_j$ to $\hat u_\infty$ is in $C^\infty_{loc}$, we must have $(0, \by_j)\cdot \partial_y \tilde u_j>0$ in $P:= Q_{2\eps^{-1}}\times [0, \min\{2\eps^{-2}, -\log(e^{-T_j}+s_j)\}]$ when $j\gg 1$. By Lemma \ref{Lem_App_Graph over Cylinder} (ii), this means the unit normal vector of $\Spt\tilde\cM_j(\tau)$ pointing away from $\spine(\cC_{n,k})$ satisfies $\tilde \nu_j\cdot (0, \by_j)<0$ in $P$, contradicts to (b) since time slices of $\cM_j$ are just rescalings of time slices of $\tilde\cM_j$, thus have the same unit normal.
	\end{proof}

	\section{Geometric and topological properties of flow passing nondegenerate singularities}\label{S:GandTProperty}
	The goal of this section is to prove the Theorem \ref{Thm_Isol_Main}.  
	
	Since the behavior backward in time has been proved in \cite{SunXue2022_generic_cylindrical}, the bulk of this section is focused on the case forward in time.
	
	We start with the following clearing-out lemma, which is essentially Theorem 6.1 in \cite{ColdingMinicozzi16_SingularSet},  and can be proved directly by a blow-up argument. This clearing-out lemma does not require the cylindrical singularity to be nondegenerate. 
	\begin{Lem}[Theorem 6.1 in \cite{ColdingMinicozzi16_SingularSet}]
		\label{Lem_Isol_spt(M) subset cW}
		Assume the assumptions in Theorem \ref{Thm_Isol_Main}. There exist $t_1\in (0, 1)$ and an increasing function $\zeta: [0, t_1]\to \RR_{\geq 0}$ (both depending on $\mbfM$) such that $\lim_{r\to 0}r^{-1}\zeta(r) = +\infty$ and \[
		\spt(\mbfM) \cap (Q_{\zeta(t_1)}\times [0, t_1]) \subset \cW:= \{(x, y, t)\in \RR^{n-k+1}\times \RR^k\times [0, t_1]: \zeta(|x|+\sqrt{t}) \leq |y|\} \,.
		\]
	\end{Lem}
	
	The major effort of this section is devoted to the following characterization of the blow up models:
	\begin{Thm} \label{Thm_Isol_Blowup Model}
		Let $\bM$ be the same as Theorem \ref{Thm_Isol_Main}. We further assume for the moment \footnote{These assumptions are always true and are proved a posteriori in Proposition \ref{Prop:MCF with boundary} without using Theorem \ref{Thm_Isol_Blowup Model}. The second bullet point is obtained via elliptic regularization.} that 
		\begin{itemize}
			\item the weak set flow $\spt\bM\cap \overline{Q_1}\times(-1, 1)$ is mean convex in the sense of Remark \ref{Rem:LSF};
			\item there exist $\al>0$, $r_\circ, t_\circ\in (0, 1)$ such that $t\mapsto \bM(t)\times \RR$ is a limit of a sequence of smooth $\al$-noncollapsing flows in $Q_{r_\circ}\times \RR$ over $ (-t_\circ, t_\circ)$.
		\end{itemize}
		
		Suppose $p_j=(x_j, y_j, t_j)\in \spt\mbfM$, $\by\in \RR^{k}$ be a unit vector such that 
		\begin{align*}
			p_j \to (\orig, 0)\,, & & \frac{y_j}{|y_j|}\to \by\,, & & t_j\geq 0\,,
		\end{align*}
		as $j\to \infty$. Then there exist $\lambda_j\to 0$ such that $\tilde\mbfM_j:= \lambda_j^{-1}(\mbfM-p_j)$ locally smoothly subconverges to some ancient mean curvature flow $\tilde\mbfM_\infty$ on $\RR^{n+1}\times \RR_{\leq 0}$, with one of the following holds,
		\begin{enumerate}[label={\normalfont(\alph*)}] %\arabic*; \alph*
			\item $\tilde\bM_\infty(0)$ is a translation and dilation of $\cC_{n,k}$. Moreover, let $\tilde\nu_j$ be the unit normal of $\Spt \tilde\bM_j$ pointing away from $\spine(\tilde\bM_\infty)$, then for $j\gg 1$, \[
			\tilde\nu_j(\orig, 0)\cdot (0, \by) < 0 \,.
			\] 
			\item $\tilde\bM_\infty(0) \in \mathscr{B}_{n,k}$ (see the notation in Section \ref{Subsec_Classif by Du-Zhu}) with translating direction $\by$.
		\end{enumerate}
	\end{Thm}
	Applying Brakke-White epsilon regularity for mean curvature flow, a direct consequence of Theorem \ref{Thm_Isol_Blowup Model} is that $(\orig, 0)$ is an isolated singularity of $\bM$ in a forward neighborhood.
	\begin{Rem} \label{Rem_Topo_Blow rate uniq}
		The scaling factor $\lambda_j$ in the Theorem above are unique up to a finite multiple. More precisely, if $(\lambda_j^\pm)_{j\geq 1}$ are two sequences of positive real numbers such that for each $i\in \{\pm\}$, $(\lambda_j^i)^{-1}(\mbfM-p_j)$ converges in $C^\infty_{loc}$ to $\bM_\infty^i$ in $\RR^{n+1}\times \RR_{\leq 0}$, and that $\bM_\infty^i(0) \in \overline{\mathscr{B}_{n,k}}$ (see the notation in Section \ref{Subsec_Classif by Du-Zhu}). Then the following limit exists, \[
		0< \lim_{j\to \infty} (\lambda_j^-)^{-1}\cdot\lambda_j^+ <+\infty \,.
		\]
		And $\bM_\infty^+(0)$ is a rescaling of $\bM_\infty^-(0)$. 
		
		To see this, since there's no scaling invariant element in $\overline{\mathscr{B}_{n,k}}$, by possibly flipping $\lambda_j^+$ and $\lambda_j^-$, it suffices to show that \[
		\lim_{j\to \infty} (\lambda_j^-)^{-1}\cdot\lambda_j^+ < +\infty\,.
		\] 
		Suppose for contradiction that, after passing to a subsequence,  $(\lambda_j^-)^{-1}\cdot\lambda_j^+\to +\infty$ as $j\to \infty$. Then for every $T>1$, 
		\begin{align*}
			\cF[\bM_\infty^+(-1)] = \lim_{j\to \infty} \cF[((\lambda_j^+)^{-1} (\bM - p_j))(-1)] & = \lim_{j\to \infty} \Theta_{p_j}((\lambda_j^+)^2; \bM) \\
			& \geq \liminf_{j\to \infty} \Theta_{p_j}((\lambda_j^-\cdot T)^2; \bM) = \cF[T^{-1}\bM_\infty^-(-T^2)]\,.
		\end{align*}
		Sending $T\to \infty$, we find \[
		\cF[\cC_{n,k}] = \lim_{T\to +\infty}\cF[T^{-1}\bM_\infty^-(-T^2)] \leq \cF[\bM^+_\infty(-1)]\leq \lim_{\tilde T\to +\infty} \cF[\tilde T^{-1}\bM_\infty^+(-\tilde T^2)] = \cF[\cC_{n,k}]\,.  \] 
		Then by the rigidity of Huisken's monotonicity formula, $\bM_\infty^+$ is a self-shrinker, contradicting that $\bM_\infty^+(0)$ is a smooth hypersurface in $\overline{\mathscr{B}_{n,k}}$.
	\end{Rem}
	
	\begin{proof}[Proof of Theorem \ref{Thm_Isol_Blowup Model}.]
		Let $\eps_1\in (0, 1/4)$ be fixed for the moment such that $\lambda[\mbfM]\leq \eps_1^{-1}$. The main goal is to find a sequence of blow-up factors $\{\lambda_j\}_{j\geq 1}$ (possibly depending on $\eps_1$) such that the subsequent blow-up limit $\tilde\bM_\infty$ as in Theorem \ref{Thm_Isol_Blowup Model} satisfies either (a) or the following 
		\begin{enumerate}
			\item [(b)'] $\tilde\bM_\infty(0) \in \mathscr{B}_{n,k}$ with translating direction $C_n\eps_1$-close to $\by$ in $\RR^k$.
		\end{enumerate}
		This, together with Remark \ref{Rem_Topo_Blow rate uniq}, proves Theorem \ref{Thm_Isol_Blowup Model} immediately by sending $\eps_1\to 0$. 
		
		Let $\bM,\ p_j=(x_j, y_j, t_j)$ be as in the Theorem. By Lemma \ref{Lem_Isol_spt(M) subset cW}, when $j\gg 1$, $y_j\neq \orig$. Let $\tau\mapsto \cM(\tau)$ be the rescaled mean curvature flow of $\bM$ at $(\orig, 0)$. Then by \eqref{eq:RMCF change of base point}, the rescaled mean curvature flow of $\mbfM$ at $(0, y_j, t_j)$ is $\tau\mapsto \cM_j(\tau):=\cM^{(0, y_j, t_j)}(\tau)$, where 
		\begin{align}
			\cM_j(\tau) = \sqrt{1-t_je^\tau}\cdot \cM \left(\tau-\log(1-t_je^\tau) \right)-e^{\tau/2}(0, y_j) \,, \label{Equ_Isol_RMCF w new base point}        
		\end{align}
		whenever $t_je^{\tau}<1$. We fix $\mbfL %= \mbfL(n,\epsilon)
		\gg 2n$ to be determined. For $j\gg 1$, let \[
		a_j:= 2\log(|y_j|^{-1}\mbfL) \,.
		\] 
		\begin{Claim} \label{Claim_Time Transl cM_j small d_n,k and bd cN}
			There exists $\eps_2(n,\eps_1)\in (0, \eps_1)$ such that for every $\eps\in (0, \eps_2]$, we have 
			\begin{align}
				\limsup_{j\to \infty} & \sup_{|\tau-a_j|\leq 3, |s|\leq \eps e^{-a_j}} \mbfd_{n,k}((\cM_j)^{(\orig, s)}(\tau)) \leq C(n)\eps \,; \label{Equ_Isol_Time Transl cM_j has d_n,k small} \\
				\limsup_{j\to \infty} & \sup_{|\tau-a_j|\leq 3, |s|\leq \eps_2 e^{-a_j}} \cN_{n,k}(\tau, (\cM_j)^{(\orig, s)}) < \eps_2^{-1}\,. \label{Equ_Isol_Time Transl cM_j has bded cN}
			\end{align}
		\end{Claim}
		\begin{proof}
			First note that $a_j\to +\infty$ as $j\to \infty$; and by Lemma \ref{Lem_Isol_spt(M) subset cW}, since $p_j\in\cW$ when $j\gg1$, we have 
			\begin{align}
				e^{a_j/2}|x_j| + e^{a_j}t_j \to 0\,, \quad \text{ as }j\to \infty\,. \label{Equ_Isol_e^(a_j)(|x_j|^2+ |t_j|) to 0}
			\end{align}
			Also note that \[
			(\cM_j)^{(\orig, s)}(\tau) = \cM^{(0, y_j, t_j+s)}(\tau) = \sqrt{1-(t_j+s)e^\tau}\cdot  \cM \left(\tau-\log(1-(t_j+s)e^\tau) \right)-e^{\tau/2}(0, y_j) \,,  \]
			and when $|\tau-a_j|\leq 3$, $|e^{\tau/2}y_j|\leq 10\mbfL$. Since $\cM$ has finite entropy and $\cM(\tau)\to \cC_{n,k}$ in $C^\infty_{loc}$ as $\tau\to +\infty$, we have \[
			\limsup_{j\to \infty} \sup_{|\tau-a_j|\leq 3, |s|\leq \eps e^{-a_j}} \mbfd_{n,k}((\cM_j)^{(\orig, s)}(\tau)) \leq \sup_{|s|\leq 50\eps}\mbfd_{n,k}(\sqrt{1-s}\,\cC_{n,k}) \leq C_n \eps \,,
			\]  
			when $\eps<1/100$. This proves \eqref{Equ_Isol_Time Transl cM_j has d_n,k small}. 
			
			To prove \eqref{Equ_Isol_Time Transl cM_j has bded cN}, first recall that since $(\orig, 0)$ is a non-degenerate singular point of $\bM$, by the Example \ref{Eg_Nondegen sing has cN = 0}, there exists $\tau_\circ\gg 1$ such that 
			\begin{itemize}
				\item $\cM$ is $\delta_4$-$L^2$ close to $\cC_{n,k}$ over $[\tau_\circ, +\infty)$, where $\delta_4(n, \eps_1)$ is determined by Corollary \ref{Cor_L^2 Mono_cN(cM^X) bd};
				\item $|\cN_{n,k}(\tau, \cM)|\leq \eps_1$ for $\tau\geq \tau_\circ$.
			\end{itemize}
			In particular, \[
			\delta_4 \mbfd_{n,k}(\cM(\tau_\circ)) > 0 = \lim_{j\to \infty} \sup_{|s|\leq  e^{-a_j}} \left(e^{\tau_\circ/2}|y_j| + e^{\tau_\circ}|t_j+s| \right) \,.
			\]
			Thus by Corollary \ref{Cor_L^2 Mono_cN(cM^X) bd} and setting $\eps_2(n,\eps_1)\ll 1$ such that $C(n)\eps<\delta_4$ in \eqref{Equ_Isol_Time Transl cM_j has d_n,k small}, \eqref{Equ_Isol_Time Transl cM_j has bded cN} holds with an even smaller $\eps_2(n, \eps_1)>0$.
		\end{proof}
		
		\begin{Claim} \label{Claim_u_j dominated by h_1}
			Let $u_j(\cdot, \tau)$ be the graphical function of $\cM_j(\tau)$ in $Q_\mbfL$ and $0$-extended to an $L^\infty$ function on $\cC_{n,k}$. Then 
			
			\begin{align}
				\limsup_{j\to \infty} \inf_{c>0, c'\in \RR} \|c^{-1}u_j(\cdot, a_j) - c' - y\cdot \by\|_{L^2} & \leq \Psi(\mbfL^{-1}|n)\,; \label{Equ_Isol_u_j dominated by h_1} \\
				\limsup_{j\to \infty} \|u_j(\cdot, a_j)\|_{L^2}^{-1}\cdot \|\Pi_{\leq -1/2} (u_j(\cdot, a_j))\|_{L^2} & \geq 1-\Psi(\mbfL^{-1}|n) \,. \label{Equ_Isol_u_j dominated by leq -1/2 modes}
			\end{align}
		\end{Claim}
		\begin{proof}
			\eqref{Equ_Isol_u_j dominated by leq -1/2 modes} is a direct consequence of \eqref{Equ_Isol_u_j dominated by h_1}. We now focus on proving \eqref{Equ_Isol_u_j dominated by h_1}.
			
			Since $(\orig, 0)$ is a nondegenerate singularity of $\bM$, by Theorem \ref{thm:NormalForm} and interior parabolic estimate, for $\tau\geq \tau_0(n, \mbfL, \bM)\gg 1$, $\cM(\tau)$ is graphical over $\cC_{n,k}\cap Q_{3\mbfL}$, with the graphical function $u(\cdot, \tau)$ satisfying pointwisely \[
			u(\theta, y; \tau) = c_{n,k}\tau^{-1}(|y|^2 - 2) + o(\tau^{-1})\,, \quad \text{ in }\ \cC_{n,k}\cap Q_{3\mbfL}
			\]
			for some $c_{n,k}>0$. While by \eqref{Equ_Isol_RMCF w new base point}, $\cM_j(a_j) = \lambda_j\cdot \cM(a_j') - (0, \hat y_j)$, where by \eqref{Equ_Isol_e^(a_j)(|x_j|^2+ |t_j|) to 0},
			\begin{align*}
				\lambda_j:= \sqrt{1-t_j e^{a_j}} \to 1\,, & &
				a_j':= a_j -\log(1-t_je^{a_j}) \to +\infty\,, & &
				\hat y_j:= \lambda_je^{a_j/2}y_j
			\end{align*}
			as $j\to \infty$, and by definition of $a_j$, $|\hat y_j|\to \mbfL$.
			
			By Lemma \ref{Lem_App_Graph over Cylinder} (iv), in $\cC_{n,k}\cap Q_\mbfL$, for $j\gg 1$, 
			\begin{align*}
				u_j(\theta, y; a_j) & = \lambda_j u(\theta, \lambda_j^{-1}(y+\hat y_j); a_j') + \sqrt{2(n-k)}(\lambda_j - 1) \\
				& = \underbrace{\lambda_j \cdot c_{n,k} (a_j')^{-1}(\left|\lambda_j^{-1}(y+\hat y_j)\right|^2-2) + \sqrt{2(n-k)}(\lambda_j - 1)}_{=:\ q_j(\theta, y)} + o((a_j')^{-1}) \,.
			\end{align*}
			Here $q_j$ is a quadratic polynomial in $y$ and is invariant in $\theta$. Explicitly, 
			
			\[ q_j(\theta, y) = \alpha_{j,2}|y|^2 + \alpha_{j,1} \ y\cdot\frac{y_j}{|y_j|} + \al_{j,0}, \]
			where 
			\begin{align*}
				\al_{j,2} & :=  \lambda_j^{-1}c_{n,k}(a_j')^{-1}, \\ 
				\al_{j,1} & := 2\lambda_j^{-1} \cdot c_{n,k}(a_j')^{-1}|\hat y_j|\,, \\
				\al_{j,0} & := \lambda_j^{-1}c_{n,k}(a_j')^{-1}(|\hat y_j|^2-2\lambda_j^2) + \sqrt{2(n-k)}(\lambda_j-1).
			\end{align*}
			
			Since $y_j/|y_j|\to \by$, we then have as $j\to \infty$,
			\begin{align*}
				\inf_{c>0, c'\in \RR}\|c^{-1}u_j(\cdot, a_j) - c' - y\cdot\by\|_{L^2} & \leq \|\al_{j,1}^{-1}(q_j - \al_{j,0}) - y\cdot \by\|_{L^2(\cC_{n,k}\cap Q_\mbfL)} + o(\al_{j,1}^{-1}(a_j')^{-1}) \\
				& \leq \al_{j, 1}^{-1}\,\al_{j,2}\||y|^2-2\|_{L^2(\cC_{n,k}\cap Q_\mbfL)} + o(1)  = \Psi(\mbfL^{-1}|n) + o(1)\,.
			\end{align*}
			This proves \eqref{Equ_Isol_u_j dominated by h_1}. 
			
		\end{proof}
		With this Claim, let $R_1(n, \eps_2)\gg 1$, $\bar{\delta}:= \min\{\delta_j(n, \delta_7(n, \eps_2)): 1\leq j\leq 7\}\in (0, \eps_2)$ be specified throughout Section \ref{Sec_L^2 Mono}, $\mbfL(n, \eps)\gg1$ be such that $\mbfL\geq R_1(n, \eps_2)$ and $\Psi(\mbfL^{-1}|n)< \bar\delta$. Define \[
		b_j:= \sup\{\tau\geq a_j: \mbfd_{n,k}(\cM_j(\tau'))\leq \bar{\delta}, \; \forall\, \tau'\in [a_j, \tau+2]\}
		\]
		Then when $j\to \infty$, we have $b_j-a_j\to +\infty$ (a priori, $b_j$ might be $+\infty$.) 
		
		We first conclude from \eqref{Equ_Isol_Time Transl cM_j has bded cN}, \eqref{Equ_Isol_u_j dominated by leq -1/2 modes}, Corollary \ref{Cor_L^2 Mono_RMCF w graphical eigenfunc has cN = spectrum} and \ref{Cor_L^2 Mono_Discrete Growth Mono} that \[
		\limsup_{j\to \infty} \sup_{\tau\in [a_j, b_j],\ |s|\leq \eps_2e^{-\tau}} \cN(\tau; \cM_j^{(\orig, s)}) \leq -1/2+\eps_2 < -1/4\,,
		\]
		Hence when $j\gg 1$, by definition of $\cN_{n, k}(\cdot, \cM_j)$ and Corollary \ref{Cor_L^2 Mono_Weighted L^2 Mono and Est for RMCF}, $\forall\,\tau_1<\tau_2\in [a_j, b_j)$, we have \[
		\mbfd_{n,k}(\cM_j(\tau_2)) \geq C(n)^{-1}e^{(\tau_2-\tau_1)/4}\mbfd_{n,k}(\cM_j(\tau_1))\,.
		\]
		Together with the definition of $b_j$, this in particular implies $b_j<+\infty$ and
		\begin{align}
			\mbfd_{n,k}(\cM_j(\tau)) \leq C(n)e^{(\tau- b_j)/4}\mbfd_{n,k}(\cM_j(b_j)) \,, \quad \forall\, \tau \in [a_j, b_j]\,.  \label{Equ_Isol_d(tau) < e^(tau/4) backward in time}
		\end{align}
		
		Also by \eqref{Equ_Isol_Time Transl cM_j has bded cN} and Lemma \ref{Lem_L^2 Mono_h_1 dominate in all scale}, when $j\gg1$, 
		\begin{align}
			\limsup_{j\to \infty} \inf_{c>0, c'\in \RR}\|c^{-1}u_j(\cdot, b_j) - c' - y\cdot\by\|_{L^2} \leq \delta_7(n, \eps_2)\,.  \label{Equ_Isol_In Pf_|(u_j)_y(b_j) - y|<eps}
		\end{align}
		where $\delta_7$ is specified in Lemma \ref{Lem_L^2 Mono_nu cdot by <0}.
		
		Let $\lambda_j:= e^{-b_2/2}$ so that $\hat{\mbfM}_j:= \lambda_j^{-1}(\mbfM - (0, y_j, t_j))$ has its rescaled mean curvature flow based at $(\orig, 0)$ to be $\tau\mapsto \hat{\cM}_j = \cM_j(\tau + b_j)$.  We collect properties of $\hat{\cM}_j$ from the analysis above:
		\begin{enumerate}[label={\normalfont(\roman*)}] %\arabic*; \alph*
			\item By definition of $b_j$, we must have $\mbfd_{n,k}(\hat{\cM}_j(2)) = \bar{\delta}$. 
			
			\item When $\tau\in [a_j - b_j, 2]$, by \eqref{Equ_Isol_d(tau) < e^(tau/4) backward in time} 
			\begin{align}
				\mbfd_{n,k}(\hat{\cM}_j(\tau)) \leq \min\{\bar{\delta}, C(n)e^{\tau/4}\bar{\delta}\}\, \label{Equ_Isol_d_n,k(hat cM_j(tau)) exp decay near -infty} . 
			\end{align}
			\item Let $\hat{u}_j(\cdot, \tau)$ be the graphical function of $\hat{\cM}_j(\tau)$ over $\cC_{n, k}\cap Q_\mbfL$, $\tau\in [a_j-b_j, 0]$. Then by \eqref{Equ_Isol_In Pf_|(u_j)_y(b_j) - y|<eps}, 
			\begin{align*}
				\limsup_{j\to \infty} \inf_{c>0, c'\in \RR}\|c^{-1}\hat{u}_j(\cdot, 0) - c' - y\cdot\by\|_{L^2} \leq \delta_7(n, \eps_2)\,.  \label{Equ_Isol_In Pf_|(u_j)_y(b_j) - y|<eps}
			\end{align*}
		\end{enumerate}
		
		Hence let $j\to \infty$, $\hat\bM_j$ subconverges to some Brakke motion $\hat\bM_\infty$ with rescaled mean curvature flow $\hat\cM_\infty$ based at $(\orig, 0)$.  (i) guarantees that $\hat{\cM}_\infty\neq \cC_{n,k}$. While by (ii), (iii) and Remark \ref{Rem_Pre_exp decay near -infty implies bowl or cylinder}, $\hat{\bM}_\infty$ is one of the following
		\begin{enumerate}[label={\normalfont\alph*)}] %\arabic*; \alph*
			\item a space-time translation of round shrinking cylinder $\mbfC_{n,k}: \tau\mapsto \sqrt{-t}\cdot\cC_{n,k}$, $\tau\leq 0$, which doesn't agree with any spacial translation of $\mbfC_{n,k}$;
			\item a mean curvature flow generated by an element in $\mathscr{B}_{n,k}$, i.e. a translation, rotation and dilation of a bowl soliton $\times \RR^{k-1}$.
		\end{enumerate}
		
		We finally address the translation in $\RR^{n-k+1}$-direction. Let $\tilde\bM_j:= \lambda_j^{-1}(\bM - p_j) = \hat\bM_j - (\lambda_j^{-1}x_j, 0, 0)$. Note that since $p_j\in \Spt\bM$, we have $(\orig, 0)\in \tilde\bM_j$.
		\begin{Claim} \label{Claim_x_j/lambda_j bounded}
			We have $\ \limsup_{j\to \infty} \lambda_j^{-1}|x_j| <+\infty$.
		\end{Claim}
		We first finish the proof of Theorem \ref{Thm_Isol_Blowup Model} assuming this Claim. Clearly, Claim \ref{Claim_x_j/lambda_j bounded} guarantees that the subsequent limit $\tilde\bM_\infty$ of $\lambda_j^{-1}(\bM - p_j)$ is a spacial translation of $\hat\bM_\infty$, and then also satisfies one of a) and b). Since $(\orig, 0)\in \Spt\tilde\bM_\infty$, in case a), $\tilde\bM_\infty(0) \neq 0$, and hence is a smooth translation and dilation of $\cC_{n,k}$. Then by (iii) and Lemma \ref{Lem_L^2 Mono_nu cdot by <0}, we must have $\tilde\nu_j(\orig, 0)\cdot (0, \by)<0$ for $j\gg 1$.
		While in case b), still by (iii) and Lemma \ref{Lem_Pre_Compare h_1 dominate direction with transl direction}, the translating direction of $\tilde\bM_\infty(0)$ is $C(n)\delta_7$-close to $\by$ in $\RR^k$. This proves (b)'.
		
		\begin{proof}[Proof of Claim \ref{Claim_x_j/lambda_j bounded}.]
			Suppose for contradiction that, after passing to subsequences, $\Lambda_j:= \lambda_j^{-1}|x_j|\to +\infty$. Let 
			\begin{align*}
				\tilde x_j:= \Lambda_j^{-1}\cdot \lambda_j^{-1}x_j\,, & &
				\hat\bM_j':= \Lambda_j^{-1}\cdot \hat\bM_j\,, & &
				\tilde\bM_j':= \Lambda_j^{-1}\cdot \tilde\bM_j = \hat\bM_j' - (\tilde x_j, 0, 0)\,.
			\end{align*}
			Note that  $|\tilde x_j| = 1$, $(\orig, 0)\in \Spt\tilde\bM_j'$ and the rescaled mean curvature flow of $\hat\bM_j'$ at $(\orig, 0)$ is $\tau\mapsto \hat\cM_j'(\tau) = \hat\cM_j(\tau - 2\log \Lambda_j)$. Also note that $\Lambda_j\to +\infty$ and by \eqref{Equ_Isol_e^(a_j)(|x_j|^2+ |t_j|) to 0}, \[
			-2\log\Lambda_j = 2\log\lambda_j - 2\log|x_j| = -b_j - 2\log(e^{-a_j/2}o(1)) = (a_j - b_j) - 2\log (o(1)) \,.
			\]
			In particular, for $j\gg 1$, $-2\log \Lambda_j \in (a_j-b_j, 0)$. Hence by \eqref{Equ_Isol_d_n,k(hat cM_j(tau)) exp decay near -infty}, $\mbfd_{n,k}(\hat\cM_j'(\tau)) \to 0$ in $C^0_{loc}(\RR)$ and therefore, $\hat \bM_j' \to \mbfC_{n,k}$ in the Brakke sense as $j\to \infty$. Suppose that $\tilde x_j$ subconverges to some unit vector $\tilde x_\infty\in \RR^{n-k+1}$. Then, $\tilde\bM_j'$ subconverges to $\mbfC_{n,k} + (\tilde x_\infty, 0, 0)$, whose support does not contain $(\orig, 0)$. This is a contradiction.
		\end{proof}

		The following topological consequence is an implication of Theorem \ref{Thm_Isol_Blowup Model}.
		\begin{Cor} \label{Cor_Isol_<nu, y><0}
			Let $\mbfM$ be as in Theorem \ref{Thm_Isol_Main}, $\nu_t$ be the outward unit normal field of $\bM(t)$. Then there exist $r_\circ, t_\circ>0$ depending on $\mbfM$ such that for every $t\in (0, t_\circ]$, \[
			\phi(x, y; t):= \nu_t(x, y)\cdot (\orig, y)<0\,,  \]
			for every $(x, y)\in \bM(t)\cap \overline{\BB_{r_\circ}^{n-k+1}\times \BB_{r_\circ}^k}$. 
		\end{Cor}
		\begin{proof}
			Suppose for  contradiction, there exists $p_j = (x_j, y_j, t_j)\in \spt\mbfM$ approaching $(\orig, 0)$ such that $t_j\geq 0$ but
			\begin{align}
				\nu_{t_j}(x_j, y_j)\cdot (\orig, y_j) \geq 0\,.  \label{Equ_Isol_In Pf_<nu, y>geq 0}
			\end{align}
			Also by Lemma \ref{Lem_Isol_spt(M) subset cW}, $y_j\neq 0$ for $j\gg1$. Hence by possibly passing to a subsequence, there exists a unit vector $\by\in \RR^k$ such that $y_j/|y_j|\to \by$. By Theorem \ref{Thm_Isol_Blowup Model}, there exists $\lambda_j\searrow 0$ such that $\tilde\mbfM_j:= \lambda_j^{-1}(\mbfM - p_j)$ locally smoothly subconverges to some $\tilde\mbfM_\infty$ in $\RR^{n+1}\times \RR_{\leq 0}$ satisfying either (a) or (b) in Theorem \ref{Thm_Isol_Blowup Model}. But \eqref{Equ_Isol_In Pf_<nu, y>geq 0} suggests that case (a) can't happen;  While if case (b) happens, then by (\ref{Equ_Isol_In Pf_<nu, y>geq 0}), $\hat{\nu}_0\cdot (\orig, \by) \geq 0$, where $\hat{\nu}_0$ denotes the outward unit normal field of $\tilde\bM_\infty(0)$, which is a bowl soliton $\times \RR^{k-1}$ translating in $\by$-direction. This contradicts to Lemma \ref{Lem_Pre_Compare h_1 dominate direction with transl direction},
		\end{proof}
		
	\end{proof}

	We end this section by the following topological lemma, which is used in proving item \ref{Enum_MainThm_t>0} of Theorem \ref{Thm_Isol_Main}.
	\begin{Lem} \label{Lem_Isol_Abstract lemma of diffeo}
		Let $M$ be a connected compact $n$ manifold with nonempty boundary, $S$ be a closed simply connected $n$ manifold, $\rmp:M\to S$ be a local diffeomorphism onto its image, and restricted to a bijection near $\partial M$. If $S\setminus \rmp(\partial M) = S_+\sqcup S_-$, and $\rmp$ maps a collar neighborhood of $\partial M$ to a collar neighborhood of $\partial S_+$, then $\rmp$ is a diffeomorphism onto $S_+$. 
	\end{Lem}
	\begin{proof}
		Notice that the glued map $\rmp\cup_{\partial M}\id_{S_-}: M\cup_{\partial M} S_- \to S $ is a local homeomorphism between closed manifolds, hence a covering map. Since $M$ is connected and $S$ is simply connected, this is a bijection, so is $\rmp$. 
	\end{proof}

	%%%%%%%%%%%%%%%%%%%%%%%%%%%%%%%%%%%%%%%%%%%%%%%%%%%%%%%%%%%%%%%%%%%%%%%%
	\section{Proof of Main Theorem}\label{S:Application} 
	The goal of this section is to complete the proof of the main theorem \ref{Thm_Isol_Main}. 
	\begin{proof}[Proof of Theorem \ref{Thm_Isol_Main}.]  
		Let $t\mapsto \bM(t)$ be the mean curvature flow as in Theorem \ref{Thm_Isol_Main}. 
		
		The backward in time cases of item \ref{Enum_MainThm_Iso} (isolatedness), \ref{Enum_MainThm_MeanConvex} (mean convexity) and item \ref{Enum_MainThm_t<0} (graphical before singular time) have been proved in \cite[Theorem 1.6 \& 1.7]{SunXue2022_generic_cylindrical}
		
		Item \ref{Enum_MainThm_BdyEvol} (boundary evolution) and \ref{Enum_MainThm_t=0} (profile at the singular time) will be proved in the following section \ref{SSSec_PseudoLoc} using pseudolocality.
		
		The forward in time case of item \ref{Enum_MainThm_MeanConvex} (mean convexity) and \ref{Enum_MainThm_Noncollap} (noncollapsing) will be proved in section \ref{SSSec_Ellip Reg}, with a review of elliptic regularization construction.
		
		The forward in time case of item \ref{Enum_MainThm_Iso} (isolatedness) and \ref{Enum_MainThm_t>0} (graphical after singular time) are then direct consequences of Theorem \ref{Thm_Isol_Blowup Model}, Lemma \ref{Lem_Isol_Abstract lemma of diffeo} and the proved item \ref{Enum_MainThm_MeanConvex} and \ref{Enum_MainThm_Noncollap} above. 
		
		More precisely, notice that Theorem \ref{Thm_Isol_Blowup Model} can be applied since the mean convexity and noncollapsing have been established, to prove item \ref{Enum_MainThm_Iso} (isolatedness) in a sufficiently small forward neighborhood $\BB^{n+1}_{r_\circ}\times [0, t_\circ]$, suppose for contradiction that there exists $p_j:= (x_j, y_j, t_j)\in \Sing(\mbfM)$ approaching $(\orig, 0)$ with $t_j\geq 0$ for $j\gg1$; By Lemma \ref{Lem_Isol_spt(M) subset cW}, $y_j \neq \orig$ for $j\gg1$. Then by possibly passing to a subsequence, there exists a unit vector $\by\in \RR^k$ such that $y_j/|y_j|\to \by$. By Theorem \ref{Thm_Isol_Blowup Model}, there exists $\lambda_j\searrow 0$ such that $\mbfM_j:= \lambda_j^{-1}(\mbfM - p_j)$ subconverges to some $\mbfM_\infty$ in the Brakke sense satisfying either (a) or (b) in Theorem \ref{Thm_Isol_Blowup Model}.
		Since in both cases, $\bM(0)$ is a smooth hypersurface with multiplicity $1$, by Brakke-White's epsilon Regularity Theorem \cite{White05_MCFReg}, when $j\gg1$, $\mbfM_j$ is regular near $(\orig, 0)$, and hence $p_j$ can't be a singular point of $\mbfM$. This is a contradiction.
		
		To prove item \ref{Enum_MainThm_t>0} (graphical after singular time), let $\phi$ be defined as in Corollary \ref{Cor_Isol_<nu, y><0}. Notice that by a direct calculation, the smooth map $\bP_t$ in \ref{Enum_MainThm_t>0} is non-degenerate at $(x, y)\in \spt\bM(t)$ if and only if $\phi(x, y; t)\neq 0$. Hence by Corollary \ref{Cor_Isol_<nu, y><0}, $\bP_t$ is a local diffeomorphism on $\bM(t)\cap \overline{\BB_{r_\circ}^{n-k+1}\times \BB_{r_\circ}^k}$ onto its image. And since by \ref{Enum_MainThm_BdyEvol}, near \[
		\partial \left(\bM(t)\cap \overline{\BB_{r_\circ}^{n-k+1}\times \BB_{r_\circ}^k} \right) = \bM(t)\cap (\BB_{r_\circ}^{n-k+1}\times \SSp_{r_\circ}^{k-1})  \,, \]
		$\bP_t$ is a diffeomorphism, it must be a global diffeomorphism by Lemma \ref{Lem_Isol_Abstract lemma of diffeo}.
		
		Finally, item \ref{Enum_MainThm_Top} (topology change) is a direct consequence of \ref{Enum_MainThm_t<0} and \ref{Enum_MainThm_t>0} proved above.      
	\end{proof}

	\subsection{Pseudo Locality} \label{SSSec_PseudoLoc}
	We first review a consequence of the pseudolocality theorem. Recall that the pseudolocality of mean curvature flow, first proved by Ecker-Huisken \cite{EckerHuisken91_Int_Reg} (see also \cite{ChenYin07_PseudoMCF, IlmanenNevesSchulze19_network}), showing that if two hypersurfaces with bounded entropy are (Lipschitz) close to each other in a ball, then under the evolution of mean curvature flow in a short amount of time, they are still close to each other in a possibly smaller ball. Using pseudolocality to the annulus region of the rescaled mean curvature flow that is very close to a cylinder with radius $\varrho'$ inside this annulus region gives the following lemma, which was proved in \cite{SunXue2022_generic_cylindrical}.
	\begin{lemma}[Lemma 6.1 in \cite{SunXue2022_generic_cylindrical}]\label{lem:PseudoLocalityAnnuli}
		Suppose $t\mapsto \mathbf{M}(t)$ is a unit regular cyclic Brakke flow, $t\in(-1,1)$, with a nondegenerate singularity modeled by $\cC_{n,k}$ at $(\orig, 0)$. 
		
		Then for any $\epsilon'\in (0, 1]$, there exists $\tau' = \tau'(\eps', \bM)>0$ and $R>0$ such that for any $\tau_0\geq\tau'$ and $t\in [t_0, -t_0/10]$, where $t_0:=-e^{-\tau_0}$, 
		$$\spt\mbfM(t)\cap \left(\BB_{\sqrt{-t_0}\sqrt{-\log(-t_0)}}\backslash \BB_{\sqrt{-t_0}(\sqrt{-\log(-t_0)}-R)} \right)$$ 
		is a smooth hypersurface $\epsilon'$-close to the homothetically shrinking mean curvature flow  $\mathbf{N}(t)$ in $C^1$-norm, and $\sqrt{-t_0}\epsilon'$-close in $C^0$-norm, where
		\[
		\mathbf N(t):=\sqrt{(-t_0/2-t)}\,\cC_{n,k}\,.
		\]
	\end{lemma}
	We remark that although in \cite{SunXue2022_generic_cylindrical} we only proved the $C^1$ closeness for the rescaled mean curvature flow, the $C^1$ and $C^0$ closeness for the mean curvature flow is just obtained directly from rescaling.
	
	%\label{Subsec_App_Asymp Profile}
	Item \ref{Enum_MainThm_BdyEvol} of Theorem \ref{Thm_Isol_Main} follows directly from this Lemma by taking $r_\circ<<1$. Another direct consequence of Lemma \ref{lem:PseudoLocalityAnnuli} is the shape of the flow at $t= 0$, which proves Theorem \ref{Thm_Isol_Main} \ref{Enum_MainThm_t=0}: 
	\begin{proposition}\label{prop:0 time profile}
		Under the assumption of Lemma \ref{lem:PseudoLocalityAnnuli}, and let $t':= -e^{-\tau'}$, the hypersurfaces $\bM(t)$ will converge to a smooth hypersurface as $t\nearrow 0$, denoted by $\bM(0)$, in $C^1$-norm in $\BB_{\sqrt{-t'}\sqrt{-\log(-t')}}\ \backslash\{\orig\}$, and we can write $\bM(0)$ as a graph of function $v(\theta,y) - \varrho$ over $\cC_{n,k}$, satisfying the following $C^0$-expansion \[
		v(\theta,y) =  \frac{\varrho|y|}{2\sqrt{-\log|y|}}(1 + o_{y}(1)),
		\]
		where recall $\varrho = \sqrt{2(n-k)}$.
	\end{proposition}
	
	\begin{proof}
		Let $W(r):= r\sqrt{-\log(r^2)}$, note that $W$ is increasing on $[0, e^{-1})$ and $W(0) = 0$. 
		
		By the Pseudo-locality Lemma above, for every $t_0\in [t', 0)$ and $|y| = \sqrt{-t_0}\sqrt{-\log(-t_0)} = W(\sqrt{-t_0})$, we have $\spt\bM(0)$ is a graph over $\cC_{n,k}$ of some $v - \sqrt{2(n-k)}$ near $\partial \BB_{W(\sqrt{-t_0})}$, and  \[
		v(\theta, y) = \varrho\sqrt{-t_0/2} + o(\sqrt{-t_0}) = \frac{\varrho}{\sqrt{2}}W^{-1}(|y|)\left(1 + o_{t_0}(1) \right) \,,
		\]
		where $o_{t_0}(1)$ represents some function approaching $0$ when $t_0\to 0$. 
		
		We now study the asymptotics of $W^{-1}(|y|)$ when $|y|\sim 0$: As $r\searrow 0$, we have $W(r)/r\nearrow +\infty$ and $W(r)^2/r \searrow 0$. Hence if we let $g(s):= W^{-1}(s)/s$, then $g(s)\searrow 0$ and $g(s)/s\nearrow +\infty$ as $s\to 0$. Moreover, \[
		s = W(s\cdot g(s)) = s\cdot g(s)\sqrt{-2\log(s) - 2\log(g(s))} = (1+o_s(1))s\cdot g(s)\sqrt{-2\log(s)}\,;
		\] 
		And hence, \[
		W^{-1}(s) = s\cdot g(s) = \frac{s}{\sqrt{-2\log (s)}}\cdot (1+o_s(1))
		\]
	\end{proof}
	
	\subsection{Elliptic Regularization and Noncollapsing} \label{SSSec_Ellip Reg}
	We first recall the setting of elliptic regularization by Ilmanen \cite{Ilmanen94_EllipReg}. While this method works for a general hypersurface, we will be focused on mean convex case. Suppose $\cK_0$ is a closed mean convex hypersurface in $\R^{n+1}$, and we assume $\cK_0=\pr\Omega$ where $\Omega$ is a smooth domain. For every $\lambda>0$, let $N_\lambda$ be a minimizer (as a $n+1$-flat chain mod $2$ in $\RR^{n+1}\times \RR$) of the Ilmanen's functional \[ 
	\cI_\lambda(N):=\int_{N} e^{-\lambda x_{n+2}}d\cH^{n+1}(x) 
	\]
	in $\overline{\Omega}\times\R$ subject to $\pr N_\lambda=[\cK_0\times\{0\}]$. Here, given an $n$-dimensional submanifold (or more generally, an $n$-rectifiable set) $K$, we use $[K]$ to denote the mod $2$ $n$-current generated by $K$, and use $|K|$ to denote the integral $n$-varifold associated to $K$. By \cite[Appendix A]{White15_SubseqSing_MeanConvex}, $N_\lambda$ is the graph of a smooth function, and $t\mapsto \bN_\lambda(t) := N_\lambda -\lambda t\, \vec{e}_{n+2}$ is a mean curvature flow. Moreover, as $\lambda\to\infty$, $t\mapsto \bN_\lambda(t)$ subconverges in the Brakke sense to $t\mapsto \mbfM(t)\times\R$ in $\RR^{n+1}\times \RR$ over $t\in (0, +\infty)$, where $t\mapsto\mbfM(t)$ is a unit regular cyclic mod $2$ Brakke flow with $\mbfM(0)=[\cK_0]$. This is a method to construct Brakke flow with prescribed initial data.
	
	In \cite[Section 5]{White15_SubseqSing_MeanConvex}, White discussed how to use the elliptic regularization to construct mean convex mean curvature flow with prescribed boundary $t\to\Gamma_t$. Following the same idea, we shall show that in a special case which is sufficient for our use, the given Brakke flow coincides with the Brakke flow from the elliptic regularization and therefore shares the favorable noncollapsing property.

	\begin{proposition}\label{Prop:MCF with boundary}
		Suppose $0<r_1<r_0<r_2$. Let $t\mapsto \Sigma_t$ be a unit-regular cyclic mod $2$ Brakke flow in $\BB^{n+1}_{r_2}$ over $[0,T]$. Suppose 
		\begin{enumerate}[label={\normalfont(\alph*)}] %\roman*; \arabic*
			\item \label{Item_EllReg_M_0MeanConv} $\Sigma_0 = \|\mbfM_0\|$, where $\mbfM_0 = \partial \Omega \cap \BB_{r_2}$ and $\Omega$ is a bounded smooth strictly mean convex domain in $\RR^{n+1}$ which meets $\partial\BB_{r}$ transversely for every $r\in [r_1, r_2]$;
			\item \label{Item_EllReg_BdyMonot} In $\BB_{r_2}\backslash \BB_{r_1}$, $\Sigma_t = \|\mbfM_t\|$ for all $t\in[0,T]$, where $t\mapsto \mbfM_t$ is a non-empty classical strictly mean convex mean curvature flow (and hence, $t\mapsto \bM_t$ moves monotone inward $\Omega$);
			\item \label{Item_EllReg_BdyFamilyExt} There exist an integer $1\leq k\leq n-1$, a constant $r_0'\in (0, \varsigma_{n,k}\,r_0)$, where $\varsigma_{n,k}\in (0,1]$ is determined by Lemma \ref{Lem_App_HSFoliation}, and a smooth monotonic deformation $\{\Gamma_t\}_{t\in [0, +\infty)}$ of hypersurfaces in $\partial \BB_{r_0}\cap \Omega$ such that 
			\begin{align*}
				\Gamma_t & = \spt\bM_t\cap \partial \BB_{r_0}\,, & & \forall\, t\in [0, T]\,; \\
				\Gamma_t & = (\SSp^{n-k}(r_0')\times \R^k)\cap \partial\BB_{r_0}\,, & & \forall\, t\geq T+1\,.
			\end{align*}
			Here monotonic is in the sense that $\Gamma_t$ is the boundary of $(\Omega\cap \partial \BB_{r_0})\setminus \cup_{s\leq t}\ \Gamma_s$, $\forall\, t\geq 0$.
		\end{enumerate}
		Then 
		\begin{enumerate} [label={\normalfont(\roman*)}] %\alph*; \arabic*
			\item \label{Item_EllReg_LSFnonFat} The level set flow (see Definition \ref{def:WSF}, Remark \ref{Rem:LSF} and discussion after it) $\cK$ generated by \[
			(\mbfM_0\cap \BB_{r_0})\times \{0\}\cup \bigcup_{t\in[0,T]}(\mbfM_t\cap \pr \BB_{r_0})\times \{t\}  \]
			is non-fattening and mean convex, with each time slice $\cK(t)$ $n$-rectifiable; 
			\item \label{Item_EllReg_BrakUniq} Inside $\BB_{r_0}$, $\Sigma_t=\|\cK(t)\|$ for all $t\in[0,T]$;
			\item \label{Item_EllReg_noncollap} For every $\eps\in (0, T)$, $\exists\,\alpha>0$ depending on $\eps,\ \mbfM_0$ and $(\Gamma_t)_{t\in[0,+\infty)}$ such that $\cK\times \RR$ is a $C^\infty_{loc}$-limit of a sequence of $\alpha$-noncollapsing classical mean curvature flow in $\BB_{r_0}\times \RR$ over $[\eps, T]$.
		\end{enumerate}
	\end{proposition}
	The technical assumption \ref{Item_EllReg_BdyFamilyExt} is only used to prove \ref{Item_EllReg_noncollap}, and we conjecture that it can be dropped.
	
	It's easy to check that based on Theorem \ref{Thm_Isol_Main} \ref{Enum_MainThm_BdyEvol} (boundary evolution), \ref{Enum_MainThm_t<0} (graphical before singular time) and \ref{Enum_MainThm_t=0} (profile at singular time), this proposition implies the item \ref{Enum_MainThm_MeanConvex} (mean convexity) and \ref{Enum_MainThm_Noncollap} (noncollapsing) of Theorem \ref{Thm_Isol_Main}. 
	\begin{proof}[Proof of Proposition \ref{Prop:MCF with boundary}.]
		By assumption \ref{Item_EllReg_M_0MeanConv} and \ref{Item_EllReg_BdyMonot}, item \ref{Item_EllReg_LSFnonFat} follows from \cite[Theorem 4]{White15_SubseqSing_MeanConvex}, where we take $W=\BB_{r_0}\cap \Omega$, $\Sigma=\mbfM_0\cap \BB_{r_0}$, and $\Sigma'=\pr \BB_{r_0}\cap \Omega$, $\Gamma_t=\mbfM_t\cap \pr \BB_{r_0}$, and because we only care about the behavior of the flow in a finite amount of time, we do not need to consider $\Gamma_\infty$. As a by-product, the singular set of $\cK$ has parabolic Hausdorff dimension $\leq n-1$. 
		
		To prove item \ref{Item_EllReg_BrakUniq}, it remains to prove that $\Sigma_t=\|\cK(t)\|$. This is the consequence of a uniqueness theorem. By \cite{Ilmanen94_EllipReg, HershkovitsWhite23_Avoidance}, the closure of the support of $\Sigma_t$ is a weak set flow, hence it is contained in $\cK(t)$. By White's stratification theorem \cite{White97_Stratif} and White's classification of tangent flows of mean convex mean curvature flows \cite{White00, White03}, together with Chodosh-Choi-Mantoulidis-Schulze's characterization of Brakke flow with small singular sets \cite[Corollary G.5]{CCMS20_GenericMCF}, we know that the regular part $\Reg(\cK)$ is connected in spacetime. Together with the uniqueness of Brakke flow of regular mean curvature flow, e.g. \cite[Appendix C]{CCMS20_GenericMCF}, we know that the regular part of $\Sigma_t$ coincides with the regular part of $\cK(t)$. Finally, since the singular set of $\cK$ is small as proved in \ref{Item_EllReg_LSFnonFat}, we derive $\Sigma_t=\|\cK(t)\|$ for every $t\in [0, T]$. 
		
		Finally, we prove item \ref{Item_EllReg_noncollap} by revisiting the elliptic regularization construction, since the noncollapsing estimates in \cite{ShengWang09_SingMCF, Andrews12_Noncollapsing}, as well as some later works such as \cite{AndrewsLM13, Brendle15_inscribed} seem to rely on the parabolic maximum principle, and therefore the smoothness assumption of the flow. Our argument closely follows the process in White \cite[Section 5]{White15_SubseqSing_MeanConvex} and Haslhofer-Kleiner \cite[Section 4]{HaslhoferKleiner17}. First note that since $\cK(0) = \bM_0\cap \BB_{r_0}$ is smooth and strictly mean convex, by taking $\eps\ll 1$, we may further assume that $\cK(\eps)$ is a smooth strictly mean convex hypersurface with boundary.
		
		In the construction below, for every $\lambda>0$, $N_\lambda$ will be a smooth mean convex minimizer of the functional $\cI^\lambda$ with certain prescribed boundary to be specified, and suppose $V$ is the tangential projection of the time vector field $-\lambda \vec{e}_{n+2}$. Following \cite{AndrewsLM13}, we define
		\begin{align*}
			Z(x,y) = 2\langle x-y,\bn(x)\rangle|x-y|^{-2}, & & 
			Z^*(x) := \sup_{x\neq y\in N_\lambda}Z(x,y), & & 
			Z_*(x) := \inf_{x\neq y\in N_\lambda}Z(x,y).        
		\end{align*}
		Then \cite[Theorem 2]{AndrewsLM13} (see also \cite[(4.7)]{HaslhoferKleiner17}) shows the following inequalities in the viscosity sense, where $H>0$ is the (scalar) mean curvature of $N_\lambda$,
		\begin{align}
			\begin{split}
				\Delta \frac{Z^*}{H}
				+2\left\langle\nabla 
				\log H,\nabla \frac{Z^*}{H}\right\rangle
				-\nabla_V\frac{Z^*}{H} & \geq 0,
				\\
				\Delta \frac{Z_*}{H}
				+2\left\langle\nabla 
				\log H,\nabla \frac{Z_*}{H}\right\rangle
				-\nabla_V\frac{Z_*}{H} & \leq 0.
			\end{split}
		\end{align}
		Thus, $\frac{Z^*}{H}$ attains its maximum over $N_\lambda$ at the boundary $\pr N_\lambda$, and $\frac{Z_*}{H}$ attains its minimum over $N_\lambda$ at the boundary $\pr N_\lambda$. Therefore, it suffices to show that there's a sequence $\lambda_i\to+\infty$ such that, $\frac{Z^*}{H}$ and $\frac{Z_*}{H}$ associated to $N_{\lambda_i}$ above has a uniform two-sided bounded near $\pr N_{\lambda_i}$. %We use $\alpha_i$ to denote the Andrews constant.
		
		Now we give some detailed descriptions of $N_{\lambda}$ and its boundary. Let $\tilde\cK$ be the level set flow generated by \[
		S:= (\mbfM_0\cap \BB_{r_0}) \times \{0\}\cup \bigcup_{t\in[0, +\infty)}\Gamma_t\times \{t\}  \]
		Again by White \cite[Theorem 4]{White15_SubseqSing_MeanConvex}, using the boundary regularity of Brakke flow \cite{White21_MCFBoundary}, $\tilde\cK(t) = \cK(t)$ for $t\leq T$, and as $t\to +\infty$, $\tilde\cK(t)$ converges to some minimal variety $\tilde\cK_\infty$ which is smooth near $\partial \tilde\cK_\infty = \Gamma_{T+1}$, and then by Lemma \ref{Lem_App_HSFoliation}, must be a smooth minimal hypersurface with boundary. Hence there exists $T'\geq T+1$ such that $\tilde\cK(t)$ is smooth and strictly mean convex in $\BB_{r_0}$ for all $t\geq T'-1$.
		
		Now we prescribe the boundary for $N_\lambda$: \[
		S_\lambda := (\mbfM_0\times\{0\}) \cup (\tilde\cK(T')\times\{\lambda T'\}) \cup \bigcup_{t\in[0,{T'}]}(\Gamma_t\times\{\lambda t\})
		\]
		And let $N_\lambda$ be a minimizer of the functional $\cI^\lambda$ among all flat chain mod $2$ with boundary $[S_\lambda]$. Then by White \cite[Theorem 10]{White15_SubseqSing_MeanConvex}, $N_\lambda$ is the flat chain associated to a smooth strictly mean convex hypersurface in $\RR^{n+1}\times [0, \lambda T']$. By the process of Elliptic Regularization introduced above, the Brakke flow $\bN_\lambda: t\mapsto N_\lambda - \lambda t\, \vec{e}_{n+2}$ subconverges to a unit-regular cyclic Brakke flow $t\mapsto \mu_t\times \RR$ over $(0, T']$, where the support of $\{\mu_t\}_{t\in [0, T']}$ is a weak set flow generated by $S\cap \{0\leq t\leq T'\}$. Then by item \ref{Item_EllReg_BrakUniq}, $\mu_t = \|\tilde\cK(t)\cap \BB_{r_0}\|$ for every $t\in [0, T']$, and in particular, $\{\mu_t\}_{t\in [0, T']}$ is smooth and strictly mean convex in a neighborhood $\cU$ of $(S\cap \{0\leq t\leq T'\} )\cup (\cK(T')\times \{T'\})$. Hence in $\cU$, the flow $t\mapsto \tilde\cK(t)$ is $2\alpha$-noncollapsing for some $\al>0$. Then by Brakke-White interior Regularity \cite{White05_MCFReg} and White's boundary regularity \cite{White21_MCFBoundary}, there's a subsequence $\lambda_i\to +\infty$ such that $\bN_{\lambda_i}$ converges locally smoothly in $\cU\times \RR$ to $t\mapsto \tilde\cK(t)\times \RR$ over $(0, T']$, which implies a uniform two-sided bound of $\frac{Z^*}{H}$ and $\frac{Z_*}{H}$ near \[
		S_{\lambda_i}^\eps := (\cK(\eps)\times\{\lambda\eps\}) \cup (\tilde\cK(T')\times\{\lambda T'\}) \cup \bigcup_{t\in[\eps,{T'}]}(\Gamma_t\times\{\lambda t\})  \]
		for $N_{\lambda_i}$. Then by strong maximum principle as discussed above, $|\frac{Z^*}{H}|$ and $|\frac{Z_*}{H}|$ are uniformly bounded by $\frac{1}{\alpha}<+\infty$ on the whole $N_{\lambda_i}\cap \{\lambda_i\eps\leq t\leq \lambda_i T'\}$, which passes to limit as $\lambda_i\to +\infty$ and implies that $\tilde\cK\times \RR$ is a $C^\infty_{loc}$ limit of $\alpha$-noncollapsing classical mean curvature flow in $\BB_{r_0}\times \RR$ over $[\eps, T']$. This finishes the proof of \ref{Item_EllReg_noncollap}.  
	\end{proof}
	
	We remark that Brendle-Naff \cite{BrendleNaff21_LocalNoncollapsing} proved a local noncollapsing estimate for smooth mean convex mean curvature flows. If their proof can be adapted to mean convex mean curvature flow with singularities, one can get an alternative proof of Proposition \ref{Prop:MCF with boundary} \ref{Item_EllReg_noncollap}.

	\appendix
	\section{Analysis of Jacobi field equation on $\cC_{n,k}$.}\label{app_Jacobi}
	\begin{Lem} \label{Lem_App_Analysis of Parab Jacob field}
		Let $a+1< b$, $v\in C^2_{loc}(\cC_{n,k}\times (a,b))$ be a non-zero function so that $\|v(\cdot, \tau)\|_{L^2}<+\infty$ for every $\tau\in (a, b)$ and it solves
		\begin{align*}
			(\partial_\tau - L_{n,k})v = 0
		\end{align*}
		on $\cC_{n,k}\times (a,b)$. Define the \textbf{linear decay order} of $v$ at time $\tau$ by \[
		N_{n,k}(\tau; v) := \log\left( \frac{\|v(\cdot, \tau)\|_{L^2}}{\|v(\cdot, \tau + 1)\|_{L^2}}\right)
		\]
		Then we have,
		\begin{enumerate}[label={\normalfont(\roman*)}] %\arabic*; \alph*
			\item $N_{n,k}(\tau; v)\geq -1$ and is monotone non-increasing in $\tau\in (a, b-1)$.
			\item If for some $\tau_0\in (a, b)$, $\gamma\in \RR$ and $\sim\in \{\geq, >, =, <, \leq\}$, we have $\|\Pi_{\sim \gamma} (v(\cdot, \tau_0))\|_{L^2} = \|v(\cdot, \tau_0)\|_{L^2}$, where $\Pi_{\sim \gamma}$ is defined in \eqref{Equ_Pre_Proj onto sum of eigenspace}. Then for every $\tau\in (a,b-1)$, \[
			N_{n,k}(\tau; v)\sim \gamma\,.
			\] 
			In particular, if
			\[
			\int_{\cC_{n,k}} v(X, \tau_0) e^{-\frac{|X|^2}{4}}\ dX = 0\,.
			\]
			Then $N_{n,k}(\tau; v)\geq -1/2$, $\forall\,\tau\in (a,b-1)$, with equality holds for some $\tau'$ if and only if \[
			v(X, \tau) = e^{\tau/2}\psi(X)
			\]
			for some non-zero eigenfunction $\psi$ of $-L_{n,k}$ with eigenvalue $-1/2$.
			\item If for some $\tau_1<\tau_2\in (a, b-1)$, $N_{n,k}(\tau_1; v) = N_{n,k}(\tau_2; v) = \gamma$, then \[
			v(X, \tau) = e^{-\gamma\tau}\psi(X)
			\] 
			for some non-zero eigenfunction $\psi\in L^2(\cC_{n,k})$ of $-L_{n,k}$ with eigenvalue $\gamma$.
		\end{enumerate}
	\end{Lem}
	\begin{proof}
		(i) By the $L^2$-spectral decomposition, we can write 
		\begin{equation}\label{eq:Fourier}
			v(\cdot,\tau)=\sum_{k=1}^\infty e^{-\lambda_k\tau} \phi_k,
		\end{equation}
		where $\lambda_1<\lambda_2\leq \cdots$ are eigenvalues of $-L_{n,k}$ and $\phi_k$ are corresponding eigenfunctions. By Plancherel identity, we have \[
		V(\tau):= \|v(\cdot,\tau)\|_{L^2}^2=\sum_{k=1}^\infty e^{-2\lambda_k\tau}\|\phi_k\|_{L^2}^2.  \] 
		Thus \[
		N_{n,k}(\tau;v)
		= \frac{1}{2}\log \left(\frac{\sum_{k=1}^\infty e^{-2\lambda_k\tau}\|\phi_k\|_{L_2}^2}{\sum_{k=1}^\infty e^{-2\lambda_k(\tau+1)}\|\phi_k\|_{L_2}^2}\right) 
		\geq \frac{1}{2}\log \left(\frac{\sum_{k=1}^\infty e^{-2\lambda_k\tau}\|\phi_k\|_{L_2}^2}{e^{-2\lambda_1}\sum_{k=1}^\infty e^{-2\lambda_k \tau}\|\phi_k\|_{L_2}^2}\right)=\lambda_1=-1.
		\]
		This shows the lower bound of $N_{n,k}(\tau;v)$. To obtain the monotonicity, we compute that 
		\begin{align*}
			&  V(\tau)^2\cdot \pr_\tau(\log V(\tau)) = V(\tau)V''(\tau) - V'(\tau)^2 \\
			=\; & \left( \sum_{k=1}^\infty (-2\lambda_k)^2 e^{-2\lambda_k\tau}\|\phi_k\|_{L^2}^2 \right)\left( \sum_{k=1}^\infty  e^{-2\lambda_k\tau}\|\phi_k\|_{L^2}^2 \right)  - \left( \sum_{k=1}^\infty (-2\lambda_k) e^{-2\lambda_k\tau}\|\phi_k\|_{L^2}^2 \right)^2 \\
			\geq\; & 0\,.
		\end{align*}
		Therefore, $\pr_\tau(\log \|v(\cdot,\tau)\|_{L^2}^2)$ is a non-decreasing function in $\tau$, and hence $\pr_\tau N_{n,k}(\tau;v)\leq 0$. So $N_{n,k}(\tau;v)$ is monotone non-increasing.
		
		The first part of (ii) is a direct consequence of the spectral decomposition \eqref{eq:Fourier} and Plancherel identity. If $\int_{\cC_{n,k}} v(X, \tau_0) e^{-\frac{|X|^2}{4}}\ dX = 0$, by the classification of eigenfunctions of $-L_{n,k}$ with small eigenvalues in Section \ref{SS:PreCylinder}, we have $\|\Pi_{\geq  -1/2} (v(\cdot, \tau_0))\|_{L^2} = \|v(\cdot, \tau_0)\|_{L^2}$, and hence $N_{n,k}(\tau;v)\geq -1/2$. In this case, $v(\cdot,\tau)=\sum_{k=2}^\infty e^{-\lambda_k\tau} \phi_k$, and  \[
		N_{n,k}(\tau;v) 
		= \frac{1}{2}\log \left(\frac{\sum_{k=2}^\infty e^{-2\lambda_k\tau}\|\phi_k\|_{L^2}^2}{\sum_{k=2}^\infty e^{-2\lambda_k(\tau+1)}\|\phi_k\|_{L^2}^2}\right)
		\geq \frac{1}{2}\log \left(\frac{\sum_{k=2}^\infty e^{-2\lambda_k\tau}\|\phi_k\|_{L^2}^2}{e^{-2\lambda_2}\sum_{k=2}^\infty e^{-2\lambda_k \tau}\|\phi_k\|_{L^2}^2}\right)=\lambda_2=-1/2.
		\]
		Therefore, the equality holds if and only if all the nonzero terms in \eqref{eq:Fourier} have to have eigenvalue $-1/2$, which is equivalent to $v(X,\tau) = e^{\tau/2}\psi(X)$ for some eigenfunction $\psi$ with eigenvalue $-1/2$.
		
		(iii) From the proof of (i), we know that $N_{n,k}(\tau_1; v) = N_{n,k}(\tau_2; v)$ if for any $\tau\in(\tau_1,\tau_2)$, $\lambda(\tau) v= L_{n,k}v$ for some constant $\lambda(\tau)\in\R$. This implies that $v(\cdot,\tau)$ is an eigenfunction of $L_{n,k}$, and by the spectral decomposition, and $N_{n,k}(\tau_1; v) = N_{n,k}(\tau_2; v)=\gamma$, we have $\lambda(\tau)=\gamma$, and hence $\gamma\in \sigma(\cC_{n,k})$ (defined in \eqref{Equ_Pre_sigma(C_n,k)}) and $v(X, \tau) = e^{-\gamma\tau}\psi(X)$ for some non-zero eigenfunction $\psi\in L^2(\cC_{n,k})$ of $-L_{n,k}$ with eigenvalue $\gamma$.
	\end{proof}
	
	\section{Graph over a round cylinder} \label{Append_Graph over Cylind}
	Throughout this section, we parametrize $\cC_{n,k}=\SSp^{n-k}(\sqrt{2(n-k)})\times \RR^k \subset \RR^{n+1}$ by $(\theta, y)$ as before. Note that for every $(\theta, y)\in \cC_{n,k}$, $|\theta|^2=2(n-k)$. Let $\hat\theta := \theta/|\theta|$.
	
	Let $\Omega = \SSp^{n-k}(\sqrt{2(n-k)})\times \Omega_\circ\subset \cC_{n,k}$ be a subdomain, $u\in C^1(\Omega)$ such that $\inf u>-\sqrt{2(n-k)}$. We use $\nabla_\theta u$ and $\nabla_y u$ to denote the components of $\nabla u$ parallel to $\SSp^{n-k}(\sqrt{2(n-k)})$ factor and $\RR^k$ factor correspondingly. For later reference, we also denote 
	\begin{align*}
		\hat\nabla_\theta u := \left(1+\frac{u}{|\theta|}\right)^{-1}\cdot\nabla_\theta u\,, & &
		\hat\nabla u := \hat\nabla_\theta u + \nabla_y u \,.
	\end{align*}
	
	\begin{Lem} \label{Lem_App_Graph over Cylinder}
		Let $\Omega, u$ be specified as above, let \[
		\Sigma = \graph_{\cC_{n,k}}(u) = \{ ( \theta + u(\theta, y)\hat\theta, y): (\theta, y)\in \Omega \} \,.
		\]
		And we parametrized $\Sigma$ by \[
		\Phi_u: \Omega \to \Sigma\,, \quad (\theta, y)\mapsto ( \theta + u(\theta, y)\hat\theta, y) \,.
		\]
		Then $\Sigma$ is a hypersurface in $\RR^{n+1}$, $\Phi_u$ is a diffeomorphism and we have the following.
		\begin{enumerate} [label={\normalfont(\roman*)}]
			\item For every $(\theta, y)\in \Omega$, \[
			\odist_{n,k}(\Phi_u(\theta, y)) = \chi(u(\theta, y)),   \]
			\item The unit normal field of $\Sigma$ pointing away from $\{\orig\}\times \RR^k$ is \[
			\nu_\Sigma|_{\Phi_u(\theta, y)} = \left.(1+|\hat\nabla u|^2)^{-1/2}\cdot \left( \hat\theta - \hat\nabla u\right) \right|_{(\theta, y)} \,,
			\] 
			\item For every function $f\in C^0(\Sigma)$, \[
			\int_\Sigma f(x)\ dx = \int_\Omega f\circ\Phi_u(X) \cA[u]\ dX\,,
			\]
			where \[
			\cA[u](\theta, y) := \left(1+\frac{u}{|\theta|}\right)^{n-k}\cdot\sqrt{1+|\hat\nabla u|^2} \,.
			\]
			In particular, there exists $\kappa_n\in (0, 1/2),\ \bar c_{n,k} >0$ such that if $\|u\|_{C^1}\leq \kappa_n$, then \[
			\left|\mbfd_{n,k}(\Sigma)^{-1}\cdot \|u\|_{L^2} - \bar c_{n,k}\right| \leq C_n\|u\|_{C^1} \,.
			\]
			\item There exists $\kappa_n'\in (0, 1/2)$ such that if $(\bx, \by)\in \RR^{n-k+1}\times \RR^k$, $\lambda>0$ satisfy, \[
			\|u\|_{C^1(\Omega)} + |\bx| + |\lambda - 1| \leq \kappa_n' \,,
			\]
			then $\lambda\Sigma - (\bx, \by)$ is also a graph over some subdomain $\bar\Omega\subset \cC_{n,k}$, and the graphical function $\bar u$ satisfies for every $(\theta', y')\in \Omega'$, 
			\begin{align*}
				\Big|\bar u(\theta', y') + \bx\cdot \hat\theta' - \sqrt{2(n-k)}(\lambda - 1) & - \lambda u(\theta', \lambda^{-1}(y'+\by)) \Big|  \\
				& \leq C_n \left(\|\nabla_\theta u(\cdot, \lambda^{-1}(y'+\by))\|_{C^0}  + |\bx|\right)\cdot |\bx|\,.  
			\end{align*}
		\end{enumerate}
	\end{Lem}
	\begin{proof}
		(i)-(iii) follows by standard calculations, see also \cite[Appendix A]{CM15_Lojasiewicz}. We now prove (iv): Given $(\theta,y)\in \Omega$, it corresponds to a point $(\theta+u(\theta,y)\hat\theta,\, y)$ in the graph, and after the dilation and translation, the point becomes $\lambda(\theta+u(\theta,y)\hat\theta,\, y)-(\bx,\by)$. Suppose its nearest projection to $\cC_{n,k}$ is the point $(\theta',y')$, then we can write \[
		\left(\theta'+\bar u(\theta',y')\hat\theta',\, y'\right) = \lambda\left(\theta + u(\theta, y)\hat\theta,\, y \right)-(\bx,\by) .
		\]
		Taking the projection to the spine and the orthogonal complement to the spine gives \[
		y = \lambda^{-1} (y'+\by), \qquad 
		\theta'+\bar u(\theta',y')\hat\theta' = \lambda \left(\theta+ u(\theta,y)\hat\theta \right)-\bx,
		\]
		projecting to $\SSp^{n-k}$, the latter implies that \[
		|\theta'-\theta| \leq C_n|\bx|\,.
		\]
		Taking the inner product with $\hat\theta'$ gives
		\begin{align*}
			\bar u(\theta',y')+\bx\cdot\hat\theta' & - \sqrt{2(n-k)}(\lambda-1) - \lambda u(\theta',\lambda^{-1}(y'+\by)) \\ 
			& = \lambda \left[(\theta - \theta')\cdot \hat \theta' + u(\theta,y)\hat\theta\cdot \hat\theta' - u(\theta',y) \right] \\ 
			& = \lambda \left(\sqrt{2(n-k)} + u(\theta,y)\right)(\hat\theta - \hat\theta')\cdot \hat\theta' + \lambda \left(u(\theta, y) - u(\theta',y)\right) \,. 
		\end{align*}
		By mean value theorem, \[
		|u(\theta,y)-u(\theta',y)|\leq C_n\|\nabla_\theta u(\cdot,y)\|_{C^0}|\bx|\,; \]
		Since $\hat\theta,\ \hat\theta'$ are unit vectors, we also have \[
		|(\hat\theta- \hat\theta')\cdot\hat\theta'| = \frac12 |\hat\theta - \hat\theta'|^2 \leq C_n|\bx|^2\,.
		\]
		Combining them proves (iv).
	\end{proof}
	
	\begin{Lem} \label{Lem_App_RMCF equ}
		There exists $\kappa_n''\in (0, 1/2)$ with the following significance. Let $\Omega\subset \cC_{n,k}$ be as above, $u\in C^2(\Omega\times (a, b])$ with $\|u\|_{C^2}\leq \kappa_n''$ such that $\tau\mapsto \graph_{\cC_{n,k}}(u(\cdot, \tau))$ is a rescaled mean curvature flow in $\RR^{n+1}$. Then $u$ satisfies the following nonlinear parabolic equation, 
		\begin{align}
			\partial_\tau u - L_{n,k} u = \cQ(u, \nabla u, \nabla^2 u)\,; \label{Equ_App_RMCF equ}
		\end{align}
		where $\cQ$ is a smooth function in $(z, \xi, \eta)\in \RR\times T\cC_{n,k}\times T^{\otimes 2}\cC_{n,k}$ with \[
		\cQ(0, 0, 0) = \cQ_z(0, 0, 0) = \cQ_\xi(0,0,0) = \cQ_\eta(0,0,0) = 0\,.
		\]
	\end{Lem}
	\begin{proof}
		See \cite[Appendix A]{CM15_Lojasiewicz}.
	\end{proof}
	
	\section{A regularity lemma of some minimal hypersurface}
	
	\begin{Lem}\label{Lem_App_HSFoliation}
		For every $1\leq k\leq n-1$, there exists some $\varsigma=\varsigma(n,k)\in (0, 1)$ such that for every $r'\in (0, \varsigma)$, there's a smooth minimal hypersurface $\Sigma(r')$ with boundary \[
		\Gamma(r'):= \SSp^n\cap (\SSp^{n-k}(r')\times \RR^k)\,,
		\]  
		such that if $M$ is a minimal hypersurface in $\overline{\BB_1^{n+1}}$ (possibly with singularities $\Subset \BB^{n+1}_1$) with boundary $\Gamma(r')$, then $M=\Sigma(r')$.
	\end{Lem}

	\begin{proof}
		When $k=1$, consider a catenoid $S$, which is an $SO(n)\times \ZZ_2$-invariant smooth embedded minimal hypersurfaces in $\RR^n\times \RR$. Under dilations. The dilations of $S$ sweep out a solid cone region $\rmC$ in $\R^{n+1}$: \[
		\bigcup_{\lambda>0} \lambda\cdot S = \rmC := \{(x, y)\in \RR^{n}\times \RR: |x|\geq \al_n|y|\}\setminus \{\orig\}
		\]
		for some $\al_n>0$. We take $\varsigma(n,1) \in (0, 1)$ by setting $\al_n^2=\varsigma^2/(1-\varsigma^2)$, which implies for every $r'\in (0, \varsigma)$, $\Gamma(r')\cap \rmC = \emptyset$. Therefore, we can slightly translate $\mbfC$ up and down a little bit to obtain a region $U$ so that $\overline{U}\cap \Gamma(r') = \emptyset$, $\Int(\overline{U})$ is connected, and $\overline{\BB_1^{n+1}}\backslash \overline{U}$ has two connected components.  Using the maximum principle for minimal hypersurfaces possibly with singularities, if $M$ is a minimal hypersurface in $\overline{\BB_1^{n+1}}$ with boundary $\Gamma(r')$, then $M$ is disjoint from $\overline{U}$. In particular, $M$ can be decomposed into two components $M_\pm$, where $M_+$ has boundary $\Gamma(r')\cap(\R^{n}\times\{s\}\}_{s>0})$ and $M_-$ has boundary $\Gamma(r')\cap(\R^{n}\times\{s\}\}_{s<0})$. Then applying the maximum principle to each one of them, with the foliation $\{\BB_1^{n+1}\cap(\R^{n}\times\{s\})\}_{s\in\R}$ of minimal hypersurfaces, we know that $M=\Sigma(r')$ where $\Sigma(r')$ is the union of two disconnected flat disks $\BB^n_{r'}\times\{\pm\sqrt{1-(r')^2}\}$. 
		
		When $2\leq k\leq n-1$, we can take $\varsigma = \sqrt\frac{n-k}{n-1}$ if one of the following holds: 
		\begin{itemize}
			\item $n+1\geq 9$, 
			\item $n+1 = 8$ and $k\in \{3, 4, 5\}$.
		\end{itemize}
		To see this, we recall the famous Hardt-Simon foliation of minimal hypersurfaces \cite{HardtSimon85}: given a minimizing hypercone $\mbfC\subset \RR^{n+1}$ with isolated singularity $\orig$, there exists a unique foliation by minimal hypersurfaces $(\Sigma_\lambda)_{\lambda\in\R}$ of $\R^{n+1}$ such that $\Sigma_0 = \mbfC$ and for any $a>0$, $\Sigma_{\pm a}$ is a dilation of $\Sigma_{\pm 1}$, which are smooth minimizing hypersurfaces. In particular, when \[
		\mbfC = \mbfC^{n-k, k-1}:= \{(x,y)\in \RR^{n-k+1}\times \RR^k: (k-1)|x|^2 = (n-k)|y|^2\}
		\]
		and $2\leq k\leq n-1$ has the constraint above, by \cite{Simoes74_MinCone}, $\mbfC$ is an $SO(n-k+1)\times SO(k)$-invariant minimizing hypercone, and hence the foliation described above is also $SO(n-k+1)\times SO(k)$-invariant, which implies $\Gamma(r')$ is the intersection of $\SSp^n$ with one leaf $\Sigma_{\lambda'\neq 0}$. By the strong maximum principle for minimal hypersurfaces with possibly singularities \cite{white10_Maxim}, the only minimal hypersurface in $\overline{\BB^{n+1}_1}$ with boundary $\Gamma(r')$ must be $\Sigma(r'):= \Sigma_{\lambda'}\cap \overline{\BB_1}$. %So the Lemma is proven for these cases.
		
		Finally, we consider the remaining cases that $k\geq 2$, $n+1\leq 7$ or $n+1=8$ and $k\in\{2,6\}$. In all these cases, $\mbfC^{n-k, k-1}$ is minimal but not minimizing, and then there is no global Hardt-Simon foliation by minimal hypersurfaces as above. However, a Hardt-Simon type foliation still exists within a subdomain.   
		In fact, from \cite{AlencarBarrosPalmasReyesSantos05_OmOnMin}, there exists a smooth embedded complete minimal hypersurface $\Sigma$ that is $SO(n-k+1)\times SO(k)$-invariant, asymptotic to $\mbfC^{n-k,k-1}$, and intersects the subspace $\{0\}\times\R^{k}$ orthogonally. As a consequence, there's an $SO(n-k+1)\times SO(k)$-invariant tubular neighborhood of $\Sigma\cap\{0\}\times\R^{k}$ within $\Sigma$, denoted by $\Sigma_1$, which is still a radial graph. By a rescaling, we may assume that $\partial \Sigma_1\subset \SSp^n$. Then, $\{\Sigma_\lambda := \lambda\cdot \Sigma_1\}_{\lambda\geq 1}$ is a foliation by minimal hypersurfaces with boundary $\subset \RR^{n+1}\setminus\BB_1$ of some closed domain $E$. %Moreover, for any $r'>0$ such that $\Gamma(r')\subset E$, there exist $\lambda(r')\geq 1$ so that $\Gamma(r')\subset \Sigma_{\lambda(r')}$. 
		
		We claim that there exists $\varsigma(n,k)>0$ such that when $r'\in(0,\varsigma)$, $\Gamma(r')\subset E$, and if $M$ is a minimal hypersurface with possibly singularities whose boundary is $\Gamma(r')$, then $M\subset E$. We prove the claim by contradiction. Suppose the claim is false, then there exists $r_i\to 0$ such that $\Gamma(r_i)$ bounds a minimal hypersurface $M_i$ possibly with singularities, such that $M_i\backslash E \neq \emptyset$. Suppose $p_i \in M_i\backslash E$. Then $2d:=\liminf_{i\to\infty}\dist(p_i,\Gamma(r_i))>0$, and by monotonicity formula of minimal hypersurfaces, this implies that $\vol(M_i)\geq C_nd^n$ for $i\gg 1$. On the other hand, $\lim_{i\to\infty}\area(\Gamma(r_i))= 0$. This contradicts the isoperimetric inequality of minimal hypersurfaces \cite{Almgren86_Isop}.
		
		With this claim, when $r'<\varsigma(n,k)$, if $M$ is a minimal hypersurface with possibly singularities whose boundary is $\Gamma(r')$, then $M\subset E$. Since $E$ is foliated by minimal hypersurfaces $\{\Sigma_\lambda\}_{\lambda\geq 1}$ and $\Gamma(r')\subset \Sigma_{\lambda(r')}$ for some $\lambda(r')>0$, the strong maximum principle of minimal hypersurfaces shows that $M$ must be $\BB_1\cap \Sigma_{\lambda(r')}$. This completes the proof of the Lemma for the remaining cases.
	\end{proof}

	\bigskip

	\bibliographystyle{alpha}
	\bibliography{GMT}
\end{document}